\documentclass[11pt]{article}

\usepackage{amsmath,amsthm,amsfonts,amssymb,amsbsy}
\usepackage[numbers,sort&compress]{natbib}
\usepackage[colorlinks,citecolor=MidnightBlue,linkcolor=RedOrange,urlcolor=RedOrange]{hyperref}
\usepackage[dvipsnames]{xcolor}
\usepackage{graphicx}
\usepackage{bbm,comment,enumerate,longtable}
\usepackage{csquotes}
\usepackage{aliascnt}
\usepackage[noabbrev,capitalize]{cleveref}
\usepackage{authblk}
\usepackage[top=0.72in,bottom=0.72in,left=0.82in,right=0.82in]{geometry}

\usepackage{microtype}
\usepackage{titlesec}
\titlespacing*{\section}{0pt}{1.5ex plus 0.5ex minus 0.3ex}{0.8ex plus 0.2ex}
\titlespacing*{\subsection}{0pt}{1.2ex plus 0.4ex minus 0.2ex}{0.6ex plus 0.2ex}
\titlespacing*{\subsubsection}{0pt}{1.0ex plus 0.3ex minus 0.2ex}{0.5ex plus 0.1ex}

\makeatletter
\def\thm@space@setup{%
  \thm@preskip=5pt plus 2pt minus 2pt
  \thm@postskip=5pt plus 2pt minus 2pt
}
\makeatother

\AtBeginDocument{%
  \abovedisplayskip=7pt plus 2pt minus 3pt
  \belowdisplayskip=7pt plus 2pt minus 3pt
  \abovedisplayshortskip=3pt plus 2pt minus 1pt
  \belowdisplayshortskip=3pt plus 2pt minus 1pt
}

\usepackage{enumitem}
\setlist{topsep=2pt,itemsep=1pt,parsep=1pt}

\newcommand{\mv}{\mathfrak{v}}

\newcommand{\vep}{\varepsilon}

\newcommand{\mr}{\mathcal{R}}

\newcommand{\tm}{t_{\theta,B,\mu}}

\newcommand{\E}{\mathbb{E}}
\newcommand{\tml}{\widetilde{\mathfrak{L}}_n(\mathbf{X})}
\newcommand{\tmln}{\mathfrak{L}_n(\mathbf{m}_R)}

\newtheorem{thm}{Theorem}
\numberwithin{thm}{section}

\newaliascnt{lmm}{thm}
\newtheorem{lmm}[lmm]{Lemma}
\aliascntresetthe{lmm}
\crefname{lmm}{Lemma}{Lemmas}
\Crefname{lmm}{Lemma}{Lemmas}

\newaliascnt{cor}{thm}
\newtheorem{cor}[cor]{Corollary}
\aliascntresetthe{cor}
\crefname{cor}{Corollary}{Corollaries}
\Crefname{cor}{Corollary}{Corollaries}

\newaliascnt{prop}{thm}
\newtheorem{prop}[prop]{Proposition}
\aliascntresetthe{prop}
\crefname{prop}{Proposition}{Propositions}
\Crefname{prop}{Proposition}{Propositions}

\newaliascnt{problem}{thm}

\aliascntresetthe{problem}

\newtheorem{defn}{Definition}

\theoremstyle{definition}
\newtheorem{remark}{Remark}
\newtheorem{ex}{Example}

\numberwithin{remark}{section}
\numberwithin{ex}{section}
\numberwithin{assm}{section}
\numberwithin{defn}{section}

\newcommand{\argmin}{\operatorname{argmin}}

\newcommand{\R}{\mathbb{R}}

\newcommand{\isZ}{Z^{\mathrm{quad}}}

\newcommand{\isR}{\mathbb{R}^{\mathrm{quad}}_{n,\theta,B}}
\newcommand{\isRR}{\mathbb{R}^{\mathrm{quad}}_{n,\theta}}

\newcommand{\tfm}{\mathfrak{L}_n(\mathbf{m})}

\newcommand{\mf}{\mathcal{F}}

\newcommand{\ml}{\mathcal{L}}

\newcommand{\wmm}{\widetilde{\mathcal{M}}}

\newcommand{\tmm}{t_{\tilde{\theta},\tilde{B},\mu_0}}

\newcommand{\fmm}{\mathfrak{m}}

\renewcommand{\P}{\mathbb{P}}

\newcommand{\tph}{\tilde{\mathfrak{u}}}
\newcommand{\tth}{\tilde{\theta}}
\newcommand{\tB}{\tilde{B}}
\newcommand{\tq}{\tilde{q}}
\newcommand{\tp}{p'}
\newcommand{\tQh}{\mathrm{Sym}[Q_n]}

\newcommand{\cN}{\mathcal{N}}

\numberwithin{equation}{section}

\newcommand{\mJ}{\mathcal{J}}

\newcommand{\tlns}{\widetilde{\mathfrak{L}}_n(\mathbf{m}_R)}
\newcommand{\tlnl}{\widetilde{\mathfrak{L}}_{n}(\mathbf{m}_V^{(L)})}
\newcommand{\tln}{\widetilde{\mathfrak{L}}_{n}(\mathbf{X}^{(L)})}
\allowdisplaybreaks
\title{Gibbs measures with multilinear forms}
\author[1]{Sohom Bhattacharya}
\author[2]{Nabarun Deb}
\author[3]{Sumit Mukherjee}
\affil[1]{Department of Statistics, University of Florida, \texttt{bhattacharya.s@ufl.edu}}
\affil[2]{Econometrics and Statistics, University of Chicago Booth School of Business, \texttt{nabarun.deb@chicagobooth.edu}}
\affil[3]{Department of Statistics, Columbia University, \texttt{sm3949@columbia.edu}}
\date{}

\begin{document}

\maketitle

\begin{abstract}
In this paper, we study a class of multilinear Gibbs measures with Hamiltonian given by a generalized $\mathrm{U}$-statistic and with a general base measure.  Expressing the asymptotic free energy as an optimization problem over a space of functions, we obtain sufficient conditions for replica-symmetry, and provide examples to show why these conditions are also necessary.
Utilizing this, we obtain weak limits for a large class of statistics of interest, which includes the \enquote{local fields/magnetization}, the Hamiltonian, the global magnetization, etc. An interesting consequence is a universal weak law for contrasts under replica symmetry, namely, $n^{-1}\sum_{i=1}^n c_i X_i\to 0$ weakly, if $\sum_{i=1}^n c_i=o(n)$.
Our results yield a probabilistic interpretation for the optimizers arising out of the limiting free energy. We also prove the existence of a sharp phase transition point in terms of the temperature parameter, thereby generalizing existing results that were only known for quadratic Hamiltonians. As a by-product of our proof technique, we obtain exponential concentration bounds on local and global magnetizations, which are of independent interest.
\end{abstract}

\noindent\textit{MSC2020 subject classifications.} Primary 82B20; secondary 05C80.

\section{Introduction}
\noindent Suppose $\mu$ is a (non-degenerate) probability measure on $\mathcal{\R}$. Let $H=(V(H),E(H))$ be a finite graph with $v:=|V(H)|\ge 2$ vertices labeled $[v]=\{1,2,\ldots,v\}$, and maximum degree $\Delta$. For $\theta\in \R$, positive integer $n $, and a symmetric $n\times n$ matrix $Q_n$ with $0$ on the diagonal, define the \textit{log-partition function/free energy} 
\begin{align}\label{eq:zn}
Z_n(\theta):=\frac{1}{n}\log \E_{\mu^{\otimes n}} e^{n\theta \mathbb{U}_n({\bf X})}\in (-\infty,\infty],
\end{align}
where the Hamiltonian $\mathbb{U}_n({\bf X})$ is a multilinear form, defined by
\begin{align}\label{eq:U}
	\mathbb{U}_n({\bf X}):=\frac{1}{n^v}\sum_{(i_1,\ldots,i_v)\in \mathcal{S}(n,v) }\Big(\prod_{a=1}^v X_{i_a}\Big)\prod_{(a,b)\in E(H)}Q_n(i_a,i_b).
\end{align}
Here $\mathcal{S}(n,v)$ is the set of all distinct tuples from $[n]^v$ (so that $|\mathcal{S}(n,v)|=v!\binom{n}{v}$). If $\theta$ is such that $Z_n(\theta)$ is finite, we can define a Gibbs probability measure $\mathfrak{R}_{n,\theta}$ on $\R^n$ by setting
	\begin{align}\label{eq:gibbs}
		\frac{d\mathfrak{R}_{n,\theta}}{d\mu^{\otimes n}}({\bf x})=\exp\Big(n\theta \mathbb{U}_n({\bf x})-nZ_n(\theta)\Big).
	\end{align}
%In the context of the above model, the quantity
%$$\lim_{n\to\infty} Z_n(\theta)$$
%is referred to as the (asymptotic) free-energy in the statistical physics literature. 
%Model~\eqref{eq:gibbs} was introduced in its present generality in our paper~\cite{bhattacharya2024ldp}.  
Several Gibbs measures of interest can be expressed in the form \eqref{eq:gibbs} with various choices of $(Q_n,H,\mu)$. Below we give two examples of such Gibbs measures which have been well studied in Probability and Statistics.
%Reasons to study inhomogeneous U-statistics? The homogeneous case is already studied. How should we connect to Gibbs measure?

\begin{itemize}
	\item
	If $H=K_2$ is an edge, then 
	$$\frac{d\mathfrak{R}_{n,\theta}}{d\mu^{\otimes n}}({\bf x})=\exp\left(\frac{1}{n}\sum_{i\ne j} Q_n(i,j) x_i x_j-n Z_n(\theta)\right)$$ is a Gibbs measure with a quadratic Hamiltonian. In particular if $\mu$ is supported on $\{-1,1\}$, then $\mathfrak{R}_{n,\theta}$ is the celebrated Ising model on $\{-1,1\}^n$ with coupling matrix $Q_n$ (see \cite{ising1925beitrag,Basak2017,Bresler2019,bhattacharya2025nonsense,Rados2019,lee2025fluctuations} for various examples). Popular examples of $Q_n$ include the adjacency matrix of the complete graph \cite{Ellis1978}, line graph \cite{ising1925beitrag}, random graphs such as Erd\H{o}s-R\'{e}nyi \cite{kabluchko2019fluctuations} or random $d$-regular \cite{Xu2023}, and so on.
	\\
 
	\item
	If $H=K_v$ is the complete graph on $v$ vertices, and $Q_n$ is the adjacency matrix of a complete graph, then model \eqref{eq:gibbs} reduces to
	$$\frac{d\mathfrak{R}_{n,\theta}}{d\mu^{\otimes n}}({\bf x})=\exp\left(\frac{1}{n^{v-1}}\sum_{(i_1,\ldots ,i_v)\in \mathcal{S}_{n,v}} \prod_{a=1}^v x_{i_a}-nZ_n(\theta)\right).$$ For the special case where $\mu$ is supported on $\{-1,1\}$, $\mathfrak{R}_{n,\theta}$ is just the $v$-spin version of the Curie-Weiss model, which has attracted attention in recent years (see~\cite{mukherjee2020estimation,mukherjee2021fluctuations,yamashiro2019dynamics,daskalakis2020logistic}). 
\end{itemize}
\textbf{Our Contributions:}
In this paper, we study the generalized model~\eqref{eq:gibbs}, when the sequence of matrices $\{Q_n\}_{n\ge 1}$ converge in cut norm (see \eqref{eq:cut_con}). Our main contributions can be summarized as follow: 

(a) We show the limit of the log-partition function can be written using an infinite-dimensional optimization problem (\cref{prop:freen}).
In \cref{prop:propopt}, we obtain sufficient conditions for replica-symmetry, and provide examples to show why these conditions are also necessary.

(b) We obtain weak limits for a large family of statistics which include the Hamiltonian, local and global magnetizations, and contrasts (see \cref{lem:wealimi} and \cref{prop:higherord}).

(c) We provide exponential tail bounds for local and global magnetizations (see \cref{lem:basicprop}) which is of possible independent interest.

(d) As our final example, we show the existence of a ``phase transition" for Gibbs measures with higher order interactions of the form \eqref{eq:gibbs} with compactly supported $\mu$ (see \cref{prop:genopt}). 

\subsection{Main results}\label{sec:mainres}

To state our results, we will need some notations which we introduce below, and use throughout the rest of the paper. 
\\

\noindent Our main assumption on the sequence of matrices $\{Q_n\}_{n\ge 1}$ is that it converges in the cut norm (defined below). Cut norm has been introduced in the combinatorics literature to study limits of graphs and matrices (see~\cite{FriezeKannan1999}), and have received significant attention in the recent literature (\cite{bc_lpi,borgs2018p,borgsdense1,borgsdense2}). For more details on cut norm and its manifold applications, we refer the interested reader to \cite{Lovasz2012}. Below we formally introduce the notion of cut norm used in this paper.

\begin{defn}[Cut norm]\label{def:defirst}
	
   For any $\kappa\ge 1$, let $L^{\kappa}([0,1]^2)$ denote the space of all measurable functions $W$ on the unit square, such that $$\|W\|_\kappa:=\Big(\int_{[0,1]^2}\|W(x,y)\|^\kappa\Big)^{1/\kappa}<\infty.$$
Let $\mathcal{W}$ be the space of all symmetric real-valued functions in $L^1([0,1]^2)$. Given two functions $W_1,W_2\in \mathcal{W}$, define the cut norm between $W_1, W_2$ by setting
	$$d_\square(W_1,W_2):=\sup_{S,T}\Big|\int_{S\times T} \Big[W_1(x,y)-W_2(x,y)\Big]dx dy\Big|.$$
	In the above display, the supremum is taken over all measurable subsets $S,T$ of $[0,1]$.  
	\\
    
	\noindent Given a symmetric matrix $Q_n$, define a function $W_{Q_n}\in \mathcal{W}$ by setting
	\begin{align*}
		W_{Q_n}(x,y)=& Q_n(i,j)\text{ if }\lceil nx\rceil =i, \lceil ny\rceil =j.
	\end{align*}
	%With this definition, note that 
	%$$\int_{[0,1]^2}W_{Q_n}(x,y)dx dy=\frac{1}{n^2}\sum_{i,j=1}^nQ_n(i,j).$$

 %\begin{assm}
	We will assume throughout the paper that 
	the sequence of matrices $\{Q_n\}_{n\ge 1}$ introduced in \eqref{eq:U} converge in cut norm,~i.e. for some $W\in \mathcal{W}$,
	\begin{align}\label{eq:cut_con}
		d_{\square}(W_{Q_n},W) \rightarrow 0.
	\end{align} 
%\end{assm}
\end{defn}

\begin{defn}\label{def:M}
Let $\mathcal{M}$ denote the set of probability measures on $[0,1]\times \R$, equipped with weak topology. Given a probability measure $\nu\in\mathcal{M}$, let $\nu_{(1)}$ and  $\nu_{(2)}$ denote its first and second marginals respectively. Define $\wmm\subseteq \mathcal{M}$ to be the set of probability measures on $[0,1]\times \R$ with first marginal uniform, i.e. $\wmm:=\{\nu\in\mathcal{M}:\ \nu_{(1)}=\mathrm{Unif}[0,1]\}.$ 
For any $\nu \in \mathcal{M}$, set $\fmm_p(\nu):=\int |x|^p \,d\nu_{(2)}(x)$ for $p\ge 0$.
 Define $\wmm_p\subseteq \wmm$ to be the set of all probability measures on $[0,1]\times \R$ with first marginal uniform, and second marginal having a finite $p^{th}$ moment, i.e.
    $\wmm_p:=\{\nu\in\wmm:\ \fmm_p(\nu)<\infty\}$.

%Note that $\wmm_p$ is a closed subset of $\wmm$ (by Fatou's Lemma), and $\wmm$ is a closed subset of $\mathcal{M}$, in the weak topology.  
\end{defn}

%\begin{defn}
%    For two measures $\nu_1,\nu_2$ on $[0,1]\times \R$, define  the bounded-Lipschitz metric between $\nu_1$ and $\nu_2$ by
%	\[
%	d_{\ell}(\nu_1,\nu_2):= \sup_{f \in \text{Lip}(1),\|f\|_\infty\le 1}\Bigg|\int f d\nu_1-\int f d\nu_2\Bigg|,
%	\]
%	where the supremum is over the set of functions $f:[0,1]\times \R\mapsto [-1,1]$ which are $1$-Lipschitz. For any set of measures $\mathcal{A}$, define $d_{\ell}(\nu_1,A)= \inf_{\nu_2 \in \mathcal{A}} d_{\ell}(\nu_1,\nu_2)$.
%\end{defn}

We now introduce the exponential tilt of the base measure $\mu$, and some related notations. This requires the following assumption, which we make throughout the paper:
 For all $\lambda>0$ and some $p\in [v,\infty]$, we have
	\begin{align}\label{eq:tailp}
		\E_{\mu} e^{\lambda |X_1|^p}<\infty,
	\end{align}
	where the case $p=\infty$ corresponds to assuming $\mu$ is compactly supported.

\begin{defn}[Exponential tilting]\label{def:tilt}
Given
\eqref{eq:tailp}, 
%remove assn
 %\begin{assm}
 %for all $\lambda>0$ we have
%	\begin{align}\label{eq:tailp}
%		\E_{\mu} e^{\lambda |X_1|^p}<\infty,
%	\end{align}
%	where the case $p=\infty$ corresponds to %assuming $\mu$ is compactly supported.
%\end{assm}
%Suppose $\mu$ is a non-degenerate probability measure on $\R$, such that 
%Then 
the function $$\alpha(\theta):=\log \int_{\R} e^{\theta x}d\mu(x)$$ is finite for all $\theta\in \R$. Define the $\theta$-exponential tilt of $\mu$ by setting $$\frac{ d\mu_{\theta}}{d \mu}(x):= \exp( \theta x - \alpha(\theta)).$$  Then the function $\alpha(.)$ is infinitely differentiable, with $$\alpha'(\theta)=\mathbb{E}_{\mu_\theta}(X),\quad \alpha''(\theta)=\mathrm{Var}_{\mu_\theta}(X)>0.$$  Consequently the function $\alpha'(.)$ is strictly increasing on $\R$, and has an inverse $\beta(.):\mathcal{N}\mapsto \R $, where $\mathcal{N}:=\alpha'(\R)$ is an open interval. Let $\mathrm{cl}$ denote the closure of a set in $\R$, and extend $\beta(.)$ to a (possibly infinite valued) function on $\mathrm{cl}(\mathcal{N})$ by setting
\begin{align*}
	\beta(\sup\{\mathcal{N}\})=&+\infty\text{ if }\sup\{\mathcal{N}\}<\infty,\\
	\beta(\inf\{\mathcal{N}\})=&-\infty\text{ if }\inf\{\mathcal{N}\}>-\infty.
\end{align*}
We write $D(\cdot|\cdot)$ to denote the standard Kullback-Leibler divergence. Define a function $\gamma:\beta(\mathrm{cl}(\mathcal{N}))\mapsto [0,\infty]$ by setting
\begin{eqnarray*}
	\gamma(\theta):=&D(\mu_{\theta}\| \mu)= \theta \alpha'(\theta)-\alpha(\theta)&\text{ if }\theta\in \R=\beta(\mathcal{N}),\\
	\gamma(\infty) :=&D(\delta_{\sup\{\mathcal{N}\}}|\mu)&\text{ if }\sup\{\mathcal{N}\}<\infty,\\
	\gamma(-\infty) :=&D(\delta_{\inf\{\mathcal{N}\}}|\mu)&\text{ if }\inf\{\mathcal{N}\}>-\infty.
\end{eqnarray*}
\end{defn}

\begin{defn}\label{def:tilt2}
 Let $\mathcal{L}_p$ denote the space of all measurable functions $f:[0,1]\mapsto {\rm cl}(\mathcal{N})$ such that $\int_0^1|f(u)|^pdu<\infty$. 
Define a map $\Xi:\mathcal{L}_p\mapsto \widetilde{\mathcal{M}}$ as follows:
\\

For any $f\in \mathcal{L}_p$, if $(U,V)\sim \Xi(f)$, then $U\sim \mathrm{U}[0,1]$, and given $U=u$, one has
\begin{eqnarray*}
	V\sim &\mu_{\beta(f(u))}\text{ if }f(u)\in \mathcal{N},&\\
	=&\sup\{\alpha'(\R)\}\text{ if }f(u)=\sup\{\mathcal{N}\},&\text{ (this can only happen if }\sup\{\mathcal{N}\}<\infty),\\
	=&\inf\{\alpha'(\R)\}\text{ if }f(u)=\inf\{\mathcal{N}\},& \text{ (this can only happen if }\inf\{\mathcal{N}\}>-\infty).
\end{eqnarray*}
\end{defn}

Note that \cref{def:tilt} and \cref{def:tilt2} require three sub-cases for the definition of $\gamma$ and $V$. This is because we allow the support of the  base measure $\mu$ to be finite/infinite at either end. If we only worked with measures $\mu$ with support the whole of $\R$, we only need to consider the first sub-case.
%\begin{defn}\label{def:xi1}

%\end{defn}

\begin{defn}\label{def:g1}
Fix $W\in \mathcal{W}$ and let $\mathcal{L}_p$ be as defined above.
%denote the space of all measurable functions $f:[0,1]\mapsto \mathrm{cl}(\mathcal{N})$ such that $\int_0^1 |f(u)|^p\,du<\infty$. 
Define the functional $G_{W}(.):\mathcal{L}_p\mapsto \R$ by setting
$$G_{W}(f):=\int_{[0,1]^v}\left(\prod_{(a,b)\in E(H)} W(x_a,x_b)\right) \left(\prod_{a=1}^v f(x_a)dx_a\right),$$
whenever $G_{|W|}(|f|)<\infty$ (see~\cref{prop:freen} below for sufficient conditions).  %where $f$ and $W$ play the role of ${\bf X}$ and the matrix $Q_n$ respectively.}
\end{defn}
One can think of $f$ and $G_W(f)$ as continuum analogues of a vector ${\mathbf x}\in \R^n$, and the Hamiltonian $\mathbb{U}_n({\bf x})$ defined in \eqref{eq:U},  evaluated at $f$.

Finally, let $\mathfrak{L}_n(\cdot)$ be the empirical measure map from $\R^n$ to $\mathcal{M}$  defined by
%hich plays a crucial role in the sequel.
\begin{equation}\label{eq:ln}
\mathfrak{L}_n(\mathbf{x}) := \frac{1}{n}\sum_{i=1}^{n}\delta_{(\frac{i}{n},x_i)}, \qquad \mathbf{x}=(x_1,\ldots ,x_n).
\end{equation}
 Note that including the vector $(\frac{1}{n}, \frac{2}{n},\cdots,1)$ in the first co-ordinate allows us to write the multilinear form $U_n(X)$ introduced in \eqref{eq:U} as a \emph{nice} function of $\mathfrak{L}_n$ (see \cref{rem:twph} below).
%{\color{black} For any graphon $W \in \mathcal{W}$ and $q>0$, define $\|W\|^q_q:= \iint |W(x,y)|^q dx dy$. Further, for an $n \times n$ matrix $B$, define $\|B\|^q_q:= \sum_{i,j=1}^{n}|B_{i,j}|^q$.}

The following proposition characterizes the asymptotics of the log partition function/free energy in terms of an infinite dimensional optimization problem. % and gives a characterization for the class of optimizers in terms of a fixed point equation.

\begin{prop}\label{prop:freen}
Suppose that $\mu$ satisfies~\eqref{eq:tailp} for some $p\ge v$ and all $\lambda>0$. Let $\{Q_n\}_{n\ge 1}$ be a sequence of matrices such that \eqref{eq:cut_con} holds for some $W\in\mathcal{W}$, and 
\begin{equation}\label{eq:q}
    \limsup_{n\to\infty}\ \lVert W_{Q_n}\rVert_{q\Delta}<\infty,
\end{equation}
for some $q>1$ such that $\frac{1}{p}+\frac{1}{q}\leq 1$. Here, $\Delta$ denotes the maximum degree of $H$. Then the following conclusions hold.

(i) The function $G_{W}(.)$ is well-defined on $\mathcal{L}_p$, i.e., $G_{|W|}(|f|)<\infty$ for all $f\in\mathcal{L}_p$, where $\mathcal{L}_p$ is defined as in~\cref{def:tilt2}. 
%Suppose $\mu$ is a measure on $\R$ which satisfies \eqref{eq:pg1} for some $p\ge v$. 

% the probability measure defined in \eqref{eq:gibbs} with $\mathbb{U}_n({\bf X})=N_1(H,Q_n)$. 

(ii) With $Z_n(\theta)$ as in \eqref{eq:zn}, we have  $\sup_{n\ge 1} Z_n (\theta) < \infty$ and 
\begin{align}\label{eq:gibbsop}
\lim_{n\to\infty}Z_n(\theta)%&\sup_{t\in \R:I(t)<\infty}\{\theta t-I(t)\}  \\
=\sup_{f\in\mathcal{L}_p:\ \int_{[0,1]} \gamma(\beta(f(x)))dx<\infty}\left\{\theta G_{W}(f)-\int_{[0,1]}\gamma(\beta(f(x)))dx\right\}.
\end{align}

(iii) Let $d_{\ell}(\cdot,\cdot)$ denote the bounded Lipschitz metric between two probability measures. The supremum in \eqref{eq:gibbsop} is achieved on a set $F_\theta\subset \mathcal{L}_p$ (say), which satisfies \begin{equation}\label{eq:multweaklim}
	d_{\ell}(\mathfrak{L}_n({\bf X}),\Xi(F_\theta))\overset{P}{\longrightarrow}0
\end{equation}
 under ${\bf X}\sim \mathfrak{R}_{n,\theta}$ (as in \eqref{eq:gibbs}), where $\Xi$ is defined by \cref{def:tilt2}, and for any set of measures $\mathcal{A}$, $d_{\ell}(\nu_1,A)= \inf_{\nu_2 \in \mathcal{A}} d_{\ell}(\nu_1,\nu_2)$. Moreover $\Xi(F_{\theta})$ is compact in the weak topology.

%{\color{black}(iv) Recall $D(\cdot|\cdot)$ denotes the Kullback-Leibler divergence between two probability measures. Then $$\inf_{(P_1,\ldots ,P_n)} D(\otimes_{i=1}^n P_i|\mathfrak{R}_{n,\theta})=o(n)$$
%where the infimum is taken over all $n$-tuples of probability measures on $\R$.}
\end{prop}

\Cref{prop:freen} provides a characterization of the limiting free energy $Z_n(\theta)$ of the Gibbs measure $\mathfrak{R}_{n,\theta}$, introduced in \eqref{eq:gibbs}, via the expression~\eqref{eq:gibbsop}.
As illustrated by part (ii), the above proposition is a Mean-Field approximation result, and so, as often happens in Mean-Field approximation type results (see~\cite{augeri2020nonlinear,Chatterjee2016,eldan2018gaussian,yan2020nonlinear}), the limiting log normalizing constant/partition function is expressed in terms of a variational formula. Indeed, as indicated in \cref{def:g1}, $G_W(\cdot)$ is the limiting analogue of our Hamiltonian, and the function $\gamma(\cdot)$ introduced in \cref{def:tilt} is the entropy/Kullback-Leibler divergence term. The function $\beta(f)$ is the limiting analogue of a vector of tilts over which one optimizes in the variational approximation. 
The proof of the above proposition follows from~\cite[Theorems 1.1 and 1.6]{bhattacharya2024ldp}. Similar applications of large deviation principle to compute the asymptotics of the log-normalization constant can be found in~\cite[Theorem 3.1]{chatterjee2013estimating},~\cite[Theorem 1.5]{mukherjee2016estimation},~\cite[Theorem 2.8]{borgsdense2}, and \cite[Theorem 1.6]{Chatterjee2016}.

%\begin{remark}\label{rem:otasn}
\begin{remark}\label{rem:implic}
Under assumptions~\eqref{eq:cut_con}~and~\eqref{eq:q}, ~\cite[Theorem 2.13]{borgs2018p} gives
\begin{equation}\label{eq:W_q}
    \|W\|_{q \Delta} < \infty,
\end{equation} 
for any $q>1$, a fact that we use throughout the paper. We note in passing that under stronger assumptions on $H$ and $\mu$ (similar to~\cite[Theorem 1.2]{bhattacharya2022normal}) it is possible to forego the requirement in~\eqref{eq:q} and replace it with weaker assumptions.
\end{remark}

\begin{remark}[Relaxing assumption to cut distance]
    While we have assumed convergence of $W_{Q_n}$ to $W$ in the cut norm (see \eqref{eq:cut_con}), for many of our results it can indeed be relaxed to convergence in cut distance (see \cite[Equation 2.1]{borgs2018p}). In fact, by a standard permutation invariance argument on $\mathbb{U}_n({\bf x})$, there is no loss of generality in assuming cut norm convergence instead of cut distance convergence for \cref{prop:freen}, part (ii).  However, part (iii) of \cref{prop:freen} does require the stronger assumption of convergence in cut norm, and convergence in cut distance is simply not enough. 
\end{remark}

\subsubsection{Replica-symmetry} The above proposition shows that the infinite dimensional optimization problem in the second line of \eqref{eq:gibbsop} is useful for understanding the Gibbs measure $\mathfrak{R}_{n,\theta}$ (see parts (ii) and (iii)).
    %Our main focus in this paper is understanding Gibbs measures, {\color{blue}but as it( turns out, we can sometimes relate the constrained optimization to the unconstrained optimization, and simplify the rate function in \eqref{eq:i1} --- see \cite[Lemma 1.4]{bhattacharya2024ldp} for further details.} 
%\textbf{Replica-symmetry}: 
A natural question is when does the set of optimizers of \eqref{eq:gibbsop} consist only of constant functions. Equivalently, borrowing terminology from statistical physics, we want to understand the \enquote{replica-symmetry} phase of the Gibbs measure $\mathfrak{R}_{n,\theta}$. %Note that if $f$ is a constant function, then the corresponding measure $\Xi(f)$ is a product measure.  
Our first main result provides sufficient conditions for optimizers to be constant functions, and provide examples to show the necessity of those conditions. For this we need the following two definitions.
\begin{defn}\label{def:symmfndef}
Given a symmetric matrix $Q_n$, define a symmetric tensor 
$$\mathrm{Sym}[Q_n](i_1,\ldots,i_v):=\frac{1}{v!}\sum_{\sigma \in S_v} \prod_{(a,b)\in E(H)} Q_n\Big(i_{\sigma(a)}, i_{\sigma(b)}\Big) ,$$
where $S_v$ denote the set of all permutations of $[v]$.
Using this definition, a simple computation gives
\begin{align}\label{eq:ref_cond}
\mathbb{U}_n(\mathbf{X})=\frac{1}{n^v}\sum_{(i_1,\cdots,i_v)\in \mathcal{S}(n,v)}\Big(\prod_{a=1}^v X_{i_a}\Big)\mathrm{Sym}[Q_n](i_1,\cdots,i_v).
\end{align}

In a similar vein, given a symmetric function $W\in \mathcal{W} $, define the symmetric function
$$\mathrm{Sym}[W](x_1,\ldots,x_v):=\frac{1}{v!}\sum_{\sigma \in S_v} \prod_{(a,b)\in E(H)} W\Big(x_{\sigma(a)}, x_{\sigma(b)}\Big) .$$
As an example, if $Q_n$ is the adjacency matrix of a simple labeled graph, then $\mathrm{Sym}[Q_n](i_1,\cdots,i_v)$ counts the (scaled) number of copies of $H$ in $Q_n$ spanned by the vertices $(i_1,\cdots,i_v)$. A similar interpretation holds for $\mathrm{Sym}[W]$, which can be thought of as the continuum analogue of $\mathrm{Sym}[Q_n]$.
\\

 For any $W\in \mathcal{W}$, let 
\begin{align}\label{eq:rowsum}\mathcal{T}[\mathrm{Sym}[W]](x):=\int_{[0,1]^{v-1}} \mathrm{Sym}[W](x,x_2,\ldots ,x_v)\prod_{a=2}^v \,dx_a,\end{align}
provided the integral exists, and is finite. This can be thought of as the continuum analogue of the number of copies of $H$  passing through the vertex $x$ in the limiting graphon $W$. 
\end{defn}

\noindent To illustrate, if $H=K_{1,2}$, then we have
\begin{align*}
    \mathrm{Sym}[Q_n](i,j,k)=&\frac{1}{3}\Big[Q_n(i,j)Q_n(i,k)+Q_n(i,k) Q_n(j,k)+Q_n(i,j)Q_n(j,k)\Big], \\
    \mathrm{Sym}[W](x_1,x_2,x_3)=&\frac{1}{3}\Big[W(x_1,x_2) W(x_1,x_3)+W(x_1,x_2)W(x_2,x_3)+W(x_1,x_3)W(x_2,x_3)\Big]. 
\end{align*}
On the other hand, if $H=K_3$, then
\begin{align*}
    \mathrm{Sym}[Q_n](i,j,k)=Q_n(i,j)Q_n(j,k)Q_n(i,k), \\
    \mathrm{Sym}[W](x_1,x_2,x_3)=W(x_1,x_2) W(x_1,x_3)W(x_2,x_3). 
\end{align*}
In both cases, we have 
$$\mathbb{U}_n({\bf X})=\frac{1}{n^3}\sum_{(i_1,i_2,i_3)\in \mathcal{S}(n,3)} \left(\prod_{a=1}^3 X_{i_a}\right)\mathrm{Sym}[Q_n](i_1,i_2,i_3).$$

\begin{defn}\label{def:stochastic_nn}
Let $\mu$ be a measure in $\mathbb{R}$ and $\beta(.),\gamma(.),\cN$ be as in~\cref{def:tilt}. We will say $\mu$ is stochastically non-negative, if for any $t>0$, if $-t\in \cN$ then $t\in \cN$, and $\gamma(\beta(t))\le \gamma(\beta(-t)).$ 
\end{defn}

We now state the first main result of this paper.
\begin{thm}[Replica-symmetry]\label{prop:propopt}
Suppose we are in the setting of~\cref{prop:freen}. Then $\mathcal{T}[\mathrm{Sym}[W]](.)$ is finite a.s., and  the following conclusions hold:

(i) Any maximizer $f$ of the optimization problem~\eqref{eq:gibbsop} satisfies the fixed point equation
\begin{equation}\label{eq:propoptshow} 
	f(x)\stackrel{a.s.}{=}\alpha'\left(\theta v\int_{[0,1]^{v-1}} \mathrm{Sym}[W](x,x_2,\ldots ,x_v)\left(\prod_{a=2}^v f(x_a)\,dx_a\right)\right).%\quad \mbox{a.e.}\ x\in [0,1].
\end{equation}

(ii)
If $\theta\neq 0$ and $\mathcal{T}[\mathrm{Sym}[W]](\cdot)$ is not constant a.s., none of the maximizers in~\eqref{eq:gibbsop} are non-zero constant functions.

(iii) If $\mathcal{T}[\mathrm{Sym}[W]](\cdot)$ is constant a.s., and  $\theta W$ is strictly positive a.s., % a.e. $(x,y)\in [0,1]^2$, 
then all of the maximizers in~\eqref{eq:gibbsop} are constant functions, provided either 
(a) $v$ is even or (b) $\mu$ is stochastically non-negative according to \cref{def:stochastic_nn}.
%For any $t>0$, if $-t\in \mathcal{N}$ then $t\in \mathcal{N}$, and
%$\gamma(\beta(t))\leq \gamma(\beta(-t))$.
%Here we set $\beta(x)=\infty$ (respectively $-\infty$) for when $x$ is the upper (respectively, lower) end point of $\alpha'(\R)$. Accordingly $\gamma(\infty)$ (respectively, $\gamma(-\infty)$) is the point mass on the upper (respectively, lower) end point of $\alpha'(\R)$.}
\end{thm}

\begin{remark}\label{rem:fntomeas}
%[Connections to $\mathfrak{L}_n$ in~\eqref{eq:ln}]
%\cref{prop:propopt} has immediate consequences in the context of weak limits. 
As is often the case, the solution to the Mean-Field optimization problem in \cref{prop:freen}, part (ii), is characterized by the  fixed point equation in \cref{prop:propopt} part (i). We utilize this fixed point equation to characterize sufficient conditions under which the optimizers are constant functions. This is of possible interest, as
 the construction of the map $\Xi$ in~\cref{def:tilt2} implies that  $f\in \mathcal{L}_p$ is a constant function iff $\Xi(f)\in \wmm$ is a product measure. Thus under the conditions of~\cref{prop:propopt} part (iii), any weak limit of the empirical measure $\mathfrak{L}_n$ (introduced in~\eqref{eq:ln}) under $\mathfrak{R}_{n,\theta}$ is a product measure. %Intuitively, this means that the $X_i$'s are asymptotically exchangeable under $\mathfrak{R}_{n,\theta}$. 
Further these product measures have first marginal $\mathrm{Unif}[0,1]$ and second marginal of the form $\mu_{\beta(t)}$ where $t$ satisfies the fixed point equation
$$t=\alpha'(\theta v t^{v-1} ),$$
by~\cref{prop:propopt} part (i), assuming $\mathcal{T}[\mathrm{Sym}[W]] \equiv 1$ for simplicity. In particular if $v=2$ and $\mu$ is supported on $\{-1,1\}$ with $\mu(1)=\exp(2B)/(1+\exp(2B))$ for some $B\in\R$, the above fixed point equation simplifies to
$$t=\tanh(2\theta t+B).$$
The solutions of this equation for $\theta\geq 0$, $B\in\R$ are well understood, see e.g.,~\cite[Page 2]{Comets1991}~and~\cite[page 144, Section 1.1.3]{dembo2010gibbs}. In~\cref{lem:fixsol} below we study a broader class of measures $\mu$, which includes this as a special case.
\end{remark}

Part (iii) of \cref{prop:propopt} provides a sufficient condition for replica symmetry, one of which is the requirement that the base measure $\mu$ is stochastically non-negative, which can be difficult to verify. In the lemma below, we provide simple sufficient conditions for a measure $\mu$ to be stochastically non-negative.

\begin{prop}\label{prop:mu_suff}
     $\mu$ is stochastically non-negative if one of the following conditions hold: 

(a) $\mu$ is supported on the non-negative half line, or 

(b) $\mu$ is a non-negative  tilt of a symmetric measure, i.e.~there exists $B\ge0$ and a symmetric measure $\mu^{(s)}$, such that $\frac{d\mu}{d\mu^{(s)}}(x)=\exp(Bx-C(x))$.
\end{prop}

\begin{ex}[Necessity of conditions on $\mu$ from \cref{prop:mu_suff}] To demonstrate the necessity of stochastic non-negativeness for $\mu$ in~\cref{prop:propopt} part (iii) and \cref{prop:mu_suff}, we provide an example of a measure $\mu$  and a graphon $W$ with $\mathcal{T}[\mathrm{Sym}[W]](\cdot)$ is constant a.s.~and $\theta W\ge 0$, but none of the maximizers are constant functions. Set $H=K_3$ and
\begin{align*}
W(x,y)=&0\text{ if }(x,y)\in \Big[0,\frac13\Big)^2 \text{ or }(x,y)\in \Big[\frac13,\frac23\Big)^2\text{ or }(x,y)\in  \Big[\frac23,1\Big)^2,\\
=&1\text{ otherwise }.
\end{align*}
In this case $W$ is the graphon corresponding to a complete tripartite graph. Let $\mu$ be a probability measure on $\{-1,1\}$ with $\mu(1)= e^{-4}/(e^{-4} + e^4)$. Set $\theta=9$ and 
\begin{align*}
f(x):=\begin{cases}-0.99 & \text{ if }0\leq x<\frac23\\ 
 +0.83 & \text{ if }\frac23\leq x\leq 1.\end{cases}
\end{align*}
Numerical computations show that
\begin{align*}
    \theta G_{W}(f)-\int_{[0,1]}\gamma(\beta(f(x)))dx > \sup_{t \in [-1,1]} \left\{\frac{2}{3}\theta t^v -\gamma(\beta(t))\right\}.
\end{align*}
Thus all global optimizers must be non-constant functions when $v=3$, and the measure $\mu$ is a negative tilt of a symmetric distribution.  Note that the function $W$ in this counterexample is not strictly positive a.s., but this can be circumvented by a continuity argument to allow for small positive values in the diagonal blocks.
\end{ex}

\begin{ex}[Necessity of $\theta>0$]\label{rem:thetnec}
The requirement $\theta> 0$ is indeed necessary in~\cref{prop:propopt} part (iii). To see a counterexample, consider the case $v=2$ and 
\begin{align}\label{eq:baje_W}
W(x,y)=&2\text{ if }(x,y)\in (0,.5)\times (.5, 1) \text{ or }(x,y)\in (.5,1)\times (0,.5), \nonumber\\
=&0\text{ otherwise }.
\end{align}
In words $W$ is the (scaled) limiting graphon of a complete bipartite graph. Let $\mu$ be a compactly supported probability measure which is symmetric about $0$. If $\theta$ is large and negative, then it is not hard to show that no optimizer of~\eqref{eq:gibbsop} is a constant function, and any optimizer is of the form
\begin{align}\label{eq:two_piece_f}
f(x)=&a\text{ if }0<x<0.5; \nonumber \\
=&b\text{ if }.5<x<1,
\end{align}
where $a$ and $b$ are of opposite signs.  Similar to the previous remark, even though $W$ is not strictly positive a.s., this can be circumvented by a continuity argument to allow for small positive values in the diagonal blocks. 
\end{ex}

\begin{ex}[Presence of both constant and non-constant optimizers]
    The requirement $\theta W\stackrel{a.s.}{>}0$ is necessary to ensure that all global maximizers are constants in~\cref{prop:propopt} part (iii). To see this, we now give an example with simultaneous constant and non-constant maximizers. Indeed, consider $H=K_2$ and $W$ given by
     \begin{align}
 W(x,y)=&2\text{ if }(x,y)\in (0,.5)\times (0, .5) \text{ or }(x,y)\in (.5,1)\times (.5,1), \nonumber\\
 =&0\text{ otherwise }.
 \end{align}
 In words, $W$ is the (scaled) limiting graphon for a disjoint union of complete graphs.
Let $\mu= \frac{1}{2} (\delta_{+1}+\delta_{-1})$. It is not hard to show any optimizer is two-piece constants, $f$ is given by~\eqref{eq:two_piece_f}, where $a,b$ are solutions of $x= \tanh(\theta x)$. If $\theta>1$, the equation has three solutions $(t,0,-t)$ with $t= t_\theta >0$. Hence, there exists four optimizers $(t,t), (-t,-t), (t,-t), (-t,t)$. Two of these optimizers are constant functions whereas two of them are not.
\end{ex}

%We now present our weak law results. 

\subsubsection{Weak laws and tail bounds} Having studied the optimizers of the limiting free energy under model \eqref{eq:gibbs} in \cref{prop:propopt}, the next natural question is to obtain weak laws for various statistics of interest under \eqref{eq:gibbs}. Some popular examples include the Hamiltonian $\mathbb{U}_n$ (see \eqref{eq:U}), the global magnetization $\sum_{i=1}^n X_i$ or other interesting linear combinations of $X_i$'s. In the sequel, we will obtain a family of such weak limits in a unified fashion. Along the way, we will show that the vector of conditional expectations of $X_i$ (given $(X_j)_{j\ne i}$) converge to the optimizers of the limiting free energy (see \cref{cor:conmeanint} for a clean version of this statement).
The key tool that will help us address these question simultaneously is a sharp analysis of the vector of \emph{local fields}, which we define below.
%\begin{defn}
%With $Q_n$ denoting an $n\times n$ matrix as before, define a tensor $\tQh$ by setting
%\begin{equation}\label{eq:symmdef}
%	\tQh(i_1,\ldots ,i_v):=\frac{1}{v!}\sum_{\sigma\in %S_v}\prod_{(a,b)\in E(H)}Q_n(i_{\sigma(a)},i_{\sigma(b)}),
%	\end{equation}
%	for $(i_1,\ldots ,i_v)\in [n]^v$. Note that $\tQh$ above is the symmetrized version of $Q_n$, similar to what was done in~\cref{def:symmfndef}. 
 %\end{defn}
\begin{defn}[Local fields]\label{def:ocalmag}
Define the local magnetization/field at the $i$-th observation as follows:
\begin{align}\label{eq:mdeff}
	m_i:=\frac{v}{n^{v-1}}\sum_{(i_2,\ldots ,i_v)\in \mathcal{S}(n,v,i)} \tQh(i,i_2,\ldots ,i_v)\left(\prod_{a=2}^v X_{i_a}\right),
\end{align}
for $i\in [n]$. Here, for  $i \in [n]$, $\mathcal{S}(n,v,i)$ denotes the set of all distinct tuples in
$[n]^{v-1}$, such that none of the elements equal to $i$, and $\mathrm{Sym}[Q_n]$ is as in \cref{def:symmfndef}. Set $\mathbf{m}:=(m_1,\ldots ,m_n)$. Following \eqref{eq:ln},  consider the associated empirical measure
\begin{equation}\label{eq:empmi}
	\tfm=\frac{1}{n}\sum_{i=1}^n \delta_{\left(\frac{i}{n},m_i\right)},
\end{equation}
\end{defn}
The $m_i$'s defined in~\eqref{eq:mdeff} are local magnetizations/fields which capture ``how well'' one can predict $X_i$ given all $X_j$, $j\neq i$. More precisely, using the fact that the Hamiltonian $\mathbb{U}_n({\bf X})$ can be written in terms of the symmetrized tensor $\mathrm{Sym}(Q_n)$ (see \eqref{eq:ref_cond}), 
under model~\eqref{eq:gibbs} the conditional distribution of $X_i$ given $X_j$, $j\neq i$ is completely determined by $m_i$ in the following manner:
  \begin{equation}\label{eq:cond}
  \frac{d\mathfrak{R}_{n,\theta}(x_i|x_j,\ j\neq i)}{d\mu(x_i)}=\mu_{\theta m_i},
  \end{equation}
  with $\mu_{.}$ as in \cref{def:tilt}. In particular, %by \eqref{eq:cond}, $\E[X_i|X_j,\ j\neq i]$ which is the natural predictor for $X_i$ given all $X_j$, $j\neq i$, can be rewritten as:
  %It follows 
  by~\eqref{eq:cond} and~\cref{def:tilt}, %that
     \begin{equation}\label{eq:condex}
     \E[X_i|X_j,\ j\neq i]=\alpha'(\theta m_i).
     \end{equation}
  Consequently, understanding the behavior of the vector ${\bf m}$ or
  \begin{equation}\label{eq:alpha_define}
      \boldsymbol{\alpha}:=\alpha'(\theta {\bf m})=(\alpha'(\theta m_1),\ldots ,\alpha'(\theta m_n)),
  \end{equation}
  plays an important role in obtaining correlation bounds, tail decay estimates and fluctuations for such Gibbs measures (see~\cite{gheissari2019ising,Chatterjee2007,deb2020fluctuations,deb2020detecting,bhattacharya2025sharp}~and the references therein). The major focus of the existing literature is on the special case of $v=2$ and $\mu$ supported on $\{-1,1\}$, which is not needed in our paper.
  \\
  
  %Should we add a remark about applications of studying $m_i$'s? Mention studying MLE/PL consistency+weak laws and then say, we are focusing on weak laws here?
   Our second main result gives a weak law for the vector ${\bf m}$, in terms of the empirical measure $\tfm$. For any measure $\nu\in \wmm_p$ (see \cref{def:M}), define 
\begin{equation}\label{eq:varthdef}
\vartheta_{W,\nu}(u):=v\E\left[\mathrm{Sym}[W](U_1,\ldots ,U_v)\left(\prod_{a=2}^v V_a\right)\bigg|U_1=u\right],\quad \mbox{for}\ u\in [0,1],
\end{equation}
	where $(U_1,V_1),\ldots ,(U_v,V_v)\overset{i.i.d.}{\sim}\nu$, and $\mathrm{Sym}[.]$ is as in~\cref{def:symmfndef}. In the Theorem below, we will provide sufficient conditions under which $\vartheta_{W,\nu}(\cdot)$ is well-defined. While the definition of $\vartheta_{W,\nu}$ may seem abstract at first, it simplifies nicely in the context of \cref{prop:propopt}. To see this, assume that $(U,V) \sim \nu\in \Xi(F_{\theta})$ for some $\theta\in\R$, where $F_{\theta}$ is as in~\cref{prop:freen} part (iii).  By \cref{def:tilt2} we have $\E[V|U=u]=f(u)$ for some $f\in F_{\theta}$. So, 
 \begin{align*}
     \vartheta_{W,\nu}(u)&=v\E\left[\mathrm{Sym}[W](u,U_2,\ldots ,U_v)\left(\prod_{a=2}^v \E[V_a|U_a]\right)\right]\\ &=v\int_{[0,1]^{v-1}} \mathrm{Sym}[W](u,u_2,\ldots ,u_v)\left(\prod_{a=2}^v f(u_a)\,d u_a\right).
 \end{align*}
 By invoking \eqref{eq:propoptshow}, we then get for a.e. $u\in [0,1]$,
 \begin{equation}\label{eq:calter}
     f(u)=\alpha'(\theta \vartheta_{W,\nu}(u)).
 \end{equation}
 Therefore, there is a direct one-one correspondence between the two sets of functions $F_{\theta}$ and $\{\vartheta_{W,\nu}(.),\nu\in \Xi(F_{\theta})\}$. This observation will be crucial in the proof of \cref{cor:conmeanint} below.

 We are now in a position to state our second main result.
  
\begin{thm}\label{lem:wealimi}
	Suppose~\eqref{eq:tailp} holds for some $p\in [v,\infty]$. Assume that~\eqref{eq:cut_con} and \eqref{eq:q} holds with some $q$ satisfying $\frac{1}{p}+\frac{1}{q}<1$. Then the following conclusions hold: 
 
 \noindent (i) With $F_{\theta}$ as in~\cref{prop:freen} part (iii), 
 and $\vartheta_{W,\nu}$ as in \eqref{eq:varthdef}, we have $\Xi(F_{\theta})\subseteq \wmm_p$ and $\vartheta_{W,\nu}(\cdot)$ is well-defined for every $\nu\in \wmm_p$,  a.s. on $[0,1]$. 
 
 \noindent (ii) Set
	$$\mathfrak{B}_{\theta}:=\{\mathrm{Law}(U,\vartheta_{W,\nu}(U)): \nu\in \Xi(F_{\theta})\}\subset \wmm.$$
	Then we have:
\begin{equation}\label{eq:weak}	
 d_{\ell}\left(\tfm,\mathfrak{B}_{\theta}\right)\overset{P}{\longrightarrow} 0.
 \end{equation}
\end{thm}

The above theorem gives a weak limit for the empirical measure of the local field vector $\bf m$. The weak law for the empirical measure of the conditional means (introduced in \eqref{eq:condex}) then follows from \cref{lem:wealimi} by a continuous mapping type argument. The limit in that case will naturally be the set 
\begin{align}\label{eq:tildef}
\mathfrak{B}^*_{\theta}:=\{\mbox{Law}(U,\alpha'(\theta \vartheta_{W,\nu}(U))):\ \nu\in \Xi(F_{\theta})\},
\end{align}
We stress here that \emph{no assumption of replica-symmetry is necessary} for~\cref{lem:wealimi} to hold.

\begin{remark}\label{rem:extend}
Given $\nu\in\Xi(F_{\theta})$, let $\nu^{(2|1)}(\cdot)$ denote the conditional distribution of the second coordinate given the first coordinate. The proof of~\eqref{eq:weak} can  be adapted to show the following stronger conclusion:
$$d_{\ell}\left(\frac{1}{n}\sum_{i=1}^n \delta_{\left(\frac{i}{n},X_i,m_i\right)},\mathfrak{B}^{(2|1)}_{\theta}\right)\overset{P}{\longrightarrow}0,$$
where
$$\mathfrak{B}^{(2|1)}_{\theta}:=\{(U,\nu^{(2|1)}(U),\vartheta_{W,\nu}(U)):\ \nu\in \Xi(F_{\theta})\}.$$
Since this version is not necessary for our applications, we do not prove it here.
\end{remark}

%ND --- The limiting set $\mathfrak{B}_{\theta}$ suggests that in the limit the $X_i$'s behave like they are independent. Is it possible to make an intuitive connection to the averaging effect of the $m_i$'s based on the limiting set?

\noindent 
In order to obtain  weak laws for common statistics of interest using \cref{lem:wealimi}, we require appropriate tail estimates for the $X_i$'s, the $m_i$'s and the $\alpha'(\theta m_i)$'s. In particular, we will derive \emph{exponential tail bounds} for the said quantities below, which is of possible independent interest.
\begin{thm}\label{lem:basicprop}
    Consider the same setting as in \cref{prop:freen}. Then the following conclusions hold:
    \begin{enumerate}

     \item[(i)] There exists $K_0>0$, free of $n$, such that for all $K\ge K_0$ the following holds:
     \begin{align}\label{eq:lip11}
  \mathfrak{R}_{n,\theta}\left(\frac{1}{n}\sum_{i=1}^n |X_i|^p\ge K\right)\le 3 \exp\bigg(-n\frac{K}{K_0}\bigg) ,\\
    \label{eq:lip13}
    \mathfrak{R}_{n,\theta}\left(\frac{1}{n}\sum_{i=1}^n |m_i|^q\ge K\right)<3 \exp\bigg(-n\frac{K}{K_0}\bigg),\\
    \label{eq:lip14}
    \mathfrak{R}_{n,\theta}\left(\frac{1}{n}\sum_{i=1}^n |\alpha'(\theta m_i)|^p\ge K\right)<3 \exp\bigg(-n\frac{K}{K_0}\bigg).
    \end{align}
    As a consequence, respectively we have $$\sum_{i=1}^n \E|X_i|^p=O(n),\quad \sum_{i=1}^n \E|m_i|^q=O(n),\quad \sum_{i=1}^n \E|\alpha'(\theta m_i)|^p=O(n).$$
    
    \item[(ii)] Moreover, 
     \begin{align}\label{eq:lip12}
    \sup_{\nu\in\Xi(F_{\theta})} \fmm_p(\nu)<\infty, \quad \sup_{\nu\in\mathfrak{B}_{\theta}} \fmm_q(\nu)<\infty, \quad \sup_{\nu\in\mathfrak{B}^*_{\theta}} \fmm_p(\nu)<\infty,
    \end{align}
where $\fmm_p(\nu), \fmm_q(\nu)$ are given by \cref{def:M}.
    \end{enumerate}
\end{thm}

In order to interpret part (ii) of the above theorem, note that $\mathfrak{L}_n({\bf X})$, $\mathfrak{L}_n({\bf m})$, and $\mathfrak{L}_n({\boldsymbol{\alpha}})$ converge weakly in probability to the set of probability measures $\Xi(F_{\theta})$ (by \eqref{eq:multweaklim}), $\mathfrak{B}_{\theta}$ (by \eqref{eq:weak}), and $\mathfrak{B}^*_{\theta}$ (as discussed around \eqref{eq:tildef}) respectively.  \cref{lem:basicprop} part (ii) shows that these limiting set of measures have uniformly bounded moments of a suitable order. %in the appropriate $L^p$ or $L^q$  topologies.

\vspace{0.05in}

\noindent   In view of the above results, coupled with the observation made in \eqref{eq:calter}, it seems intuitive to expect a correspondence between elements of $F_{\theta}$ and the map $u\mapsto \alpha'(\theta m_{\lceil n u\rceil}))$. This is made precise in the following corollary.

\begin{cor}\label{cor:conmeanint}
    In the setting of \cref{lem:wealimi}, we have 
    \begin{equation}\label{eq:lpcon}
 \inf_{f\in F_{\theta}}\int_0^1 |\alpha'(\theta m_{\lceil n u\rceil})-f(u)|^{p'}\,du\overset{P}{\longrightarrow} 0,
 \end{equation}
	for any $p'<p$, under the measure~\eqref{eq:gibbs}.	
\end{cor}

\begin{remark}\label{rem:probin}
     
     Recall from \eqref{eq:condex} that $\alpha'(\theta m_i)=\E[X_i|X_j,\ j\neq i]$. Therefore~\eqref{eq:lpcon} shows that the functions in $F_{\theta}$ are ``close" to the vector of conditional expectations of $X_i$'s given all the other coordinates. In particular, if  $F_{\theta}=\{f\}$ is a singleton and $f$ is continuous on $[0,1]$, then~\eqref{eq:lpcon}~and~\eqref{eq:condex} together imply that
     $$\frac{1}{n}\sum_{i=1}^n \bigg|\alpha'(\theta m_i)-f\left(\frac{i}{n}\right)\bigg|^{p'}\overset{P}{\longrightarrow} 0.$$
      For the special case where  $v=2$, $\mu$ is supported on $\{-1,1\}$ with $\mu(1)=\mu(-1)=0.5$, and $\mathcal{T}[\mathrm{Sym}[W]](\cdot)=1$ a.s, $F_{\theta}$ consists of two constant functions of the form $\{-t_{\theta},t_{\theta}\}$ for some $t_{\theta}>0$ (see~\cite[Page 144]{dembo2010gibbs}; also see~\cref{lem:fixsol} for %a treatment of more 
     general $\mu$). The symmetry of $\mathfrak{R}_{n,\theta}$ around $0$, coupled with~\eqref{eq:lpcon} implies
     $$\min \left(\frac{1}{n}\sum_{i=1}^n \bigg|\tanh(\theta m_i)-t_{\theta}\bigg|, \frac{1}{n}\sum_{i=1}^n \bigg|\tanh(\theta m_i)+t_{\theta}\bigg| \right) \overset{P}{\longrightarrow} 0.$$
     \begin{remark}
         Note that the above displays are not true with $X_i$ replacing $\tanh(\theta m_i)=\E[X_i|X_j,\ j\neq i]$. This shows $m_i
    $'s are \enquote{more concentrated} than $X_i$'s.
     \end{remark} %additional averaging effect in~\eqref{eq:mdeff} is crucial.
\end{remark}

As applications of~\cref{lem:wealimi} and~\cref{lem:basicprop}, we obtain  weak limits for linear statistics as well as the Hamiltonian $\mathfrak{R}_{n,\theta}$ under replica-symmetry, both of which are of independent interest.

\begin{thm}\label{prop:higherord}
Suppose ${\mathbf X}\sim \mathfrak{R}_{n,\theta}$ (defined via~\eqref{eq:gibbs}) for some base measure $\mu$ which satisfies~\eqref{eq:tailp} for some $p\in [v,\infty]$. Assume that either $v$ is even or $\mu$ is stochastically non-negative. Further, suppose  that $\{Q_n\}_{n\ge 1}$ satisfies~\eqref{eq:cut_con}~and~\eqref{eq:q} for some $q>1$ satisfying $\frac{1}{p}+\frac{1}{q}<1$, and that the limiting $W$ is strictly positive and satisfies $\mathcal{T}[\mathrm{Sym}[W]](x)=1$ a.s.~Then, setting
\begin{align}\label{eq:defset}
\mathcal{A}_{\theta}:=\argmin_{t\in\mathcal{N}} \left[\gamma(\beta(t))-\theta t^v\right],
\end{align}
we have $\mathcal{A}_{\theta}$ is a finite set, and the following conclusions hold:
\\

(i)  Suppose $\{c_i\}_{i\geq 1}$ is a real sequence satisfying $\sum_{i=1}^n c_i=o(n)$, and $\sum_{i=1}^n |c_i|^r=O(n)$ for some $r$ such that $\frac{1}{p}+\frac{1}{r}<1$. 
Then we have: $$\frac{1}{n}\sum_{i=1}^n c_iX_i\overset{P}{\longrightarrow}0.$$

(ii) If we replace $\sum_{i=1}^n c_i=o(n)$ with $n^{-1}\sum_{i=1}^n c_i\to c_0$, then 
$$d_{\ell}\left(\frac{1}{n}\sum_{i=1}^n c_iX_i,\{{c_0t}:\ t\in \mathcal{A}_{\theta}\}\right)\overset{P}{\longrightarrow}0.$$

(iii) The Hamiltonian satisfies
$$d_{\ell}\left(\frac{1}{n}\sum_{i=1}^n X_i m_i,\{{vt^v}:\ t\in \mathcal{A}_{\theta}\}\right)\overset{P}{\longrightarrow}0.$$

\end{thm}

\begin{remark}
Part (i) of the above Theorem shows that for contrast vectors ${\bf c}$ (i.e.~$\sum_{i=1}^nc_i=0$) which are delocalized (in the sense $\sum_{i=1}^n|c_i|^r=O(n)$), the corresponding linear statistic exhibits a \emph{universal} behavior across general Gibbs measures with higher order multilinear interactions which doesn't depend on the matrix sequence $\{Q_n\}_{n\ge 1}$, as long as $\mathcal{T}[\mathrm{Sym}[W]](.)$ is constant, i.e.~the symmetrized tensor is regular. In a similar manner, part (ii) gives a universal behavior for the global magnetization $\bar{X}$ for regular tensors. Universality results for $\bar{X}$ were earlier obtained for regular Ising models, which correspond to $v=2$ and $\mu$ is supported on $\{-1,1\}$ with $\mu(-1)=\mu(+1)=0.5$ (see \cite[Theorem 2.1]{Basak2017} and \cite[Theorems 1.1---1.4]{deb2020fluctuations}). In this special case, for $\theta>0.5$ we have
%special case when $v$  As an example, suppose $\mu$ is supported on $\{-1,1\}$ with $\mu(-1)=0.5$, then for $\theta>0.5$, 
$\mathcal{A}_{\theta}=\{-t_{\theta},t_{\theta}\}$ for some $t_{\theta}>0$ (see~\cref{rem:probin}). In this case symmetry implies (see~\cref{prop:isingdom} part (ii) below for a more general result) that:
%global magnetization $\bar{X}$ has a non-degenerate weak limit by symmetry, i.e.,
$$\bar{X}\overset{d}{\longrightarrow} \frac{\delta_{-t_{\theta}}+\delta_{t_{\theta}}}{2},\quad 
%However for all contrasts, the limit is degenerate at $0$ and the Hamiltonian too has a degenerate limit given by
\frac{1}{n}\sum_{i=1}^n X_im_i\overset{P}{\longrightarrow} 2t_{\theta}^2.$$
%On the other hand, for all contrasts ${\bf c}$ which are not delocalized (i.e.~satisfying $\sum_{i=1}^n|c_i|^r=O(n)$)  the limit is $0$. 
The more recent work of \cite{lacker2022mean} demonstrates universality for quadratic interactions for log concave $\mu$ (see \cite[Theorem 1.1 and Corollary 1.4]{lacker2022mean}). We note that the results of the current paper requires neither quadratic interactions, nor log-concave base measures.
In the following subsection, we will apply~\cref{prop:higherord} to analyze a broad class of Gibbs measures which are not necessarily quadratic, and cover cubic and higher order interactions (see \cref{prop:genopt} below).
\end{remark}

\subsection{Examples}\label{sec:isingpotts} 
We now apply our general results to analyze some Gibbs measures of interest. In~\cref{prop:propopt}, we proved that for regular tensors (i.e.~when $\mathcal{T}[\mathrm{Sym}[W]](\cdot)=1$ a.s.) the optimization problem~\eqref{eq:gibbsop} has only constant functions as optimizers, under mild assumptions on $\mu$ or $v$. In this section, we %present some examples under the assumption $\mathcal{T}[\mathrm{Sym}[W]](\cdot)=1$, 
focus on particular examples of the regular case, and provide more explicit description for the optimizers. %Before diving into our examples, let us consider the following related constrained optimization problem:
%\begin{equation}\label{eq:i1}
%	I(t):=\inf_{f\in\mathcal{L}_p:\ \int_{[0,1]} \gamma(\beta(f(x)))dx<\infty,\ G_{W}(f)=t}  \int_{[0,1]} \gamma(\beta(f(x)))dx.
%\end{equation}
%Note that~\eqref{eq:i1} can be viewed as a constrained version of the optimization problem~\eqref{eq:gibbsop}. This problem arises naturally as the rate function of a large deviation problem as described below.
%\begin{prop}\label{thm:multilinear}
%\end{prop}
%The above result was proved in our paper~\cite[Corollary 1.5]{bhattacharya2024ldp}. Related results under stronger assumptions can be found in \cite{mukherjee2020replica,cook2020large,chatterjee2016nonlinear}, which makes~\eqref{eq:i1} a problem of independent interest. We will see in the sequel that while the replica-symmetry prediction holds in the unconstrained (see~\eqref{eq:gibbsop}) problem, the same may not be true in the constrained problem (see~\eqref{eq:i1}). 

\subsubsection{Quadratic interaction models with symmetric base measure}

Suppose $\mu$ is a probability measure on $\R$ which is symmetric about the origin. Define a Gibbs measure on $\R^n$ by setting
\begin{align}\label{eq:ising}
\frac{d\isR}{d\mu^{\otimes n}}({\bf X})=\exp\Big(\frac{\theta}{n}\sum_{i\neq j}  Q_n(i,j) X_i X_j+B\sum_{i=1}^n X_i-n\isZ_n(\theta,B)\Big),
\end{align}
where $\theta\geq 0$, $B\in\R$. In particular if $\mu$ is supported on $\{-1,1\}$ reduces the above model to the celebrated Ising model, which has attracted significant attention in probability and statistics (c.f.~\cite{Basak2017,dembo2010gibbs,Ellis1978,Sly2014, deb2020fluctuations} and references there-in). The following results analyzes the optimization problem~\eqref{eq:gibbsop} in the particular setting (which corresponds to setting $H=K_2$).

\begin{lmm}\label{lem:fixsol} 
Let $\mu$ be a probability measure symmetric about $0$, which satisfies~\eqref{eq:tailp} with $p\geq 2$, and let $\alpha(\cdot)$, $\beta(\cdot)$, and $\mathcal{N}$ be as in~\cref{def:tilt}. Assume that
\begin{equation}\label{eq:secasn}
\alpha''(x)\leq \alpha''(y) \quad \mbox{ for all }\quad |x|\geq |y|.
\end{equation}
Then, setting
\begin{equation*}
\mv_{\theta,B,\mu}(x):=\theta x^2+Bx-x \beta(x)+\alpha(\beta(x))
\end{equation*}
for $\theta\geq 0$, $B\in\R, x\in {\rm cl}(\cN)$, 
 the following conclusions hold:
\begin{enumerate}
\item[(i)] If $2\theta\leq (\alpha''(0))^{-1}$ and $B=0$, then $\mv_{\theta,B,\mu}(\cdot)$ has a unique maximizer at $\tm=0$.
\item[(ii)] If $B\neq 0$, then $\mv_{\theta,B,\mu}(\cdot)$ has a unique maximizer $\tm$ with the same sign as that of $B$ which satisfies $\tm=\alpha'(2\theta \tm+B)$.
\item[(iii)] If $2\theta> (\alpha''(0))^{-1}$ and $B=0$, then $\mv_{\theta,B,\mu}(\cdot)$ has two maximizers $\pm \tm$, where $\tm>0$, and $\tm=\alpha'(2\theta \tm)$.
\end{enumerate} 
\end{lmm}
\noindent The proof of~\cref{lem:fixsol} is provided in~\cref{sec:exrespf}.
\begin{remark}\label{rem:monotonejust}
A few comments about the extra condition~\eqref{eq:secasn} utilized in the above lemma are in order. 
First note that that if any measure $\mu$ satisfies the celebrated GHS inequality of statistical physics (see \cite{Ellis1976}), then $\mu$ must satisfy \eqref{eq:secasn}. Indeed, taking the matrix $J$ in \cite[(1.2)]{Ellis1976} to be the ${\bf 0}$ matrix, it follows on applying the GHS inequality (\cite[(1.4)]{Ellis1976} that $\alpha'''(\theta)\le 0$ for $\theta\ge 0$, which immediately implies \eqref{eq:secasn}.  Sufficient conditions on $\mu$ for the GHS inequality (and hence~\eqref{eq:secasn}) can be found in~\cite[Theorem 1.2]{Ellis1976}. In~\cite[(1.5)]{Ellis1976} the authors give a counterexample where GHS inequality fails. Using the same example, it is not hard to show that \eqref{eq:secasn} fails in this case as well, and
%It is easy to construct examples of $\mu$ for which \eqref{eq:secasn} does not hold and 
$\mv(\cdot)$ does not have a unique maximizer for $B=0$ and some $\theta\leq (\alpha''(0))^{-1}/2$. %see e.g.,~\cite[Equation 1.5]{Ellis1976}.  (see~\cite{Ginibre1970,Griffiths1970,Newman1975}). Sufficient conditions on $\mu$ for~\eqref{eq:secasn} to hold can be seen in~\cite[Theorem 1.2]{Ellis1976} --- ND.}  %consequence of the celebrated \textit{GHS inequality} . }
\end{remark}

\begin{prop}\label{prop:isingdom}
Suppose that the measure $\mu$ satisfies the assumptions of \cref{lem:fixsol}, and $\{Q_n\}_{n\ge 1}$ satisfy \eqref{eq:cut_con} and \eqref{eq:q} with $H=K_2$ (i.e.~$\Delta=1$).
Also assume that the limiting graphon $W$ is strictly positive a.s., and satisfies $\int_0^1 W(.,y)dy=1$ a.s.~With $\mathcal{L}_p$ as in~\cref{def:tilt2}, define the functional $G_W:\mathcal{L}_p\mapsto \R$ by setting
\begin{align}\label{eq:quadG}
G_W(f):=\int_{[0,1]^2}W(x,y) f(x)f(y)dx dy.
\end{align}
Then $G_W$ is well defined by~\cref{prop:freen} part (i), and further, 
%under $\isR$ (as in~\eqref{eq:ising}), 
the following conclusions hold:

%(i) For all $\theta\ge 0$ we have
%$$\lim_{n\to\infty} \isZ_n(\theta,B)=\theta \tm^2 + B\tm - \tm \beta(\tm) + \alpha(\beta(\tm))=: \isZ(\theta,B).$$

(i) With $\isZ_n(\theta,B)$ as in \eqref{eq:ising}, we have 
\begin{align}\label{eq:opt_uncon}
\lim\limits_{n \rightarrow \infty} \isZ_n(\theta,B) = \sup_{f\in \mathcal{L}_p}\{\theta G_W(f) +B \int_{[0,1]} f(x) dx-\int_{[0,1]}\gamma(\beta(f(x)))dx\}
\end{align}

(ii) For any $\theta\ge 0$, $B\in\R$ the optimization problem \eqref{eq:opt_uncon} has only constant global maximizers, given  by
$${F}_{\theta,B}\equiv \begin{cases} 0&\text{ if }\theta \le (\alpha''(0))^{-1}/{2},B=0,\\
\pm {\tm}&\text{ if }\theta >(\alpha''(0))^{-1}/{2},B=0,\\
{\tm}&\text{ if }B\neq 0\end{cases}$$
%\begin{eqnarray*}
%\bar{F}\equiv &0&\text{ if }\theta \le (\alpha''(0))^{-1}/{2},B=0,\\
%\bar{F}\equiv& \pm {\tm}&\text{ if }\theta >(\alpha''(0))^{-1}/{2},B=0,\\
%\bar{F}_{\theta}\equiv & {\tm}&\text{ if }B\neq 0.
%\end{eqnarray*}
Here $\tm$ is as in~\cref{lem:fixsol}.
%$\mathfrak{L}_n$ satisfies a LDP with the good rate function \begin{align}\label{eq:proverate}I_{\theta,B}(\nu)=D(\nu|\rho)-\theta \E_{\nu}[W(U_1,U_2)V_1 V_2]+\isZ(\theta,B).\end{align}

(iii) 
\begin{comment}With $\mu_{\theta}$ as in~\cref{def:tilt} we have 
\begin{align*}
d_{\ell}(\mathfrak{L}_n,\mathrm{Unif}[0,1]\otimes \mu_{0})\overset{P}{\longrightarrow}0 \qquad & \mbox{if}\quad \theta\leq (\alpha''(0))^{-1}/2, B=0,\\ d_{\ell}\left(\mathfrak{L}_n,\mathrm{Unif}[0,1]\otimes \left(\frac{1}{2}\mu_{\beta(\tm)}+\frac{1}{2}\mu_{-\beta(\tm)}\right)\right)\overset{P}{\longrightarrow}0 \qquad & \mbox{if}\quad \theta > (\alpha''(0))^{-1}/2, B=0,\\ d_{\ell}(\mathfrak{L}_n,\mathrm{Unif}[0,1]\otimes\mu_{\beta(\tm)})\overset{P}{\longrightarrow}0 \qquad & \mbox{if}\quad B\neq 0.
\end{align*}
\end{comment}
The following weak limits hold under $\isR$ (as in \eqref{eq:ising}):
$$\frac{1}{n}\sum_{i=1}^n X_im_i\overset{d}{\longrightarrow} 2\tm^2, \qquad \bar{X}\overset{d}{\longrightarrow} \begin{cases} 0&\text{ if }\theta \le (\alpha''(0))^{-1}/{2},B=0,\\
\frac{\delta_{\tm}+\delta_{-\tm}}{2}&\text{ if }\theta >(\alpha''(0))^{-1}/{2},B=0,\\
{\tm}&\text{ if }B\neq 0\end{cases}.$$
%where $\nu^*_{\tm}$ is the equally weighted mixture of the distributions $\mathrm{Unif}[0,1]\otimes \nu_{\beta(\tm)}$ and $\mathrm{Unif}[0,1]\otimes \nu_{-\beta(\tm)}$. 
%\\

%(iii) $\isZ(.,B)$ is differentiable on the set $(0,\infty)$ for any $B\in \R$, and $\frac{d\isZ(\theta,B)}{d\theta}=\tm^2$.

\end{prop}
\begin{remark}\label{rem:symmconst}
We note here that in contrast to ~\cref{prop:higherord} part (iii) which only allows us to identify possible limit points for the random variable $\frac{1}{n}\sum_{i=1}^nX_im_i$, in this case we are able to identify the limit, even in the low temperature regime $\theta>(\alpha''(0))^{-1}/2$. This is because, even though  $\mv_{\theta,B,\mu}(.)$ has two roots which are both global optimzers (i.e.~in $\mathcal{A}_\theta$ defined in~\cref{prop:higherord}), the optimizers are symmetric, and the limit of the Hamiltonian is the same under both optimizers.
%This is because in the special case when $v=2$, we are able to use symmetry of the optimizers to get a better limiting statement.
\end{remark}

\subsubsection{Gibbs measures with higher order interactions} We now focus on Gibbs measures with higher order interactions, which has gained significant attention in recent years (see~\cite{mukherjee2020estimation,barra2009notes,heringa1989phase,liu2019ising,suzuki1971zeros,turban2016one,yamashiro2019dynamics,bhattacharya2024ldp} and the references therein). Here, we analyze the optimization problem~\eqref{eq:defset}, under some conditions on $\theta$ and $\mu$. We point the reader to~\cite[Section 2.1]{bhattacharya2020second} for related results in the special case where $\mu$ is supported on $\{-1,1\}$. %In in part (iii) of the following Proposition, we demonstrate existence of a \enquote{phase transition} in the parameter $\theta$

\begin{thm}
\label{prop:genopt}
          Consider the optimization problem 
          \begin{align}\label{eq:with_h_opt}
          \sup_{f\in \mathcal{L}_p:\ \int_{[0,1]} \gamma(\beta(f(x)))dx<\infty}\left\{\theta G_{W}(f) +B \int f(x) dx-\int_{[0,1]}\gamma(\beta(f(x)))dx\right\}.\end{align}
          Then the following conclusions hold:
          
          (i) Fixing $\theta\in \R$, the maximizers of the optimization problem are attained and satisfies the equation
         \begin{align}\label{eq:higher_fixed_point}
             f(x)\stackrel{a.s.}{=}\alpha'\left(\theta v\int_{[0,1]^{v-1}} \mathrm{Sym}[W](x,x_2,\ldots ,x_v)\left(\prod_{a=2}^v f(x_a)\,dx_a\right) + B \right).
         \end{align} 
(ii) If $\mathcal{T}[\mathrm{Sym}[W]](\cdot)$ is constant a.s., and $W$ is strictly positive a.s., and $\theta, B \ge 0$, % a.e. $(x,y)\in [0,1]^2$, 
then all of the maximizers are constant functions, provided either $v$ is even or $\mu$ is
stochastically non-negative. Further any such constant maximizer $t$ satisfies
\begin{equation}\label{eq:propoptlat} 
	t\stackrel{a.s.}{=}\alpha'\left(\theta v t^{v-1}+B\right).%\quad \mbox{a.e.}\ x\in [0,1].
\end{equation}

(iii) Suppose further that $\mu$ is compactly supported on $[-1,1]$ %nd satisfies~\eqref{eq:secasn}. 
Then the following hold:

\begin{enumerate}
    \item[(a)] There exists $B_0=B_0(\theta,v)$ such that if $B > B_0$, the optimization problem has a unique maximizer.
    \item[(b)] If $B=0$ and $\alpha'(0)=0$, there exists $\theta_c\in (0,\infty)$ such that if $\theta<\theta_c$, the optimization problem has the unique maximizer $t=0$, whereas if $\theta > \theta_c$, then $t=0$ is not a global maximizer. 
    % \item[(c)]
    % If $B=0$,  there exists $\widetilde{\theta}_c\ge \theta_c$ such that if $\theta> \widetilde{\theta}_c$, the optimization problem has a unique positive maximizer.
\end{enumerate}
      \end{thm}

\subsection{Connections to previous literature}
In \cref{prop:freen}, part (ii), we have shown that if $W_{Q_n}$ converges to $W$ in the cut norm (see \eqref{eq:cut_con}), then the Mean-Field approximation holds under certain integrability assumptions (see \eqref{eq:tailp} and \eqref{eq:q}). Similar results have been studied extensively in the literature; see e.g.~\cite{lacker2022mean,augeri2021transportation,eldan2018gaussian,Chatterjee2016,yan2020nonlinear,cook2024regularity,Austin2019,chatterjee2011large} and the references therein. However most of the existing result focuses on compactly supported $\mu$. To the best of our knowledge, only \cite{lacker2022mean} allows for non-compactly supported $\mu$ but imposes a log concavity restriction. On the other hand, we require neither compact support nor log concavity of $\mu$. Results similar to \cref{prop:propopt} on replica symmetry have also been studied in the literature under more stringent assumptions on $\mu$ and the Hamiltonian (see e.g.~\cite{mukherjee2020replica,lubetzky2015replica,dembo2014replica}). The concept of local fields as introduced in \eqref{eq:mdeff} plays an important role in obtaining correlation bounds, tail decay estimates and fluctuations for Gibbs measures (see~\cite{gheissari2019ising,Chatterjee2007,deb2020fluctuations,deb2020detecting,bhattacharya2025sharp}). We believe that our weak law for the local fields (see \cref{lem:wealimi}) and the resulting universal weak laws under regularity (see \cref{prop:higherord}) can have wide implications in the context of studying consistency of popularly studied estimators (such as maximum likelihood and maximum pseudo-likelihood) for the temperature parameter $\theta$ across a large class of models.

\subsection{Proof overview and future scope} Let us discuss the proof techniques employed in the characterization of replica symmetry (see \cref{prop:propopt}) and weak laws (see \cref{lem:wealimi} and \cref{prop:higherord}). In \cref{prop:propopt} part (i), we
establish the first-order conditions (in \eqref{eq:propoptshow}) for the optimization problem \eqref{eq:gibbsop}. It is immediate from \eqref{eq:propoptshow} that if any optimizer of \eqref{eq:gibbsop} is constant, then  $\mathcal{T}[\mathrm{Sym}[W]](\cdot)$ is a constant function (\cref{prop:propopt}, part (ii)). For the other direction, if $\mathcal{T}[\mathrm{Sym}[W]](\cdot)$ is a constant,  the crucial observation is that $\mathrm{Sym}[W](x_1,\ldots ,x_v)$ is a (possibly un-normalized) probability density function on $[0,1]^v$,  with $\mathrm{Unif}[0,1]$ marginals. %This allows us to rewrite the objective function in \eqref{eq:gibbsop}, at some function $f\in\mathcal{L}_p$, as follows:
%$$\E_{(Z_1,\ldots ,Z_v)\sim \mathrm{Sym}[W]}\left[\prod_{a=1}^v f(Z_a)\right]-\E_{Z_1\sim\mathrm{Unif}[0,1]}[\gamma(\beta(f(Z_1)))].$$
%The above interpretation suggests some natural applications of H\"{o}lder's inequality. 
The conclusion in \cref{prop:propopt} part (iii) then follows from the equality conditions of H\"{o}lder's inequality. We provide examples to demonstrate that our conditions required for replica symmetry are essentially tight.

The weak limits we prove involve a number of technical steps. We distil some of the main ideas here in the context of the universal result that $n^{-1}\sum_{i=1}^n c_iX_i\overset{P}{\longrightarrow} 0$ provided $\sum_{i=1}^n c_i=o(n)$, under any multilinear Gibbs in the replica symmetry phase (see \cref{prop:higherord} part (ii)). For simplicity, let us assume that the optimization problem \eqref{eq:gibbsop} has a unique constant optimizer, say $f(x)\equiv t$ (note that the actual result does not require uniqueness of optimizers). Let us split the proof outline into a few steps.

Step (i). Recall the definition of $m_i$ from \eqref{eq:mdeff}. We first show that 
$$\frac{1}{n}\sum_{i=1}^n c_i(X_i-\E[X_i|X_j,\ j\neq i])=\frac{1}{n}\sum_{i=1}^n c_i(X_i-\alpha'(\theta m_i))=o_P(1).$$
This is the subject of \cref{lem:auxtail} part (a), and proceeds with a second moment argument, after a suitable truncation. The above display now suggests that it is sufficient to show that $n^{-1}\sum_{i=1}^n c_i\alpha'(\theta m_i)\overset{P}{\longrightarrow}0$.

Step (ii). Based on step (i), it is natural to focus on the vector of local fields ${\bf m}=(m_1,\ldots,m_n)$. The advantage of working with ${\bf m}$ rather than ${\bf X}$ is that each $m_i$ is a $(v-1)$-th order ``weighted average", and hence they are much more ``concentrated" than $X_i$'s. We provide a formalization in \cref{lem:wealimi} where (in the current setting) we show that 
\begin{equation}\label{eq:prooftechonly}
\frac{1}{n}\sum_{i=1}^n \delta_{m_i}\overset{d}{\longrightarrow} \delta_{v t^{v-1}},
\end{equation}
which is a degenerate limit. In contrast, by \cref{prop:freen} part (iii), $\frac{1}{n}\sum_{i=1}^n \delta_{X_i}\overset{d}{\longrightarrow} \mu_{\beta(t)}$ which is a non-degenerate limit. The proof of \cref{lem:wealimi} relies primarily on \cref{lem:pivotlem}, which is a stability lemma, the proof of which proceeds relies on counting lemma for $L^p$ graphons. In fact, \cref{lem:pivotlem} can be viewed as a refinement of the counting lemma in \cite[Proposition 2.19]{bc_lpi} for ``star-like" graphs.

Step (iii). Based on \eqref{eq:prooftechonly} in step (ii), it is natural to consider the following approximation:
$$\frac{1}{n}\sum_{i=1}^n c_i\alpha'(\theta m_i)\approx \alpha'(\theta v t^{v-1})\frac{1}{n}\sum_{i=1}^n c_i = o(1).$$
In the first approximation, we have essentially replaced the $m_i$'s by the corresponding weak limit from \eqref{eq:prooftechonly}. To make this rigorous, we need some moment estimates which are immediate byproducts of the exponential tail bounds in \cref{lem:basicprop}. The final conclusion in the above display uses the condition $\sum_{i=1}^n c_i=o(n)$. 

\vspace{0.1in}

\noindent More generally, the weak limit of $
{\bf m}$ in \cref{lem:wealimi} has broad applications. We use it in \cref{cor:conmeanint} to provide a probabilistic  interpretation of the optimizers of \eqref{eq:gibbsop} (note that this does not require the optimizers to be constant functions). We also use \cref{lem:wealimi} to derive other weak laws of interest in \cref{prop:higherord}. We note in passing that many other statistics of interest, such as the maximum likelihood or the pseudo-maximum likelihood estimators for the inverse temperature parameter are also expressible (sometimes implicitly) as functions of ${\bf m}$. Consequently one can derive appropriate weak laws for these estimators using \cref{lem:wealimi} as well.
\\

%In this article, we provide an exact characterization for replica-symmetry for a class of multilinear Gibbs measures, in \cref{prop:propopt}. The proof involves several applications of H\"older's inequality and the sufficient condition is obtained from the equality conditions for these inequalities. Next, we derive in \cref{lem:Tgraphon0} some bounds regarding $L^p$ graphons which we invoke to obtain weak limits for
%Hamiltonian and local and global magnetizations. Moreover, under \eqref{eq:tailp} and \eqref{eq:q}, we invoke \cref{lem:Tgraphon0} again to obtain concentration and tail bounds for the random variables itself and magnetizations in \cref{lem:basicprop}. Next, we apply the weak limits and tail bounds to analyse Gibbs measure for quadratic interaction models with symmetric base measure and then higher order interaction models. The first example is a generalizion of the Ising models under Assumption \eqref{eq:secasn}. Under either examples, we exhibit phase transition.

Our work leads to several important future research directions. Our results apply, as a special case, to Ising models with quadratic Hamiltonians, and a general base measure. A first question is to extend the techniques of this paper to cover more general Hamiltonians from statistical physics, such as Potts models. Another related question is to go beyond the setting of cut norm convergence, and allow for the matrix $\{Q_n\}_{n\ge 1}$ to converge in other topologies (such as local weak convergence on bounded degree graphs). A third question is to study Gibbs measure under more general tensor Hamiltonians, which cannot be specified by a matrix $Q_n$. This would require significant development of cut norm theory for cubic and higher order functions. A starting point in this direction are the related works of \cite{eldan2018gaussian,augeri2021transportation,Chatterjee2016,cook2024regularity}, which focus on similar Mean-Field approximations for general tensors in the compact case, but provide sufficient conditions in terms of complexity bounds which are not always easy to bound. %Another interesting problem is to study multilinear forms in the case when $Q_n$ converges in a mode different from cut metric (such as graphs
%converging in local weak topology). One might also study the behavior in the context of other types of interactions, like Potts model. We plan to focus on this in an upcoming draft in more detail. 
Finally, it remains to be seen whether we can answer more delicate questions about such Gibbs measures, which include Central Limit Theorems/limit distributions.
 %In the current work, we derive weak limits for a large family of statistics, however, derivation of asymptotic distributions or establishing central limit theorems for such and other similar statistics remain unsolved. Finally, from a statistical physics perspective, it is important to obtain critical values of phase transition for multilinear Gibbs measure.}

\subsection{Outline of the paper} 
In Sections \ref{sec:pfmain} and \ref{sec:exrespf}, we prove our main results from Sections \ref{sec:mainres} and \ref{sec:isingpotts} respectively. The proofs of the major technical lemmas (in the order in which they are presented in the paper) are provided in \cref{sec:fifth}. In the Appendix \ref{sec:appen}, we defer the proof of some of our supporting lemmas, which deal with properties of the base measure $\mu$, and general results on weak convergence.

\section{Proof of Main Results}\label{sec:pfmain}

\subsection{Proofs of \cref{prop:freen} and \cref{prop:propopt}}

In order to prove \cref{prop:freen}, we need the following preparatory result. 

\begin{prop}\label{lem:Tgraphon0}
Fix any $v\geq 2$, $p\ge 1$, $q>1$ such that $\frac{1}{p}+\frac{1}{q}\leq 1$ and $W\in \mathcal{W}$. Fix any probability measure $\nu$ supported on $[0,1]\times\R$ with first marginal $\mathrm{Unif}[0,1]$ and sample $(U_1,V_1),\ldots, (U_v,V_v)\overset{i.i.d.}{\sim}\nu$. Then the following conclusions hold:
\begin{enumerate}

\item[(i)] We have:
\begin{align*}%\label{eq:tocall2}
\E\left(\prod_{(a,b)\in E(H)} |W(U_a,U_b)|\right)\le \|W\|_{ \Delta}^{|E(H)|}.
\end{align*}

    \item[(ii)] For any measurable $\phi:\R^v\to\R$  we have
\begin{align*}%\label{eq:tocall1}
\E\left[\left(\prod_{(a,b)\in E(H)} |W(U_a,U_b)|\right)|\phi(V_1,\ldots ,V_v)|\right]\le \|W\|_{q \Delta}^{|E(H)|} \Big(\E |\phi(V_1,\ldots,V_v)|^p\Big)^{\frac{1}{p}}.
\end{align*}

\item[(iii)]
With $\mathrm{Sym}[.]$  as in~\cref{def:symmfndef}, we have
\begin{align*}
     \E \Big[\mathrm{Sym}[|W|](U_1,\ldots,U_v)^q\Big] \le 
\|W\|_{q\Delta}^{q|E(H)|}.
\end{align*}

% With $\mathrm{Sym}[.]$  as in~\cref{def:symmfndef}, we have
% \begin{align*}	
% \E\left[\mathrm{Sym}[|W|](U_1,\ldots ,U_v)|\phi(V_1,\ldots ,V_v)|\right]\le \|W\|_{q \Delta}^{|E(H)|} \Big(\E |\phi(V_1,\ldots,V_v)|^p\Big)^{\frac{1}{p}}.
% \end{align*}
\end{enumerate}
\end{prop}

Parts (i) and (ii) above follow from \cite[Proposition 2.19]{borgsdense1} and \cite[Lemma 2.2]{bhattacharya2024ldp} respectively. However, part (iii) is new and a proof is provided in \cref{sec:appenaux}. While the proof of~\cref{prop:freen} only uses \cref{lem:Tgraphon0} part (ii), the other parts of \cref{lem:Tgraphon0} will be useful in the rest of the paper.

\begin{remark}\label{rem:twph} 
When the RHS of the display in part (ii) of~\cref{lem:Tgraphon0} is finite, we can define
\begin{align*}
T_{W,\phi}(\nu):=\E\left[\left(\prod_{(a,b)\in E(H)} W(U_a,U_b)\right)\phi(V_1,\ldots ,V_v)\right].
\end{align*}
\end{remark}

\begin{proof}[Proof of~\cref{prop:freen}]
Under the conditions of \cref{prop:freen},   \cref{lem:Tgraphon0} part (ii) implies that $T_{W,\phi}(\nu)$ is  well-defined and finite.

\emph{(i)} This is pointed out in \cite[Definition 1.6]{bhattacharya2024ldp} by invoking \cite[Lemma 2.2]{bhattacharya2024ldp}.

\emph{(ii), (iii)} 
These are restatements of parts (i) and (ii) of~\cite[Theorem 1.6]{bhattacharya2024ldp}. The fact that $\sup_{n\ge 1} Z_n(\theta) < \infty$ follows from the proof of \cite[Corollary 1.3]{bhattacharya2024ldp}. To prove $\Xi(F_{\theta})$ is compact, we invoke~\cite[Remark 2.1]{bhattacharya2024ldp} to note that
$$
\Xi(F_\theta)=\arg\min_{\nu\in \wmm} J(\nu),$$
where the function $J(.)$ defined by
$$J(\nu):=D(\nu|\rho)-\theta T_{W,\phi}(\nu)$$ with $\phi(x_1,\ldots ,x_v)=\prod_{a=1}^v x_a$ has compact level sets  
(by~\cite[Corollary 1.3, part (ii)]{bhattacharya2024ldp}), and $\wmm$ is a closed subset of probability measures.

\end{proof}

\noindent Next, we state an elementary property of $\gamma(\cdot)$ that will be useful in proving~\cref{prop:propopt} below. A short proof is provided in~\cref{sec:appen}.
\begin{lmm}\label{lem:KLcont}
The function $\gamma\circ \beta(\cdot): cl(\mathcal{N})\to [0,\infty]$ is a continuous (possibly extended) real-valued function.
\end{lmm}

\begin{proof}[Proof of~\cref{prop:propopt}]
\emph{ (i)} By switching the variables of integration, it is easy to check that the optimization problem~\eqref{eq:gibbsop} is equivalent to maximizing the function
			$$\mathcal{G}_{W}(f):=\theta\int_{[0,1]^v} \mathrm{Sym}[W](x_1,\ldots ,x_v)\left(\prod_{a=1}^v f(x_a)\,d x_a\right)-\int_{[0,1]}\gamma(\beta(f(x)))dx.$$
			Note that for all $\varepsilon\in[0,1]$, $g\in\mathcal{L}_p$ and $f\in F_\theta\subseteq \mathcal{L}_p$, the function $f+\varepsilon(f-g)=(1-\lambda)f+\varepsilon g\in\mathcal{L}_p$, and so %Observe that, by the definition of $\tilde{F}_{\theta}$, 
			$\mathcal{G}_{W}(f+\varepsilon (g-f))\leq \mathcal{G}_{W}(f)$. This gives
			\begin{equation*}%\label{eq:propopt1}
				\frac{d}{d\varepsilon}\mathcal{G}_{W}(f+\varepsilon (g-f))\bigg|_{\varepsilon=0}\leq 0,
			\end{equation*}
			which is equivalent to
			 \begin{align}\label{eq:propopt2}
				\footnotesize{\int_{[0,1]}\Big((g(x_1)-f(x_1))\underbrace{\left(\beta(f(x_1))-\theta v \int_{[0,1]^{v-1}} \mathrm{Sym}[W](x_1,\ldots ,x_v)\left(\prod_{a=2}^v f(x_a)\,dx_a\right)\right)}_{\delta(x_1)}\,dx_1\Big)\geq 0.}	`
			\end{align}
		We will show that $\lambda(\{x_1\in [0,1]:\ \delta(x_1)\neq 0\})=0$, where $\lambda$ denotes the Lebesgue measure on $\mathbb{R}$. Let us assume the contrary. Without loss of generality, assume that $\lambda(\{x_1\in [0,1]:\ \delta(x_1)> 0\})>0$. On this set, we have
			$$f(x_1)> \alpha'\left(\theta v \int_{[0,1]^{v-1}} \mathrm{Sym}[W](x_1,\ldots ,x_v)\left(\prod_{a=2}^v f(x_a)\,dx_a\right)\right)=:v(x_1),$$
			yielding
			$$\lambda(\{x_1\in [0,1]:\ \delta(x_1)> 0, f(x_1)>v(x_1)\})>0.$$
			This implies that there exists $\varepsilon>0$ such that $$\lambda(\mathcal{A}_\varepsilon)>0,\quad \mathcal{A}_\varepsilon:=\{x_1\in [0,1]:\ \delta(x_1)> \epsilon, f(x_1)>v(x_1)+\varepsilon\})>0.$$
			%If $\inf{\mathcal{N}}=-\infty$
			Define a function $g:[0,1]\mapsto {\rm cl}(\mathcal{N})$ by setting
			$$g(x_1):=\begin{cases} f(x_1)-\varepsilon & \mbox{if}\ x_1\in \mathcal{A}_\epsilon,\\ f(x_1) & \mbox{otherwise}.\end{cases}$$
			Note that $g\in\mathcal{L}_p$, as $f \in \mathcal{L}_p$, and	
		 $\int (g(x_1)-f(x_1))\delta(x_1) dx_1 <0$, contradicting \eqref{eq:propopt2}. This shows that $f(x_1)=v(x_1)$ a.s., as desired.
\\

				\emph{ (ii)} We will prove the contrapositive. Suppose there exists an almost surely constant function $f\in F_{\theta}$, say $f(x)=c\neq 0$ for a.e. $x\in [0,1]$. Then by~\eqref{eq:propoptshow}, we have $c=\alpha'(\theta c^{v-1}\mathcal{T}[\mathrm{Sym}[W]](x))$ for a.e. $x\in [0,1]$. This implies $\mathcal{T}[\mathrm{Sym}[W]](\cdot)=\frac{\beta(c)}{\theta c^{v-1}}$ is constant almost surely, which is a contradiction. 
			\newline 
			
			\emph{ (iii)} Without loss of generality, assume that $\mathcal{T}[\mathrm{Sym}[W]](x)=1$ for a.e. $x\in [0,1]$. Then $\mathrm{Sym}[W](x_1,\ldots ,x_v)$ is a probability density function on $[0,1]^v$ with all marginals uniformly distributed on $[0,1]$. By an application of H\"older's inequality with respect to the probability measure induced by $\mathrm{Sym}[W]$, we then have
			\begin{align*}
			    \mathcal{G}_{W}(f)=\E_{(Z_1,\ldots ,Z_v)\sim \mathrm{Sym}[W]}\bigg[\theta \prod_{a=1}^v f(Z_a)\bigg]\leq \int_{[0,1]} \theta |f(x)|^v\,dx.
			\end{align*}
			Consequently, it holds that
			\begin{align}\label{eq:con?}
				\sup_{f\in\mathcal{L}_p} \mathcal{G}_{W}(f)&\leq \sup_{f\in\mathcal{L}_p}\left\{\int_{0}^1 [\theta |f(x)|^v-\gamma(\beta(f(x)))]\,dx\right\}\leq \sup_{t\in {\rm cl}(\mathcal{N})} \{\theta |t|^v-\gamma(\beta(t))\}.
			\end{align}
			(a) If $v$ is even, then \eqref{eq:con?} gives
			$$\sup_{f\in\mathcal{L}_p} \mathcal{G}_{W}(f)\leq \sup_{t\in {\rm cl}(\mathcal{N})} \{\theta t^v-\gamma(\beta(t))\}.$$
			Equality holds in the above display by taking $f$ to be constant functions. To find out the maximizing $f$, we need
			equality in H\"older's inequality. So $f$ must be a constant function a.s.
			\\
			(b) If $\gamma(\beta(t))\le \gamma(\beta(-t))$ for all $t\in \mathcal{N}\cap[0,\infty)$, the same inequality continues to hold for all $t\in {\rm cl}(\mathcal{N})\cap (0,\infty) $ by~\cref{lem:KLcont}. Thus \eqref{eq:con?} gives
	
	$$\mathcal{G}_{W}(f)\le \sup_{\substack{t\in {\rm cl}(\mathcal{N}), t\geq 0}} \{\theta t^v-\gamma(\beta(t))\},$$
		%%	$$
		%		\sup_{f\in\mathcal{L}_p} \mathcal{G}_{1,W}(f) \leq \sup_{f\in\mathcal{L}_p}\left\{\int_{0}^1 [\theta |f|^v(x)-\gamma(\beta(f(x)))]\,dx\right\}\leq \sup_{t\in {\rm cl}(\mathcal{N})} \{\theta t^v-\gamma(\beta(t))\}.
		%	$$
			Again equality holds in the above display by taking supremum over constant functions, and the maximizing $f$ is again constant a.s..
\end{proof}				

\begin{proof}[Proof of~\cref{prop:mu_suff}]
    The result for (a) follows immediately on noting that $\gamma(\beta(t))$ is not even defined for $t<0$, and so the definition of stochastic non-negativity holds vacuously.  We thus focus on  proving (b). In this case there exists a symmetric measure $\mu^{(s)}$ such that $\mu=\mu^{(s)}_B$ (see \cref{def:tilt}). Fixing  $t> 0$ such that $-t\in \alpha'(\R)$, using symmetry of $\nu$ it follows that $t\in \alpha'(\R)$, and 
			$$\alpha(\theta)=\alpha_{\mu^{(s)}}(\theta+B)-\alpha_{\mu^{(s)}}(B),\text{ where }{\alpha}_{\mu^{(s)}}(\theta):=\log \int_\R e^{\theta x} d{\mu^{(s)}}(x)\text{ for all }\theta\in \R.$$
			Thus, with ${\beta}_{\mu^{(s)}}$ denoting the inverse of ${\alpha}_{\mu^{(s)}}$, we have ${\beta}_{\mu^{(s)}}(t)=\beta(t)+B$ for all $t\in \mathcal{N}_{\mu^{(s)}}$, where ${\mathcal{N}}_{\mu^{(s)}}$ is the natural parameter space of $\mu^{(s)}$.
			%which implies $\beta(t)=(C')^{-1}(t)-B$. 
			This gives
            \begin{small}
\begin{align}\label{eq:simplify}
\gamma(\beta(t))&=t{\beta}_{\mu^{(s)}}(t)-Bt-{\alpha}_{\mu^{(s)}}({\beta}_{\mu^{(s)}}(t)) + {\alpha}_{\mu^{(s)}}(B) \nonumber \\
&={\gamma}_{\mu^{(s)}}({\beta}_{\mu^{(s)}}(t))+{\alpha}_{\mu^{(s)}}(B)-Bt.
\end{align}
\end{small}
	%Note that the above equation holds without the symmetry of $\mu^{(s)}$. 
    \noindent As $\mu^{(s)}$ is symmetric about $0$, so ${\gamma}_{\mu^{s)}}(\cdot)$ and ${\beta}_{\mu^{(s)}}(\cdot)$ are even functions. The assumption $B\geq 0$ along with~\eqref{eq:simplify} gives $\gamma(\beta(t))\leq \gamma(\beta(-t))$ for $t\geq 0$, completing the proof.
\end{proof}

In the sequel, we will first prove \cref{lem:basicprop} independently. Then we will prove \cref{lem:wealimi} using \cref{lem:basicprop}. In order to prove \cref{lem:basicprop}, we need the following lemma whose proof we defer.
\begin{lmm}\label{lem:fixsol0}
					Suppose $\mu$ satisfies \eqref{eq:tailp} for some $p>1$. 
\newline				
     (i) Then with $\alpha(\cdot)$ as in~\cref{def:tilt} we have
	$$\lim_{\theta\to\pm\infty}\frac{\alpha'(\theta)}{|\theta|^{\frac{1}{p-1}}}=0.$$
					(ii)
					With $\beta(\cdot)$ as in~\cref{def:tilt} we have $$\lim_{x\to \{\inf\{\mathcal{N}\},\sup\{\mathcal{N}\}\}}\frac{\beta(x)}{|x|^{p-1}}=\infty.$$
					\end{lmm}

			\begin{proof}[Proof of \cref{lem:basicprop}]

\emph{(i)} 
The conclusion of \eqref{eq:lip11}  follows from the proof of \cite[Eq 2.28]{bhattacharya2024ldp} (see Page 25).
 Proceeding to verify \eqref{eq:lip13}, fix $i\in [n]$ and use H\"older's inequality to note that
   {\small\begin{align}\label{eq:pointwise}
   |m_i|&\le \frac{v}{n^{v-1}}\sum_{(i_2,\ldots ,i_v)\in [n]^{v-1}}|\tQh(i,i_2,\ldots ,i_v)|\prod_{a=2}^v |X_{i_a}|\nonumber \\
   &\le v \left(\frac{1}{n^{v-1}}\sum_{(i_2,\ldots ,i_v)\in [n]^{v-1}} |\tQh(i,i_2,\ldots ,i_v)|^q\right)^{\frac{1}{q}}\left(\frac{1}{n^{v-1}}\sum_{(i_2,\ldots ,i_v)\in [n]^{v-1}}\prod_{a=2}^v |X_{i_a}|^p \right)^{\frac{1}{p}}\nonumber \\
   &=v \left(\frac{1}{n^{v-1}}\sum_{(i_2,\ldots ,i_v)\in [n]^{v-1}} |\tQh(i,i_2,\ldots ,i_v)|^q\right)^{\frac{1}{q}}\left(\frac{1}{n}\sum_{j=1}^n |X_{j}|^p \right)^{\frac{v-1}{p}}.
   %&\le v (v!)^{q-1}\|W_{Q_n}\|_{q\Delta}^{|E(H)|}\left(\frac{1}{n}\sum_{j=1}^n |X_{j}|^p \right)^{\frac{v-1}{p}},
   \end{align}}
   %where the last inequality uses \eqref{eq:symmbd}.
   Raising both sides to the $q^{th}$ power and summing~\eqref{eq:pointwise} over $i\in [n]$ gives
   \begin{align}\label{eq:again0}\frac{1}{n}\sum_{i=1}^n |m_i|^q&\leq v^q \left(\frac{1}{n}\sum_{j=1}^n |X_j|^p\right)^{\frac{q(v-1)}{p}}\left(\frac{1}{n^v}\sum_{(i_1,\ldots ,i_v)\in [n]^v} |\tQh(i_1,i_2,\ldots ,i_v)|^q\right)\nonumber \\
   &\le v^q \left(\frac{1}{n}\sum_{j=1}^n |X_j|^p\right)^{\frac{q(v-1)}{p}}
\|W_{Q_n}\|_{q\Delta}^{q|E(H)|},
\end{align}
where the last inequality uses \cref{lem:Tgraphon0} part (c), with $W\equiv W_{Q_n}$. The conclusion then follows by \eqref{eq:q} and \eqref{eq:lip11}.

\vspace{0.05in}

Next, we will prove \eqref{eq:lip14}. By \cref{lem:fixsol0} part (i), there exists $c_{\mu}>0$ such that for all $\theta\in \R$ we have
\begin{align}\label{eq:again1}
|\alpha'(\theta)|\le c_{\mu} |\theta|^{\frac{1}{p-1}}.
\end{align}
Now, note the following chain of equalities/inequalities with explanations to follow.
   
 {\small  \begin{align*}
   |\alpha'(\theta m_i)| &=\bigg|\alpha'\left(\frac{\theta v}{n^{v-1}}\sum_{(i_2,\ldots ,i_v)\in \mathcal{S}(n,v,i)}|\tQh(i,i_2,\ldots ,i_v)|\prod_{a=2}^v |X_{i_a}|\right)\bigg|\\ &\le c_{\mu} \left(\frac{\theta v}{n^{v-1}}\sum_{(i_2,\ldots ,i_v)\in [n]^{v-1}}|\tQh(i,i_2,\ldots ,i_v)|\prod_{a=2}^v |X_{i_a}|\right)^{\frac{1}{p-1}}\\ &\le c_{\mu}(\theta v)^{\frac{1}{p-1}}\left(\frac{1}{n^{v-1}}\sum_{(i_2,\ldots ,i_v)\in [n]^{v-1}} |\tQh(i,i_2,\ldots ,i_v)|^q\right)^{\frac{1}{q(p-1)}}\left(\frac{1}{n}\sum_{j=1}^n |X_{j}|^p \right)^{\frac{v-1}{p(p-1)}}.
   \end{align*}}
 
   The first inequality follows directly from \eqref{eq:again1}. The second inequality follows from \eqref{eq:pointwise}. %The third display uses H\"{o}lder's inequality with exponents $p$ and $q$. 
   Raising both sides to the power $p$ and summing over $i\in [n]$, we get:
   \begin{align}\label{eq:again2}
    &\;\;\;\;\;\frac{1}{n}\sum_{i=1}^n |\alpha'(\theta m_i)|^p\nonumber \\ &\le \frac{1}{n}\sum_{i=1}^n c_{\mu}^p (\theta v)^{\frac{p}{p-1}}\left(\frac{1}{n^{v-1}}\sum_{(i_2,\ldots ,i_v)} |\tQh(i,i_2,\ldots ,i_v)|^q\right)^{\frac{p}{q(p-1)}}\left(\frac{1}{n}\sum_{j=1}^n |X_{j}|^p \right)^{\frac{v-1}{p-1}}\nonumber \\ &\le c_{\mu}^p (\theta v)^{\frac{p}{p-1}}\left(1+\lVert W_{Q_n}\rVert_{q\Delta}^{q|E(H)|}\right)\left(\frac{1}{n}\sum_{j=1}^n |X_{j}|^p \right)^{\frac{v-1}{p-1}}.
   \end{align}
   The final inequality follows by noting that $|x|^{\frac{p}{q(p-1)}}\le 1+|x|$ and then using \cref{lem:Tgraphon0} part (c), with $W\equiv W_{Q_n}$. The conclusion follows by \eqref{eq:q} and \eqref{eq:lip11}.
   
\vspace{0.05in}

\emph{(ii)} The proof of \eqref{eq:lip12} is very similar to the proof of (a). Firstly, $\sup_{\nu\in \Xi(F_{\theta})} \fmm_p(\nu)<\infty$ follows from \cite[Eq 2.29]{bhattacharya2024ldp}. This also implies:
\begin{equation}\label{eq:lip16}
\sup_{f\in F_{\theta}} \lVert f\rVert_p= \sup_{\nu\in \Xi(F_{\theta})} \lVert \E_{\nu} [V|U]\rVert_p\le \sup_{\nu\in \Xi(F_{\theta})} \E_{\nu}|V|^p=\sup_{\nu\in \Xi(F_{\theta})} \fmm_p(\nu)<\infty.
\end{equation}

Next, in the same vein as \eqref{eq:again0}, we get
$$\sup_{\nu\in\mathfrak{B}_{\theta}} \fmm_q(\nu)\le v^q \sup_{f\in F_{\theta}} \lVert f\rVert_p^{q(v-1)}\lVert W\rVert_{q\Delta}^{q|E(H)|}<\infty,$$
by invoking \eqref{eq:lip16} and \eqref{eq:W_q}, thereby proving the second conclusion. Finally, proceeding similar to \eqref{eq:again2}, 
%\eqref{eq:lip16} and \eqref{eq:W_q} 
we have
$$\sup_{\nu\in\mathfrak{B}^*_{\theta}} \fmm_p(\nu)\le c_{\mu}^p (\theta v)^{\frac{p}{p-1}}\left(1+\lVert W\rVert_{q\Delta}^{q|E(H)|}\right)\sup_{f\in F_{\theta}} \lVert f\rVert_p^{\frac{(v-1)p}{p-1}}<\infty,$$
where we have used \eqref{eq:lip16} and \eqref{eq:W_q}. This completes the proof of part (b).

 %Consequently, the compactness of $\mathfrak{B}_{\theta}$ follows from the following two observations:
   %\begin{itemize}

%\item[(a)] $\Xi(F_{\theta})$ is compact in the weak topology (which follows from \cref{prop:freen} part (iv)).
%\item[(b)] $\Upsilon(W,\cdot)$ is continuous in the weak topology over the space of probability measures with bounded $p$-th moments for the second marginal (this follows by combining parts (i) and (iii) of~\cref{lem:pivotlem}).

%\end{itemize} 

\end{proof}
       \subsection{Proof of \cref{lem:wealimi}}

\emph{(i)} The fact that $\Xi(F_{\theta})\subseteq \wmm_p$ follows directly from \eqref{eq:lip12}. %Plugging $\phi(x_1,\ldots ,x_v)=\prod_{a= 2}^v x_a$ in 
By an application of H\"older's inequality with \cref{lem:Tgraphon0} part (iii), we get:
$$\E\left[\mathrm{Sym}[|W|](U_1,U_2,\ldots ,U_v)\prod_{a=2}^v |V_a|\right]\le \lVert W\rVert_{q\Delta}^{|E(H)|}\left(\E|V_1|^p\right)^{\frac{v-1}{p}},$$
which is finite on using \eqref{eq:tailp} and \eqref{eq:W_q}.
By Fubini's Theorem, $\vartheta_{W,\nu}(.)$ (see~\eqref{eq:varthdef}) is well-defined a.s.~on $[0,1]$, as desired.

\begin{remark}
    Note that the above argument does not require $1/p+1/q<1$ but the weaker condition $1/p+1/q\le 1$.
\end{remark}
\vspace{0.05in} 

	\emph{(ii)} We begin the proof with the following definition.

			\begin{defn}\label{def:pivotlem}
			Let $\mathcal{W}$ and $\wmm_p$ be as in~\cref{def:defirst}~and~\cref{def:M} respectively. 
			Recall that $\fmm_p(\nu)=\int |x|^p\,d\nu_{(2)}(x)<\infty$ (from \cref{def:M}) for $\nu\in \wmm_p$. Define $$\mr:=\{(W,\nu),\ W\in \mathcal{W},\ \nu\in\wmm_p,\ \lVert W\rVert_{q\Delta}<\infty\}.$$ Construct the following function $\Upsilon:\mr\to\mathcal{M}$ (the space of probability measures on $[0,1]\times\R$) by setting:
%$$\Upsilon(W,\nu):=\mathrm{Law}\left(U_1,\underbrace{v\E_{\nu}\left[\mathrm{Sym}[W](U_1,\ldots ,U_v)\bigg(\prod_{j=2}^v V_j\bigg)\bigg|U_1\right]}_{\vartheta_{W,\nu}(U_1)}\right).$$
\begin{align}\label{eq:ups}\Upsilon(W,\nu):=\mathrm{Law}\left(U_1,\vartheta_{W,\nu}(U_1)\right).\end{align}
				Here $(U_1,V_1),\ldots ,(U_v,V_v)\overset{i.i.d.}{\sim}\nu$, %$\mathrm{Sym}[W]$ is as in~\cref{def:symmfndef},
    and $\vartheta_{W,\nu}(.)$ is as in \eqref{eq:varthdef}. Note that $\Upsilon(W,\nu)$ is well-defined for $(W,\nu)\in \mathcal{R}$, as the function $\vartheta_{W,\nu}(.)$ is well defined a.s.~by \cref{lem:wealimi} part (i). %in~\cref{lem:wealimi}. 
    Also for $L>0$ and a random variable $X$, set $X^{(L)}=X\mathbf{1}(|X|\le L)$. For any measure $\nu\in \widetilde{\mathcal{M}}$ and $(U,V)\sim\nu$, let $\nu^{(L)}$ denote the distribution of the truncated random variable $(U,V^{(L)})$.
    %$$V^{(L)}:=\begin{cases} V & \mbox{if}\ |V|\leq L,\\ -L & \mbox{if}\ V<-L,\\ L & \mbox{if}\ V>L.\end{cases}$$
			\end{defn}
%For any $L\in (0,\infty]$, note that 
		%$X_{i}^{(L)}=X_i {\bf 1}(|X_i|\le L)$. 
  Set 		\begin{equation}\label{eq:needagain}
	 m_{i,R}^{(L)}:=\frac{v}{n^{v-1}}\sum_{(i_2,\ldots ,i_v)\in [n]^{v-1}} \tQh\left(i,i_2,\ldots ,i_v\right)\left(\prod_{a=2}^v X_{i_a}^{(L)}\right).
				\end{equation}
				As a shorthand, we denote $m_{i,R}:=m_{i,R}^{\infty}$. Let us also define 
    $$\mathbf{X}^{(L)}:=\{X_1^{(L)},\ldots ,X_n^{(L)}\},\quad \mathbf{m}_R^{(L)}:=\{m_{1,R}^{(L)},\ldots ,m_{n,R}^{(L)}\}, \quad \mathbf{m}_R:=\{m_{1,R},\ldots ,m_{n,R}\}.$$
    Following \eqref{eq:ln}, we have:
    $$\tmln:=\frac{1}{n}\sum_{i=1}^n \delta_{\left(\frac{i}{n},m_{i,R}\right)}.$$
     
Next we generate $U\sim\mathrm{Unif}[0,1]$. We define a map $\widetilde{\mathfrak{L}}_n$ from $\R^n$ to  $\widetilde{\mathcal{M}}$ given by 
    \begin{equation}\label{eq:unifemp}
    \widetilde{\mathfrak{L}}_n(\mathbf{x}):=\mathrm{Law}(U,x_{\lceil n U\rceil}),\qquad \mathbf{x}=(x_1,\ldots ,x_n).
    \end{equation}
    The map $\widetilde{\mathfrak{L}}_n$ can be thought of as the continuous analogue of the discrete empirical measure map ${\mathfrak{L}}_n$.
    In view of \eqref{eq:unifemp}, note that $\tml$, $\tln$, $\tlnl$, and $\tlns$ denotes the laws of $(U,X_{\lceil nU\rceil})$, $(U,X_{\lceil nU\rceil}^{(L)})$, $(U,m_{\lceil nU\rceil, R}^{(L)})$, and $(U,m_{\lceil nU\rceil, R})$  conditioned on $X_1,\ldots ,X_n$, respectively. 
			    %It therefore suffices to obtain the weak limit of $\tlns$ for proving~\eqref{eq:toshownow}. Towards this end, 
			    Also, with $\Upsilon$ as in~\cref{def:pivotlem}, we have
			\begin{equation}\label{eq:highordd1}
		\Upsilon\left(W_{Q_n},\tml\right)=\tlns,\quad \Upsilon\left(W_{Q_n},\tln\right)=\tlnl.
				\end{equation}
			   In order to prove the above, note that,  given any bounded continuous real-valued function $f$ on $[0,1]\times \R$, we have:
		    \begin{align*}
			& \E_{\Upsilon\left(W_{Q_n},\tml\right)}[f]\\
			&=\int_0^1 f\left(u_1,v\int_{[0,1]^{v-1}} \mathrm{Sym}[Q_n](\lceil n u_1\rceil,\lceil n u_2\rceil,\ldots, \lceil n u_v\rceil)\left(\prod_{a=2}^v X_{\lceil nu_a\rceil}\right)\,du_2\ldots \,du_v\right)\,du_1 \\&=\sum_{i_1=1}^n \int_{\frac{i_1-1}{n}}^{\frac{i_1}{n}} f\left(u_1,\frac{v}{n^{v-1}}\sum_{(i_2,\ldots ,i_v)\in [n]^{v-1}} \mathrm{Sym}[Q_n](\lceil nu_1\rceil,i_2,\ldots ,i_v)\left(\prod_{a=2}^v X_{i_a}\right)\right)\,du_1\\ &=\sum_{i_1=1}^n \int_{\frac{i_1-1}{n}}^{\frac{i_1}{n}} f(u_1,m_{\lceil nu_1\rceil, R})\,du_1=\E_{\tlns}[f],
			    \end{align*}
			    and so the first conclusion of \eqref{eq:highordd1} holds. The proof of the second conclusion is similar. By the definition of  $\mathfrak{B}_{\theta}$ in~\cref{lem:wealimi}, we have
			    \begin{equation}\label{eq:highord2}
				\mathfrak{B}_{\theta}=\{\Upsilon(W,\nu):\ \nu\in \Xi(F_{\theta})\}=\Upsilon(W,\Xi(F_{\theta})).
				\end{equation}
			    
			    %where  Here~\eqref{eq:highord2} is immediate from the definition of $\mathfrak{B}$.  
			  With $\tfm$ as in~\eqref{eq:empmi}, triangle inequality gives
			    \begin{align*}%\label{eq:wealimit0}
       \notag&\;\;\;\;d_{\ell}(\tfm,\mathfrak{B}_{\theta})\\ \notag&\leq d_{\ell}(\tfm,\tmln)+d_{\ell}(\tmln,\tlns)+d_{\ell}(\tlns,\tlnl)\\ \notag&\quad +d_{\ell}(\tlnl,\Upsilon(W,\tln))+d_{\ell}(\Upsilon(W,\tln),\mathfrak{B}_{\theta})\\
     &=d_{\ell}(\tfm,\tmln)+d_{\ell}(\tmln,\tlns)+d_{\ell}(\Upsilon(W_{Q_n},\tml),\Upsilon(W_{Q_n},\tln))\\ &\notag\quad +d_{\ell}(\Upsilon(W_{Q_n},\tln),\Upsilon(W,\tln)+d_{\ell}(\Upsilon(W,\tln),\Upsilon(W,\Xi(F_{\theta})),
			    \end{align*}
			    where the second equality uses~\eqref{eq:highordd1} and~\eqref{eq:highord2}. We now show that each of the terms on the right hand side converge to $0$ as we take limits with $n\to\infty$ first, followed by $L\to\infty$. Towards this direction, we observe that:
		\begin{align}\label{eq:lipcon}
			d_{\ell}\left(\tmln,\tlns\right)&=\sup_{f\in \mathrm{Lip}(1)}\bigg|\frac{1}{n}\sum_{i=1}^n f\left(\frac{i}{n},m_{i,R}\right)-\sum_{i=1}^n \int_{\frac{i-1}{n}}^{\frac{i}{n}} f(u,m_{i,R})\,du\bigg|\nonumber \\ &\leq \sup_{f\in\mathrm{Lip}(1)}\sum_{i=1}^n \int_{\frac{i-1}{n}}^{\frac{i}{n}}\bigg|f\left(\frac{i}{n},m_{i,R}\right)-f(u,m_{i,R})\bigg|\,du\leq\frac{1}{n}\to 0.
				\end{align}

\noindent Based on the above two displays, it now suffices to prove the following: 
			\begin{equation}\label{eq:prooflater}	 d_{\ell}(\tfm,\tmln)\overset{P}{\longrightarrow}0,
				\end{equation}

    \begin{equation}\label{eq:prooflater1}	 d_{\ell}(\Upsilon(W_{Q_n},\tml),\Upsilon(W_{Q_n},\tln))\overset{P}{\longrightarrow}0,
			\end{equation}
as $n\to\infty$ followed by $L\to\infty$, and
\begin{equation}\label{eq:prooflater2}	 d_{\ell}(\Upsilon(W_{Q_n},\tln),\Upsilon(W,\tln))\overset{P}{\longrightarrow}0,
				\end{equation}

as $n\to\infty$ for every fixed $L>0$, and 

    \begin{equation}\label{eq:prooflater3}	 d_{\ell}(\Upsilon(W,\tln),\Upsilon(W,\Xi(F_{\theta}))\overset{P}{\longrightarrow}0,
				\end{equation}				
    as $n\to\infty$, followed by $L\to\infty$. 
    
   \noindent We now split the proof into four parts, proving the four preceding displays. We begin with the proof of \eqref{eq:prooflater} which requires the following lemma. It is a variant of \cite[Lemma 2.7]{bhattacharya2024ldp}. We omit the details of the proof for brevity.

\begin{lmm}\label{lmm:un_vn_same}
Suppose $Q_n$ satisfies~\eqref{eq:q} for some $q>1$. Let $\varphi:\R^{v-1}\to [-L,L]$ for some $L>0$, and $\mathcal{S}(n,v,i)$ be as in \cref{def:ocalmag}.
%from~\eqref{eq:mdeff}. 
Then given any permutation $\sigma$ of $[v]$, we get:
\begin{align*}
   \lim\limits_{n \rightarrow \infty} \frac{1}{n^v}\sup\limits_{\substack{(x_1,\ldots ,x_n)\\ \in \R^n}}\sum_{i_1=1}^n & \bigg\lvert   \sum_{\substack{(i_2,\ldots,i_v)\\ \in \mathcal{S}(n,v,i_1)}}\left(\prod_{(a,b)\in E(H)}Q_n(i_{\sigma(a)},i_{\sigma(b)})\right)\varphi(x_{i_2},\ldots ,x_{i_v})-\\ &\sum_{\substack{(i_2,\ldots,i_v)\\ \in [n]^{v-1}}}\left(\prod_{(a,b)\in E(H)}Q_n(i_{\sigma(a)},i_{\sigma(b)})\right)\varphi(x_{i_2},\ldots ,x_{i_v})\bigg\rvert= 0.
\end{align*}
\end{lmm}

\begin{proof}[Proof of \eqref{eq:prooflater}]

   With $\mathcal{S}(n,v,i)$ as in \cref{def:ocalmag}, define
			\begin{equation}\label{eq:mitrun}
			m_i^{(L)}:=\frac{v}{n^{v-1}} \sum_{(i_2,\ldots ,i_v)\in \mathcal{S}(n,v,i)} \tQh\left(i,i_2,\ldots ,i_v\right)\left(\prod_{a=2}^v X_{i_a}^{(L)}\right).
			\end{equation}
			It then suffices to prove the following:
			\begin{equation}\label{eq:show1}
			  \lim\limits_{L\to\infty}\limsup\limits_{n\to\infty}\P\left(n^{-1}\sum_{i=1}^n |m_i-m_i^{(L)}|\geq \epsilon\right)=0,
			\end{equation}
   \begin{equation}\label{eq:show3}
			  \lim\limits_{L\to\infty}\limsup\limits_{n\to\infty}\P\left(n^{-1}\sum_{i=1}^n |m_{i,R}-m_{i,R}^{(L)}|\geq \epsilon\right)=0,
			\end{equation}
			 for any $\epsilon>0$, and for any fixed $L>0$,
			\begin{equation}\label{eq:show2}
		     \frac{1}{n}\sum_{i\in [n]} |m_i^{(L)}-m_{i,R}^{(L)}|\xrightarrow{P} 0.
			\end{equation}
			
			\noindent\emph{Proof of~\eqref{eq:show1}} 
   % With $\tQh(\cdot)$ as in~\cref{def:symmfndef}, we have				\begin{align*}
 %   &\;\;\;\;\frac{1}{n^v}\sum_{(k_1,\ldots ,k_v)\in [n]^v} |\tQh(k_1,\ldots ,k_v)|^q\nonumber \\ &=\frac{1}{n^v}\sum_{(k_1,\ldots ,k_v)\in [n]^v} \bigg|\frac{1}{v!}\sum_{\sigma\in \mathcal{S}_v} \prod_{(a,b)\in E(H)} Q_n(k_{\sigma(a)},k_{\sigma(b)})\bigg|^q\nonumber \\ &\le \frac{1}{v!}\sum_{\sigma\in\mathcal{S}_v}\frac{1}{n^v}\sum_{(k_1,\ldots ,k_v)\in [n]^v} \prod_{(a,b)\in E(H)}|Q_n(k_{\sigma(a)},k_{\sigma(b)})|^q\le \lVert W_{Q_n}\rVert_{q\Delta}^{q|E(H)|}
 %   \end{align*}
 %   Here the two inequalities on the last line use H\"{o}lder's inequality, and~\cref{lem:Tgraphon0} part (a), respectively.
 % Check all applications of~\cref{lem:Tgraphon0}.
   Fix $\tp\in (1,p)$ such that $\tp^{-1}+q^{-1}<1$. For any set $A \subseteq \{2, \ldots, v\}$, define $A^c:= \{2, \ldots, v\} \setminus A$. %We will use $\lesssim$ to hide constants depending on $H$, $v$, $\tp$, $q$. 
   For any $L>1$ we have
			\begin{small}
			\begin{align}\label{eq:small}
		&\;\;\;\;\frac{1}{n}\sum_{i=1}^n |m_i-m_i^{(L)}|\nonumber\\ &\le \frac{v}{n^v}\sum_{A\subseteq \{2,\ldots ,v\},\ |A|\geq 1} \sum_{(i_1,\ldots ,i_v)\in [n]^v} \big|\tQh\left(i_1,i_2,\ldots ,i_v\right)\big|\left(\prod_{a\in A} |X_{i_a}-X_{i_a}^{(L)}|\right)\left(\prod_{a\in A^c} |X_{i_a}^{(L)}|\right)\nonumber \\
   &\le v\sum_{A\subseteq \{2,\ldots ,v\},\ |A|\geq 1} \left[\frac{1}{n^v}\sum_{(i_1,\ldots ,i_v)\in [n]^v} \big|\tQh\left(i_1,i_2,\ldots ,i_v\right)\big|^q\right]^{\frac{1}{q}} \nonumber\\
   &\; \; \; \; \; \;\; \; \; \; \; \;\; \; \; \;\;\; \left(\prod_{a\in A}\left(\frac{1}{n}\sum_{i_a=1}^n |X_{i_a}|^{\tp}\mathbf{1}(|X_{i_a}|> L)\right)\right)^{\frac{1}{\tp}}\left(\prod_{a\in A^c}\left(\frac{1}{n}\sum_{i_a=1}^n |X_{i_a}|^{\tp}\mathbf{1}(|X_{i_a}|\le  L)\right)\right)^{\frac{1}{\tp}}\nonumber \\
   & \le %vL^{\tp - p}\sum_{A\subseteq \{2,\ldots ,v\},\ |A|\geq 1} \lVert W_{Q_n}\rVert_{q\Delta}^{|E(H)|}\left(1+\frac{1}{n}\sum_{i=1}^n |X_i|^p\right)^{\frac{v-1}{\tp}}=
   v 2^{v-1} L^{\tp - p}\lVert W_{Q_n}\rVert_{q\Delta}^{|E(H)|}\left(1+\frac{1}{n}\sum_{i=1}^n |X_i|^p\right)^{\frac{v-1}{\tp}}.
			\end{align}
			\end{small}
   as $n\to\infty$, followed by $L\to\infty$. Here the second line uses~\eqref{eq:mitrun},  the third line uses H\"older's inequality, and the fourth inequality follows from%~\cref{lem:Tgraphon0} part (a) with $\nu\equiv\mathfrak{L}_n$, $W\equiv W_{Q_n}$, and 
   %$$\phi(x_1,\ldots ,x_v)=\left(\prod_{j\in A} |x_j|\mathbf{1}(|x_j|\ge L)\right)\left(\prod_{j\in A^c} |x_j|\mathbf{1}(|x_j|> L)\right),$$
~\cref{lem:Tgraphon0} part (iii) along with the inequalities $$|x|^{\tp}\mathbf{1}(|x|> L)\leq  L^{\tp-p}(1+|x|^p),\quad |x|^{\tp}\mathbf{1}(|x|\le L)\le 1+|x|^p.$$ The conclusion holds on noting that the RHS of \eqref{eq:small} converges to $0$ on letting $n\to\infty$ followed by $L\to\infty$,  since $\lVert W_{Q_n}\rVert_{q \Delta}=O(1)$ and $\fmm_p(\tml)=O_p(1)$  (which are direct consequences of~\eqref{eq:q} and ~\eqref{eq:lip11} respectively). This proves~\eqref{eq:show1}.
			
			\vspace{0.1in}

\noindent\emph{Proof of \eqref{eq:show3}.} The proof is same as that of \eqref{eq:show1}. We skip the details for brevity.

   \vspace{0.1in}
   
			\noindent\emph{Proof of~\eqref{eq:show2}.}
Using \eqref{eq:needagain} and \eqref{eq:mitrun}, observe that 
\begin{align*}
    &\frac{1}{n}\sum_{i_1=1}^n |m_{i_1}^{(L)}-m_{i_1,R}^{(L)}|\\ &=\frac{1}{n^v}\sum_{i_1=1}^n \Bigg|\frac{1}{v!}\sum_{\sigma\in \mathcal{S}_v}\Bigg[\sum_{\substack{(i_2,\ldots,i_v)\\ \in \mathcal{S}(n,v,i_1)}}\left(\prod_{(a,b)\in E(H)}Q_n(i_{\sigma(a)},i_{\sigma(b)})\right)\prod_{a=2}^v X_{i_a}^{(L)}\\ &\qquad -\sum_{\substack{(i_2,\ldots,i_v)\\ \in [n]^{v-1}}}\left(\prod_{(a,b)\in E(H)}Q_n(i_{\sigma(a)},i_{\sigma(b)})\right)\prod_{a=2}^v X_{i_a}^{(L)}\Bigg]\Bigg|\\ &\le \frac{1}{n^v}\max_{\sigma\in \mathcal{S}_v}\sum_{i_1=1}^n \bigg\lvert   \sum_{\substack{(i_2,\ldots,i_v)\\ \in \mathcal{S}(n,v,i_1)}}\left(\prod_{(a,b)\in E(H)}Q_n(i_{\sigma(a)},i_{\sigma(b)})\right)\prod_{a=2}^v X_{i_a}^{(L)}-\\ &\sum_{\substack{(i_2,\ldots,i_v)\\ \in [n]^{v-1}}}\left(\prod_{(a,b)\in E(H)}Q_n(i_{\sigma(a)},i_{\sigma(b)})\right)\prod_{a=2}^v X_{i_a}^{(L)}\bigg\rvert.
\end{align*}
 The RHS above converges to $0$ as $n\to\infty$, using~\cref{lmm:un_vn_same} with $\varphi(x_1, \ldots, x_{v})= \prod_{a=2}^{v} x^{(L)}_a$ along with triangle inequality.

   \end{proof}

   In order to prove \eqref{eq:prooflater1} and \eqref{eq:prooflater2}, we need the following additional lemma whose proof we defer to \cref{sec:fifth}.
   
			\begin{lmm}\label{lem:pivotlem}
				Fix a graph $H$ with $v$ vertices and maximum degree $\Delta$ as before. Fix $p,q>0$ such that $\frac{1}{p}+\frac{1}{q}<1$, $p\ge v$. Then $\Upsilon(\cdot,\cdot)$ is well-defined on $\mathcal{R}$, and the following conclusions hold:
				\begin{enumerate}
					\item[(i)] Fix $C>0$. Then
					$$\lim\limits_{L\to\infty}\sup_{\nu\in\wmm:\ \fmm_p(\nu)\leq C}\sup_{W\in\mathcal{W}:\ \lVert W\rVert_{q\Delta}\leq C} d_{\ell}(\Upsilon(W,\nu),\Upsilon(W,\nu^{(L)}))=0.$$
					\item[(ii)] Suppose $W_k,W_{\infty}\in\mathcal{W}$, $k\geq 1$ such that $d_{\square}(W_k,W_{\infty})\to 0$ as $k\to\infty$, and $\sup_{1 \le k \le \infty}\lVert W_k\rVert_{q\Delta}<\infty$ for $1\leq k\leq \infty$. Fix $L\in (0,\infty)$ and let $\wmm^{(L)}$ denote the subset of $\wmm$ for which the second marginal is compactly supported on $[-L,L]$. Then we have 
					$$\lim\limits_{k\to\infty}\sup_{\nu\in\wmm^{(L)}}d_{\ell}(\Upsilon(W_k,\nu),\Upsilon(W_{\infty},\nu))=0.$$
					\item[(iii)] Fix $W\in \mathcal{W}$ such that $\lVert W\rVert_{q\Delta}<\infty$, and let $\nu_k,\nu_{\infty}\in\wmm^{(L)}$ such that $d_{\ell}(\nu_k,\nu_{\infty})\to 0$. Then,
					$$\lim\limits_{k\to\infty} d_{\ell}(\Upsilon(W,\nu_k),\Upsilon(W,\nu_{\infty}))=0.$$
				\end{enumerate}
				
			\end{lmm}

\begin{proof}[Proof of \eqref{eq:prooflater1}]
    We can use~\cref{lem:pivotlem} part (i) to get the desired conclusion provided we can show $\lVert W_{Q_n}\rVert_{q\Delta}=O(1)$ and $\fmm_p(\tml)=O_p(1)$ (these requirements follow from the definition of $\mathcal{R}$, see~\cref{def:pivotlem}). But these are direct consequences of~\eqref{eq:q}~and~\eqref{eq:lip11} respectively.
\end{proof}

\begin{proof}[Proof of \eqref{eq:prooflater2}]
     We can use~\cref{lem:pivotlem} part (ii) to get the desired conclusion, if we can verify that $d_{\square}(W_{Q_n},W)\to 0$,  $\lVert W_{Q_n}\rVert_{q\Delta}=O(1)$, and  $\lVert W\rVert_{q\Delta}=O(1)$. But these are direct consequences of~\eqref{eq:cut_con},~\eqref{eq:q}, and~\eqref{eq:W_q} respectively.
\end{proof}

The final step is to establish \eqref{eq:prooflater3} for which we need two results. The first one is an immediate corollary of \cref{lem:pivotlem} parts (i) and (iii) (and hence its proof is omitted), while the second one is a simple convergence lemma, whose proof is provided in \cref{sec:appenaux}. 
\begin{cor}\label{cor:Upcont}
Consider the same setting as in \cref{lem:pivotlem}. For $C>0$, define
\begin{align}\label{eq:deftilde}
\widetilde{\mathcal{M}}_{p,C}:=\{\nu\in\widetilde{\mathcal{M}}_p:\ \fmm_p(\nu)\le C\}.
\end{align}
Suppose $W\in \mathcal{W}$ be such that $\lVert W\rVert_{q\Delta}<\infty$, then $\Upsilon(W,\cdot)$ is continuous on $\widetilde{\mathcal{M}}_{p,C}$ in the weak topology.
\end{cor}

 %\begin{lmm}\label{lem:proj_cont_map}
  %        Suppose $(X,d_X)$ and $(Y,d_Y)$ be two Polish spaces. Suppose $F$ is a compact subset of $X$. Let $\xi_n$ be a sequence of $X$-valued random variables such that $d_X(\xi_n,F) \xrightarrow{\text{P}}0$.  Then for any continuous function $g:X \mapsto Y$, we have
   %       $$d_Y(g(\xi_n),g(F)) \xrightarrow{\text{P}} 0.$$
 %\end{lmm}

\begin{lmm}\label{lem:proj_cont_map}
          Suppose $(X,d_X)$ and $(Y,d_Y)$ be two Polish spaces. Let $\xi_n$ be a sequence of $X$-valued random variables such that $d_X(\xi_n,\mf) \xrightarrow{\text{P}}0$ for some closed set $\mf\subseteq X$. Assume that there exists a compact set $K\subseteq X$ such that 
          \begin{align}\label{eq:tight}
          \lim\limits_{n\to\infty} \P(\xi_n\notin K)=0.
          \end{align}
          %Given any $\vep>0$, assume that there exists compact sets $K_{\vep}\subseteq X$ such that 
          %\begin{align}\label{eq:tight}
          %\limsup_n \P(\xi_n\notin K_{\vep})\le \vep.
          %\end{align}
          Finally consider a function $g:X \mapsto Y$ such that $g$ is continuous on $K$. %for all $\vep>0$. 
          Then we have
          $$d_Y(g(\xi_n),g(\mf)) \xrightarrow{\text{P}} 0.$$
 \end{lmm}
\begin{proof}[Proof of \eqref{eq:prooflater3}]
                
   Applying~\cref{lem:pivotlem} part (i), for every $\varepsilon>0$ we have
$$\lim_{L\to\infty}\limsup_{n\to\infty}\;\;\P\left(d_{\ell}(\Upsilon(W,\tln),\Upsilon(W,\tml))\geq \varepsilon\right)=0.$$
			It thus suffices to show that
			\begin{equation}\label{eq:target}d_{\ell}(\Upsilon(W,\tml),\Upsilon(W,\Xi(F_{\theta})))\overset{P}{\longrightarrow}0.\end{equation}
To this effect, use~\cref{prop:freen} part (iii) to note that \begin{align}\label{eq:vercond0}d_\ell(\tml, \Xi(F_\theta))\stackrel{P}{\to}0,\end{align} where the set $\Xi(F_\theta)$ is compact in the weak topology. Also note that by \eqref{eq:lip11}, there exists $C>0$ such that 
\begin{align}\label{eq:vercond1}
\lim\limits_{n\to\infty}\P(\tml\notin\widetilde{\mathcal{M}}_{p,C})=0.
\end{align}
We will now invoke~\cref{lem:proj_cont_map} with $X=\widetilde{\mathcal{M}}_p$ and $Y=\mathcal{M}$, both coupled with  weak topology, $\xi_n=\tml$, $\mf=\Xi(F_{\theta})$, $K=\widetilde{\mathcal{M}}_{p,C}$ and $g(\cdot)=\Upsilon(W,\cdot)$. Once we verify the conditions of \cref{lem:proj_cont_map} with the above specifications, we will then conclude~\eqref{eq:target}, which in turn, completes the proof. 

To verify the conditions of \cref{lem:proj_cont_map}, note that $\mf=\Xi(F_{\theta})$ is compact, and is a subset of $X=\widetilde{\mathcal{M}}_p$ by \eqref{eq:lip12}. Further, \eqref{eq:vercond0} implies $d_X(\xi_n,\mf)\overset{P}{\to} 0$. The conclusion in \eqref{eq:vercond1} implies \eqref{eq:tight}.  The fact that $g(\cdot)=\Upsilon(W,\cdot)$ is well-defined on $X$ follows from \cref{lem:wealimi} part (a) (also see \cref{def:pivotlem}). Finally, the continuity of $g$ on $K$ follows from \cref{cor:Upcont}.

This finally completes the proof of \cref{lem:wealimi}.
\end{proof}

\subsection{Proofs of \cref{cor:conmeanint} and \cref{prop:higherord}}

In order to prove \cref{cor:conmeanint}, we need the following results. The first result is a lemma about a sequence of functions converging in measure. Its proof is deferred to \cref{sec:appenaux}.
     
\begin{lmm}\label{lem:specond}
Let $U\sim\mathrm{Unif}[0,1]$ and $p\ge 1$.
\begin{enumerate}
\item[(i)]
 Suppose $\{f_n\}_{n\geq 1}$ is a sequence of measurable real-valued functions on $[0,1]$ such that $$\limsup_{n\to\infty}\E |f_n(U)|^p<\infty,\text{ and } (U,f_n(U))\overset{D}{\longrightarrow} (U,f_{\infty}(U)).$$ Then for any $\tp\in (0,p)$ we have:
\begin{equation}\label{eq:showcon}
\E|f_n(U)-f_{\infty}(U)|^{\tp}\longrightarrow 0.
\end{equation}
\item[(ii)]
If $(U,f(U))\overset{D}{=}(U,g(U))$ for some $f,g$ such that $\E |f(U)|^p<\infty$ and $\E|g(U)|^p<\infty$, then $f(U)=g(U)$ a.s.
\end{enumerate}
\end{lmm}

%We also set up some notation. 
For stating the second result, we recall the definitions of $\mathfrak{B}_{\theta}$, $\mathfrak{B}^*_{\theta}$,  $\vartheta_{W,\nu}$, $\Xi(F_{\theta})$, $\mathcal{L}_p$ from \cref{lem:wealimi}, \eqref{eq:tildef}, \eqref{eq:varthdef},  \cref{prop:freen} part (iii), and \cref{def:tilt2}, respectively. Also define
\begin{small}
\begin{equation}\label{eq:ntildef}
\widetilde{M}_{p,C}:=\left\{\mathrm{Law}(U,f(U)):\ f\in \mathcal{L}_p,\ \int_0^1 |f(u)|^p\,du\le C\right\}, \;\;\widetilde{M}_p:=\cup_{C\in \mathbb{N}}\ \widetilde{M}_{p,C}.
\end{equation}
\end{small}
With $\widetilde{\mathcal{M}}_p$ and $\widetilde{\mathcal{M}}_p$ as in \cref{def:M} and \eqref{eq:deftilde} respectively, the above definition gives $\widetilde{M}_p\subseteq \widetilde{\mathcal{M}}_p$ and $\widetilde{M}_{p,C}\subseteq \widetilde{\mathcal{M}}_{p,C}$. 
We also construct $\mathfrak{G}_1:[0,1]\times\R\to [0,1]\times \mathcal{N}$ given by $\mathfrak{G}_1(x,y):=(x,\alpha'(\theta y))$. 

We note an elementary observation here which will be used in the sequel. To wit, recall from \cref{lem:wealimi} that $\mathfrak{B}_{\theta}=\{\mathrm{Law}(U,\vartheta_{W,\nu}(U)), \nu\in\Xi(F_{\theta})\}$. Consequently from the definition of $\mathfrak{G}_1$ it follows that:
\begin{align}\label{eq:funrel}
\mathfrak{B}^*_{\theta}&= \{ \nu_1\circ\mathfrak{G}^{-1}_1: \nu_1 \in \mathfrak{B}_{\theta}\} =\{{\rm Law}(U,\alpha'(\theta \vartheta_{W,\nu}(U))), \nu\in \Xi(F_{\theta})\}.
\end{align}
%where, by a slight abuse of notation, we identify a random vector with its law, i.e., given any $\nu\in \mathfrak{B}_{\theta}$, $\mathfrak{G}_1(\nu)$ refers to the law of $\mathfrak{G}_1(A,B)$ where $(A,B)\sim \nu$. 

\vspace{0.05in} We now state the following lemma, which formalizes a key property of the sets $\mathfrak{B}_{\theta}$ and $\mathfrak{B}^*_{\theta}$. Its proof is deferred to~\cref{sec:fifth}.

\begin{lmm}\label{lem:pfcmpct}
    Consider the same setting as in \cref{lem:wealimi}. Then the set $\mathfrak{B}_{\theta}$ is a compact subset of $\widetilde{M}_q$ in the weak topology, whereas  $\mathfrak{B}^*_{\theta}$ is a compact subset of $\widetilde{M}_p$ in the weak topology.
\end{lmm}

We are now in the position to prove \cref{cor:conmeanint}. 

\begin{proof}[Proof of \cref{cor:conmeanint}]
 
   %Note that any element of $\mathfrak{B}^*_{\theta}$ above is of the form $(U,\vartheta_{W,\Xi(f)}(U))$ for some $f\in F_{\theta}$ recall $v, \Xi$. However as every element of $F_{\theta}$ satisfies~\eqref{eq:propoptshow}, we have $\vartheta_{W,\Xi(f)}(U)=f(U)$ a.s.~which implies
%\begin{equation}\label{eq:impt}
%\mathfrak{B}^*_{\theta}=\Xi(F_{\theta}),
%\end{equation}
%which is closed in the weak topology by~\cref{prop:freen} part (iv). Now as $$\mathfrak{B}_{\theta}=\left\{\Big(U,\frac{\beta(V)}{\theta}\Big):\ (U,V)\sim\nu,\ \nu\in \mathfrak{B}^*_{\theta}\right\}$$
%where $\beta=(\alpha')^{-1}$ is a continuous function, the closedness of $\mathfrak{B}_{\theta}$ in the weak topology follows.

By arguments similar to \eqref{eq:lipcon} we have
$$d_{\ell}(\mathfrak{L}_n(\mathbf{m}),\widetilde{\mathfrak{L}}_n(\mathbf{m}))\overset{P}{\longrightarrow} 0,$$
where $\widetilde{\mathfrak{L}}_n$ is defined as in~\eqref{eq:unifemp}. Consequently by invoking \cref{lem:wealimi} part (b) we get:
\begin{align}\label{eq:vercond100}
d_{\ell}(\widetilde{\mathfrak{L}}_n(\mathbf{m}),\mathfrak{B}_{\theta})\overset{P}{\longrightarrow} 0.
\end{align}
Further by \eqref{eq:lip13}, there exists $C>0$ such that 
\begin{align}\label{eq:vercond101}
\lim\limits_{n\to\infty}\P(\widetilde{\mathfrak{L}}_n(\mathbf{m})\notin \widetilde{M}_{q,C})=0.
\end{align}
With the above observations in mind, we invoke~\cref{lem:proj_cont_map} with $X=\widetilde{M}_q$, $Y=\mathcal{M}$ equipped with the  topology of weak convergence, $\xi_n= \widetilde{\mathfrak{L}}_n(\mathbf{m})$, $\mf= \mathfrak{B}_{\theta}$, $g=\mathfrak{G}_1$ and $K=\widetilde{M}_{q,C}$ (with $C$ chosen as in \eqref{eq:vercond101}). Once we verify the assumptions of \cref{lem:proj_cont_map} with the above specifications, by \eqref{eq:funrel}, we obtain:
   \begin{equation}\label{eq:refflat}
d_{\ell}\left(\widetilde{\mathfrak{L}}_n(\boldsymbol{\alpha}),\mathfrak{B}^*_{\theta}\right)\overset{P}{\longrightarrow}0,
   \end{equation}
where we recall the definition of $\boldsymbol{\alpha}$ from~\eqref{eq:alpha_define}.

To verify the conditions of \cref{lem:proj_cont_map}, note that $\mathcal{F}=\mathfrak{B}_{\theta}\subseteq X={M}_q$ by \eqref{eq:lip12}. Further, \eqref{eq:vercond100} implies $d_X(\xi_n,\mf)\overset{P}{\to} 0$ and \cref{lem:basicprop} part (ii) along with Fatou's lemma implies $\mf$ is a compact subset of $X$. The conclusion in \eqref{eq:vercond101} implies \eqref{eq:tight}.  Finally, the continuity of $g$ on $K$ follows from the continuity of $\alpha'(\cdot)$.

\vspace{0.05in} 

 We now use \eqref{eq:refflat} to complete the proof. The key tool will once again be \cref{lem:proj_cont_map}. To set things up, fixing $C>0$ equip $\widetilde{M}_{p,C}$  with the weak topology. Pick any $\nu\in\widetilde{M}_{p,C}$. Then  $\nu$ is distributed as $(U,f(U))$, where $U\sim \mathrm{Unif}[0,1]$ and  $f:[0,1]\mapsto \R$ is measurable with $\lVert f\rVert_{p}\le C$. Consequently by~\cref{lem:specond} part (ii),  the map $\mathfrak{G}_2:\wmm_{p,C}\to L^{p'}[0,1]$, (for some $p'<p$) given by $\mathfrak{G}_2(\nu)=f$ is well-defined. 

For any $f\in F_{\theta}$, setting $\nu=\Xi(f)$ use  \eqref{eq:calter} to note that $f(U)=\alpha'(\theta \vartheta_{W,\nu}(U))$ a.s. Consequently, by \eqref{eq:funrel}, we get:
   %Note that, 
%\begin{equation}\label{eq:twosame}
%\mathfrak{B}^*_{\theta}=\{(U,f(U)):\ f\in F_{\theta}\}.
%\end{equation}

   %Using \eqref{eq:twosame}, given any $\nu\in\mathfrak{B}^*_{\theta}$, there exists a unique $f$ (a.s.) such that $\nu\sim (U,f(U))$. 

%\noindent   As a result, using \eqref{eq:twosame}, we get:
   \begin{align}\label{eq:newsame}
   \mathfrak{G}_2(\mathfrak{B}^*_{\theta})=F_{\theta}.
   \end{align}
   Moreover, by \eqref{eq:lip14}, there exists $C>0$ such that 
   \begin{align}\label{eq:newsame1}
   \lim\limits_{n\to\infty} \P(\widetilde{\mathfrak{L}}_n(\boldsymbol{\alpha})\notin\widetilde{M}_{p,C})=0.
   \end{align}

   With this observation, we will invoke~\cref{lem:proj_cont_map} with $X=\widetilde{M}_p$, $Y\equiv L^{p'}[0,1]$, equipped with the topologies of weak convergence and $L^{p'}[0,1]$ respectively, and $\xi_n=\widetilde{\mathfrak{L}}_n(\boldsymbol{\alpha})$, $\mf\equiv \mathfrak{B}^*_{\theta}$, $g\equiv \mathfrak{G}_2$, and $K=\widetilde{M}_{p,C}$ with $C$ chosen from \eqref{eq:newsame1}.  %
   Once we verify the conditions of~\cref{lem:proj_cont_map}, an application of \eqref{eq:newsame} will yield
   $$\lVert \mathfrak{G}_2(U,\alpha'(\theta m_{\lceil nU\rceil})) - \mathfrak{G}_2(\mathfrak{B}^*_{\theta})\rVert_{p'}=\inf_{f\in F_{\theta}}\int_0^1 |\alpha'(\theta m_{\lceil nu\rceil})-f(u)|^{p'}\,du\overset{P}{\longrightarrow} 0,$$
   which will complete the proof of~\eqref{eq:lpcon}. %Note that in the above display, we have used~\eqref{eq:newsame}. 
      
   To verify the conditions of \cref{lem:proj_cont_map}, note that $\mf=\mathfrak{B}^*_{\theta}\subseteq X=\widetilde{M}_p$ by \eqref{eq:lip12}. Further, \eqref{eq:refflat} implies $d_X(\xi_n,\mf)\overset{P}{\to} 0$, and \cref{lem:basicprop} part (ii) implies $\mf$ is a compact subset of $X$. The conclusion in \eqref{eq:newsame1} implies \eqref{eq:tight}. The fact that $g(\cdot)=\mathfrak{G}_2(\cdot)$ is well-defined on $X=\widetilde{M}_p$ follows from \cref{lem:specond} part (b). Continuity of $g$ on $K$ follows from~\cref{lem:specond} part (i). %gives the following --- for any $\nu_k,\nu\in \widetilde{M}_{p,C}$ satisfying $\nu_k\overset{D}{\longrightarrow} 0$, we have $\lVert \mathfrak{G}_2(\nu_k)-\mathfrak{G}_2(\nu)\rVert_{p'}\to 0$. This establishes the required continuity. 

			\end{proof}

   For proving~\cref{prop:higherord}, we will need the following lemma whose proof we defer to~\cref{sec:fifth}. 
			\begin{lmm}\label{lem:auxtail}
				Suppose ${\bf X}$ is a sample from the  model~\eqref{eq:gibbs} ($\theta$ need not be non-negative). Suppose $p \in [v,\infty]$, $q>1$ satisfy~\eqref{eq:tailp}, $\limsup_{n\to\infty} \lVert W_{Q_n}\rVert_{q\Delta}<\infty$ and $\frac{1}{p}+\frac{1}{q}\leq 1$. 
    
(i) Given any vector $\mathbf{d}^{(n)}:=(d_1,d_2,\ldots ,d_n)$ such that $\lVert \mathbf{d}^{(n)}\rVert_{\infty}=O(1)$, we have
				$$\frac{1}{n}\sum_{i=1}^n d_i(X_i-\alpha'(\theta m_i))\stackrel{P}{\rightarrow}0.$$
  % \begin{lmm}\label{lem:auxtail2}
(ii) If $\frac{1}{p}+\frac{1}{q}<1$, then
   $$\frac{1}{n}\sum_{i=1}^n m_i\left(X_i-\alpha'(\theta m_i)\right)\stackrel{P}{\rightarrow}0.$$
   \end{lmm}

\begin{proof}[Proof of~\cref{prop:higherord}]
				
				By~\cref{prop:propopt} part (iii), all the optimizers of the problem in~\eqref{eq:gibbsop} are constant functions. Further,~\eqref{eq:lip12} shows that there exists $K>0$ (depending on $\theta$) such that all the optimizers of~\eqref{eq:gibbsop} have $L^p$ norm bounded by $K$. Combining these two observations, we have that $F_{\theta}$ consists only of constant functions where the constants are given by
				$$\mathcal{A}_{\theta}=\argmin_{t\in\mathcal{N},\ |t|\leq K} [\gamma(\beta(t))-\theta t^v].$$
				As analytic non-constant functions can only have finitely many optimizers in a compact set, it follows that $A_\theta$ is a finite set. %the conclusion follows.
				\\
				
				\emph{(i)} Define $c_{i,L}:=c_i 1\{|c_i|\leq L\}$ and $\bar{m}:=n^{-1}\sum_{i=1}^n m_i$. We claim that result follows given the following display: 
				\begin{equation}\label{eq:p2i1}
					\lim_{L\to\infty}\limsup_{n\to\infty}\E\left[\frac{1}{n}\sum_{i=1}^n \left|c_i-c_i^{(L)}\right||X_i|\right]=0.
				\end{equation}
				%\begin{equation}\label{eq:p2i2}
				%	\frac{1}{n}\sum_{i=1}^n c_i^{(L)}(X_i-\alpha'(\theta m_i))= o_P(1),
				%\end{equation}
				This is because given any $L>0$ and any $t\in \mathcal{A}_{\theta}$ (recall this implies $|t|\le K$), the following inequalities hold:
    \begin{align*}
    &\;\;\;\;\;\bigg|\frac{1}{n}\sum_{i=1}^n c_iX_i\bigg|\\ &\le \frac{1}{n}\sum_{i=1}^n \bigg|c_i-c_i^{(L)}\bigg||X_i|+\bigg|\frac{1}{n}\sum_{i=1}^n c_i^{(L)}(X_i-\alpha'(\theta m_i))\bigg|+\bigg
    |\frac{1}{n}\sum_{i=1}^n c_i^{(L)}(\alpha'(\theta m_i)-t)\bigg|+\frac{|t|}{n}\bigg|\sum_{i=1}^n c_i^{(L)}\bigg|\\ &\le \frac{1}{n}\sum_{i=1}^n \bigg|c_i-c_i^{(L)}\bigg||X_i|+\bigg|\frac{1}{n}\sum_{i=1}^n c_i^{(L)}(X_i-\alpha'(\theta m_i))\bigg|+\frac{L}{n}\sum_{i=1}^n |\alpha'(\theta m_i)-t|\\ &\quad\quad\quad +\frac{K}{n}\bigg|\sum_{i=1}^n c_i\bigg|+\frac{K}{n L^{r-1}}\sum_{i=1}^n |c_i|^{r}.
    \end{align*}
    Taking an infimum over $t\in \mathcal{A}_\theta$ gives the bound
    \begin{align*}
    \bigg|\frac{1}{n}\sum_{i=1}^n c_iX_i\bigg| &\le 
 \frac{1}{n}\sum_{i=1}^n \bigg|c_i-c_i^{(L)}\bigg||X_i|+\bigg|\frac{1}{n}\sum_{i=1}^n c_i^{(L)}(X_i-\alpha'(\theta m_i))\bigg|\\
 &+\inf_{t\in \mathcal{A}_\theta}\frac{L}{n}\sum_{i=1}^n |\alpha'(\theta m_i)-t|+\frac{K}{n}\bigg|\sum_{i=1}^n c_i\bigg|+\frac{K}{n L^{r-1}}\sum_{i=1}^n |c_i|^{r}.
    \end{align*}

    The first and last terms above converge to $0$ in probability as $n\to\infty$ first, followed by $L\to\infty$, by using~\eqref{eq:p2i1} and the assumption $\sum_{i=1}^n |c_i|^r=O(n)$ for $r>1$. The remaining terms converge to $0$ as $n\to\infty$ for fixed $L>0$, by using~\cref{lem:auxtail} part (i),~\eqref{eq:lpcon}, and $\sum_{i=1}^n c_i=o(n)$, respectively. This completes the proof.
    
				Next, we prove~\eqref{eq:p2i1}. Fix  $r'\in (1,r)$ such that $\frac{1}{p}+\frac{1}{r'}=1$. By H\"{o}lder's inequality,
				\begin{align*}
			&\E\left[\frac{1}{n}\sum_{i=1}^n \left|c_i-c_i^{(L)}\right||X_i|\right]=\E\left[\frac{1}{n}\sum_{i=1}^n \left|c_i\right|1\{|c_i|> L\}|X_i|\right]\\
 \leq& \left(\frac{1}{N}\sum_{i=1}^n |c_i|^{r'}1\{|c_i|>L\}\right)^{\frac{1}{r'}}\left(\frac{1}{n}\sum_{i=1}^n \E|X_i|^{p}\right)^{\frac{1}{p}}
   \leq  \frac{1}{L^{\frac{r-r'}{r'}}}\left(\frac{1}{N}\sum_{i=1}^n |c_i|^r\right)^{\frac{1}{r'}}\left(\frac{1}{n}\sum_{i=1}^n \E|X_i|^{p}\right)^{\frac{1}{p}}.
   \end{align*}
				This, along with~\eqref{eq:lip11} and the assumption $\sum_{i=1}^n |c_i|^r=O(n)$, establishes~\eqref{eq:p2i1}.
   % \\

 %   \noindent\emph{Proof of~\eqref{eq:p2i2}.} This follows directly from~\cref{lem:auxtail} part (a).				
				
				\vspace{0.1in}

    \emph{(ii)} As in part (i),  we have
    \begin{small}
    \begin{align*}
   \inf_{t\in \mathcal{A}_\theta} \bigg|\frac{1}{n}\sum_{i=1}^n c_iX_i-c_0t\bigg|\le \bigg|\frac{1}{n}\sum_{i=1}^n (c_i-c_0)X_i\bigg|+\frac{|c_0|}{n}\bigg|\sum_{i=1}^n (X_i-\alpha'(\theta m_i))\bigg|+\frac{|c_0|}{n}\inf_{t\in\mathcal{A}_{\theta}}\sum_{i=1}^n  |\alpha'(\theta m_i)-t|.
    \end{align*}
    \end{small}
    The first term converges to $0$ in probability by part (i), the second converges to $0$ by~\cref{lem:auxtail} part (a), and the third term converges to $0$ by~\eqref{eq:lpcon}.
			    This completes the proof.
				\\
				
    \emph{(iii)} 
				%By \eqref{eq:lip13}, we have:
			%\begin{align}\label{eq:lip13}
  % \sum_{i=1}^n |m_i|^{q}=O_p(n).
%			\end{align}
%			 To complete the proof using~\eqref{eq:lip13}, 
We begin by observing that for any $L>1$, we have:
			$$\mathfrak{C}_{L,1}:=\sup_{x\in [-L,L]} \bigg|\frac{d}{dx}(x\alpha'(x))\bigg|<\infty,\quad \quad \mathfrak{C}_{L,2}:=\inf_{x\in [-L,L]} \alpha''(x)>0,$$
			both of which follow from standard properties of exponential families. Recall that for any $t\in\mathcal{A}_{\theta}$, we have $t=\alpha'(\theta v t^{v-1})$ by~\cref{prop:propopt} part (i). 
           Now, if $ m_i$ and  $v t^{v-1}$ both  lie between $[-L,L]$, then we have
   $$|\alpha'(\theta m_i)-t|=|\alpha'(\theta m_i)-\alpha'(\theta v t^{v-1})|\ge |\theta|{\mathfrak{C}}_{L|\theta|,2}| m_i-v t^{v-1} |.$$
   and so for $L\ge L_0:=\max_{t\in \mathcal{A}_\theta}|v t^{v-1}|$ (recall that $\mathcal{A}_\theta$ is a finite set, as proved above) we have %for all large $L$ (depending on $\theta,K$) we have
   %With this in view, observe the following elementary inequalities for $L$ large enough (depending on $K$):
    \begin{align}\label{eq:tooshow}
			     &\;\;\;\;\inf_{t\in \mathcal{A}_{\theta}}\frac{1}{n}\sum_{i=1}^n |m_i-v t^{v-1}|\nonumber \\ &\leq  \inf_{t\in \mathcal{A}_{\theta}}\frac{1}{n\mathfrak{C}_{L|\theta|,2}}\sum_{i=1}^n |\alpha'(\theta m_i)-t|+\frac{2}{n}\sum_{i=1}^n |m_i|\mathbbm{1}(|m_i|\geq L)\nonumber \\&\leq  \inf_{t\in \mathcal{A}_{\theta}}\frac{1}{n\mathfrak{C}_{L|\theta|,2}}\sum_{i=1}^n |\alpha'(\theta m_i)-t|+\frac{2}{nL^{q-1}}\sum_{i=1}^n |m_i|^{q}.
			 \end{align}
  The RHS above converges to $0$ in probability, as $n\to\infty$, followed by $L\to\infty$. This is because, 
			the first term in~\eqref{eq:tooshow} converges to $0$ in probability as $n\to\infty$ for every fixed $L$, by using~\eqref{eq:lpcon}. The second term converges to $0$ in probability by taking $n\to\infty$ first, followed by $L\to\infty$, by using~\eqref{eq:lip13}. 

   Next choose $\tq\in (1,q)$ such that $p^{-1}+\tq^{-1}=1$. Note that for any $t\in\mathcal{A}_{\theta}$, $vt^{v-1}\alpha'(\theta vt^{v-1})=vt^v$ (using the relation $t=\alpha'(\theta v
 t^{v-1})$). In the same vein as~\eqref{eq:tooshow}, by using~\eqref{eq:condex}, we also get for all $L$ large enough:
   \begin{small}
   \begin{align}\label{eq:2sh}
&\;\;\;\;\;\inf_{t\in\mathcal{A}_{\theta}}\frac{1}{n}\sum_{i=1}^n |m_i\alpha'(\theta m_i)-vt^v|\nonumber\\ &\le \inf_{t\in\mathcal{A}_{\theta}}\frac{\mathfrak{C}_{L|\theta|,1}}{|\theta| n}\sum_{i=1}^n |m_i-vt^{v-1}|+\frac{1}{L^{\frac{q}{\tq}-1}}\left(\frac{1}{n}\sum_{i=1}^n |m_i|^q\right)^{\frac{1}{\tq}}\left(\frac{1}{n}\sum_{i=1}^n |\alpha'(\theta m_i)|^{p}\right)^{\frac{1}{p}}.
   \end{align}
   \end{small}
   The RHS above converges to $0$ in probability, as $n\to\infty$, followed by $L\to\infty$. This is because, the first term converges to $0$ as $n\to\infty$ by~\eqref{eq:tooshow}, and the second term converges to $0$ as $n\to\infty$ followed by $L\to\infty$ by using~\eqref{eq:lip13} and \eqref{eq:lip11}. Finally, the conclusion in part (iii) follows by combining~\eqref{eq:2sh} with~\cref{lem:auxtail} part (ii).

%  It only remains to verify~\eqref{eq:lip13}, which we do below.

   %\vspace{0.1in}

			\end{proof}			
   							
\section{Proof of Results from~\cref{sec:isingpotts}}\label{sec:exrespf}

  	\begin{proof}[Proof of~\cref{prop:isingdom}] \emph{(i)} Note that quadratic forms correspond to the choice $H=K_2$ and $v=2$ in \eqref{eq:U}. Let $\mu_\theta$ be the tilted probability measure on $\R$ obtained from $\mu$ as in~\cref{def:tilt}. 
    % \begin{align}\label{eq:newbase}
   %  d{\mu}_B(x):=\exp(Bx-{\alpha}(B))\,d\mu(x).
   %  \end{align}
     Then a direct computation using~\eqref{eq:ising} gives
	\begin{align}
			  \isZ_n(\theta,B)
    %\frac{1}{n}\log \E_{{\bf X}\sim {\mu}^{\otimes n}}\exp\left(\frac{\theta}{n}\sum_{i\neq j} Q_n(i,j)X_i X_j\right)
    ={\alpha}(B)+\frac{1}{n}\log{\E_{{\bf X}\sim {\mu}_B^{\otimes n}}\exp\left(\frac{\theta}{n}\sum_{i\neq j} Q_n(i,j)X_i X_j\right)}.
			\end{align}
			%we can absorb the linear term in~\eqref{eq:ising} to rewrite the measure $\isR(\cdot)$ in the same form as in~\cref{def:Gibbs}, with $\mu^{\otimes n}$ replaced with $\tilde{\mu}^{\otimes n}$. Also, by the symmetry of $\isR(\cdot)$, we can assume without loss of generality that $B\geq 0$. Consider $\tilde{\beta}$ and $\tilde{\gamma}$ according to~\cref{def:tilt}.
			%	\emph{Part (i).} {\color{black} A direct computation shows that
				%$$\isZ_n(\theta,B)=\tilde{\alpha}(B)+\frac{1}{n}\log{\E_{{\bf X}\sim \tilde{\mu}^{\otimes n}}\exp\left(\frac{\theta}{n}\sum_{i\neq j} Q_n(i,j)X_i X_j\right)}.$$
				Using this along with \cref{prop:freen} part (iii) we get
		\begin{align}
		\notag&\isZ_n(\theta,B)-\alpha(B)\\
		&\to \sup_{f\in \mathcal{L}_p} \left(\theta \int_{[0,1]^2} W(x,y)f(x)f(y)\,dx\,dy-\int_{[0,1]}{\gamma}_B({\beta}_B(f(x)))\,dx\right)\nonumber \\&=\sup_{f\in \mathcal{L}_p} \left(\theta \int_{[0,1]^2} W(x,y)f(x)f(y)\,dx\,dy-\int_{[0,1]}(\gamma(\beta(f(x)))+{\alpha}(B)-B f(x))\,dx\right).\end{align}
			    Here ${\gamma}_B(.)$ and ${\alpha}_B(.)$ are as in~\cref{def:tilt}, but for the tilted measure ${\mu}_B$ instead of $\mu$, and the last equality uses \eqref{eq:simplify}. This concludes the proof of part (i).
       %\begin{small}\begin{align}\label{eq:simplify}
    %\alpha_B(\theta)=\alpha(\theta+B)-\alpha(B),\;\; \beta_B(t)=\beta(t)-B,\;\; \gamma_B(\beta_B(t))=\gamma(\beta(t))+\alpha(B)-Bt.
    %   \end{align}
      % \end{small}
       %Let $F_{\theta,B}^{(1)}$ denote the set of optimizers in the right hand side of~\eqref{eq:interim}. 
				%Recall the definition of the set $F_{\theta}\equiv F_{\theta,B}$ from~\cref{lem:gibbs}, part (ii), and the definition of $\Xi(\cdot)$ from~\cref{def:xi1}. Then $F_{\theta,B}=\{\Xi(f):\ f\in F_{\theta,B}^{(1)}\}$.  

    \emph{(ii)} By invoking~\cref{prop:propopt} part (iii), if $v$ is even, the set of optimizers $F_\theta\equiv F_{\theta,B}$ 
    in the above display are constant functions, %, and~\eqref{eq:simplify}, we then have that $F_{\theta,B}^{(1)}$ contains only constant functions 
				where the constant is an optimizer of the following optimization problem:
	\begin{equation}			
    \sup_{x\in \alpha'(\R)} \left(\theta x^2+Bx-x \beta(x)+\alpha(\beta(x))\right),
    \end{equation}
				%\cref{lem:fixsol} characterizes the solutions of the above problem.
				where we recall from~\cref{def:tilt} that $\gamma(\theta)= \theta \alpha (\theta)- \alpha' (\theta)$. By~\cref{lem:fixsol} (parts (i) and (ii)), if either (a) $B\ne 0$, or (b) $B=0$, $\theta\leq (\alpha''(0))^{-1}/2$,  then the optimizer is $x=\tm$. 
				%singleton consisting of $\mbox{Unif}[0,1]\otimes \mu_{\beta(\tm)}$. 
				On the other hand when $B=0$ and $\theta>(\alpha''(0))^{-1}/2$, by~\cref{lem:fixsol} part (iii)
				the optimizers are $x=\pm \tm$. 
		Using this, the desired conclusion of part (ii) follows.		
		%then $F_{\theta,B}$ consists of two probability distributions $\mbox{Unif}[0,1]\otimes \mu_{\beta(\tm)}$ and $\mbox{Unif}[0,1]\otimes \mu_{-\beta(\tm)}$. By using the symmetry of $\mu_{\theta}$ when $B=0$, coupled with~\cref{lem:gibbs}, part (ii), completes the proof.
				\\

\noindent \emph{(iii)} Recall the definition of $\mathcal{A}_{\theta}\equiv \mathcal{A}_{\theta,B}$ from \eqref{eq:defset}, %where we note the dependence of $\mathcal{A}_\theta$ on the titled base measure $\mu_B$. %While this may seem like a notational abuse, note that from \eqref{eq:defset}, $\mathcal{A}_{\theta}$ implicitly depends on the base measure $\mu$. As we have identified the changed the base measure to $\mu_B$ (see \eqref{eq:newbase}), we are using the $\mathcal{A}_{\theta,B}$ notation. 
and use part (ii) to note that all functions in $F_{\theta,B}$ are constant functions, with constants belonging to the set $\mathcal{A}_{\theta,B}$. Since $v=2$, we have
$$\{{v t^v}:\ t\in\mathcal{A}_{\theta,B}\}= \{{2t^2}:\ t\in\mathcal{A}_{\theta,B}\}=2t_{\theta,B}^2\Rightarrow \frac{1}{n}\sum_{i=1}^n X_i m_i\overset{d}{\longrightarrow} 2\tm^2,$$
where we use~\cref{prop:higherord} part (iii).

For the weak limit of $\bar{X}$ we invoke \cref{prop:higherord} part (ii) with $c_i=1$ which implies $c_0=1$. The conclusion follows by noting that when $B=0$, the symmetry of $\mu$ about the origin implies that $\bar{X}$ and $-\bar{X}$ have the same distribution.
\end{proof}
   
\begin{proof}[Proof of~\cref{prop:genopt}] 
 \emph{(i)} Let $\mu_\theta$ be the tilted measure obtained from $\mu$ as in~\cref{def:tilt}, and  let $\alpha_B(.), \beta_B(.),{\gamma}_B(.)$ be as in~\cref{def:tilt}, but for the measure ${\mu}_B$ instead of $\mu$. 
Using~\eqref{eq:simplify} we get $$\gamma_B(\beta_B(t))=\gamma(\beta(t))+\alpha(B)-Bt,$$ using which the optimization problem in
\eqref{eq:with_h_opt} (ignoring the additive constant $\alpha(B)$) becomes
 \begin{align}\label{eq:opta}
          \sup_{f\in \mathcal{L}_p:\ \int_{[0,1]}  \gamma( \beta(f(x)))dx<\infty}\left\{\theta G_{W}(f) -\int_{[0,1]} \gamma_B( \beta_B(f(x)))dx\right\}.
          \end{align}
         We note that even though \eqref{eq:simplify} was derived under the assumption that $\mu$ is symmetric  (see~\cref{prop:mu_suff} part (b)), this assumption is not needed for \eqref{eq:simplify}.
Now, we invoke \cref{prop:propopt} part (i) to conclude that any maximizer of the above display satisfies the fixed point equation \eqref{eq:higher_fixed_point}.
\\
\emph{(ii)} It suffices to show that all optimizers of \eqref{eq:opta} are constant functions, for which invoking~\cref{prop:propopt} part (iii) it suffices to show that $\mu_B$ is stochastically non-negative (as per~\cref{def:stochastic_nn}). But this follows on noting that $\mu$ is stochastically non-negative, and $B\ge 0$.
%In this case we have $\gamma(\beta(t)) \le \gamma(\beta(-t))$ for $t \ge 0$. Along with \eqref{eq:simplify}, this gives
%\begin{equation*}
%     \gamma_B(\beta_B (t))= \gamma(\beta(t)) +\alpha (B) - B t \le \gamma(\beta(-t)) +\alpha (B) + B t = \gamma_B( \beta_B (-t)),
% \end{equation*}
% where we use the fact that $B\ge 0$. This shows that $\mu_B$ is stochastically non-negative as well. 
% We invoke \cref{prop:propopt} with replacing $\theta G_W(f)$ by  $\theta G_W(f) +h \int f(x) dx$. Hence, \eqref{eq:propopt2} will be replaced by the optimality condition:
% \begin{align}\label{eq:propop_with_h}
% 				\footnotesize{\int_{[0,1]}\Big((g(x_1)-f(x_1))\underbrace{\left(\beta(f(x_1)) + h f(x_1)-\theta v \int_{[0,1]^{v-1}} \mathrm{Sym}[W](x_1,\ldots ,x_v)\left(\prod_{a=2}^v f(x_a)\,dx_a\right)\right)}_{\delta(x_1)}\,dx_1\Big)\geq 0.}
% 			\end{align}
% The remainder of the proof follows the proof of \cref{prop:propopt} part (i) verbatim to yield any maximizer satisfies the fixed point equation \eqref{eq:higher_fixed_point}. 
% Hence, by the proof of \cref{prop:propopt} part (iii), the maximizers of \eqref{eq:with_h_opt} are constant functions provided either $v$ is even or $\mu$ is stochastically non-negative. 
Finally,~\eqref{eq:propoptlat} follows from~\eqref{eq:higher_fixed_point}, on setting $f(.)$ to be a constant function.
\\

%yields any constant maximizer satisfy the fixed point equation $x= \alpha'(\theta vx^{v-1}+h)$, setting $f(x) \equiv x$ in \eqref{eq:higher_fixed_point}.

\emph{(iii)(a)} The optimization problem~\eqref{eq:with_h_opt} 
 reduces to maximizing
\begin{align}\label{eq:optimize_H}
    H_{\theta,B}(x):= \theta x^v +B x - \gamma(\beta(x))
\end{align}
over $x\in [-1,1]$.
Differentiating we get
\begin{align}\label{eq:H_derivatives}
    H'_{\theta,B}(x) = \theta v x^{v-1} +B -\beta(x), \quad
    H''_{\theta,B}(x)= \theta v (v-1) x^{v-2} -\beta'(x).
\end{align}
Since, $\mu$ is supported on $[-1,1]$, we have $\lim_{\theta\to\infty}\alpha'(\theta)=1$, and so $$\alpha '' (\theta) = \mathbb{E}_{\mu_\theta}(X^2)- (\alpha ' (\theta))^2\le 1-(\alpha'(\theta))^2\rightarrow 0$$
as $\theta \rightarrow \infty$.
Hence, there exists $B_0=B_0(\theta,v)$ such that for $B\ge B_0$ we have $\alpha''(B)< \frac{1}{2 \theta v(v-1)}$. If $x$ is a global maximizer of $H(.)$, then we have 
%Since $\alpha''$ is decreasing, and $\lim_{x \rightarrow \infty} \alpha ''(x)=0$, pick $h^\star= h^\star (\theta,v)$ such that for any $h \ge h^\star$, we have $\alpha''(h)< \frac{1}{\theta v (v-1)}$. Any global maximizer of $H$ satisfies 
$$x= \alpha'(\theta v x^{v-1} +B) \ge \alpha'(B) \implies \beta(x) \ge B.$$ However, on the interval $ \{x:\beta(x) \ge B\}$, using boundedness of support, we have 
$$H''_{\theta,0}(x) \le \theta v (v-1) - \frac{1}{\alpha ''(\beta(x))} <0.$$
Thus $H_{\theta,B}(.)$ is strictly concave on the interval $\{x:\beta(x)\ge B\}$, and so the global maximizer must be unique. \\\\
% We already have shown any maximizer of \eqref{eq:optimize_H} satisfies 
% \begin{equation*}
%     x = \alpha'(\theta v x^{v-1}) \le \alpha' (\theta v) \le K,
% \end{equation*}
% if $\theta \le  \frac{1}{v} \beta(K)$. 
\emph{(iii)(b)} %Set $h=0$. Any global maximizer of $H$ satisfies the equation $x= \alpha'(\theta vx^{v-1})$. Using \cref{lem:fixsol0}, there exists $C>0$ such that $\alpha'(x)\le C x^{1/p-1}$ for $p \ge v$. Set $\theta<(C^{v-1} v)^{-1}$.\par
%Suppose $p=v$. Then 
%$$x= \alpha'(\theta vx^{v-1}) \le C (\theta v)^{1/v-1} x \implies \theta > (C^{v-1} v)^{-1},$$
%contradicting our assumption on $\theta$.\par 
%Now suppose $p>v$. 
%We have
%\begin{align*}
%    &x= \alpha'(\theta vx^{v-1}) \le C (\theta v x^{v-1} )^{1/p-1}\\
 %   &\implies x^{\frac{p-v}{p-1}} \le C (\theta v )^{1/p-1} \\
%    & \implies x \le C^{\frac{p-1}{p-v}}(\theta v )^{1/p-v}
%\end{align*}
 We break the proof into the following steps:

\begin{itemize}
    \item {There exists $\theta_{1c}\in (0,\infty)$ such that $0$ is the unique global maximizer for $H_{\theta,0}(.)$.}
    \\
    Since $\mu$ is compactly supported on $[-1,1]$, we have $$\alpha''(\theta)=\mathrm{Var}_{\mu_\theta}(X)\le 1,\text{ and so }\beta'(x)= \frac{1}{\alpha''(\beta(x))} \ge 1.$$
    %\%where we have used the fact $\alpha'(0)=0$ and monotonicity of $\alpha''(.)$.
%By continuity of $\beta'$, there exists $y_0>0$ such that $\beta'(y) > \frac{1}{2} \beta'(0)>0$ for all $y \in [0, y_0]$.
Thus for $\theta<\frac{1}{2v(v-1)}=:\theta_{1c}$ we have
%Hence, on the set $\mathcal{A}= \{ x \le C^{\frac{p-1}{p-v}}(\theta v )^{1/p-v}\}$, we have
\begin{align*}
    H''_{\theta,0}(x) \le \theta v(v-1) -\beta'(x)< 0,
\end{align*}
and so $H_{\theta,0}$ is strictly concave.  Since $H'_{\theta,0}(0)=0$, $x=0$  is the unique global maximizer of $H_{\theta,0}(.)$. %Therefore, there exists $\theta_c$ such that if $0 < \theta < \theta_c$, $x=0$ is the unique maximizer of $H$. 
\\

\item{There exists $\theta_{2c}\in (0,\infty)$ such that for $\theta>\theta_{2c}$, $0$ is not a global maximizer of $H_{\theta,0}(.)$
}

We consider two separate  cases: \begin{itemize}
\item 
$\mu$ is stochastically non-negative.

In this case there exists $x_0>0$ such that $\gamma(\beta(x_0))\in (0,\infty)$. Then setting $\theta_{2c}:=x_0^{-v}\gamma(\beta(x_0))\in (0,\infty)$, for $\theta>\theta_{2c}$ we have
$$H_{\theta,0}\left(x_0\right)= \theta x_0^{-v}- \gamma\left(\beta\left(x_0\right)\right)>0=H_{\theta,0}(0),$$
and so $0$ cannot be a global optimizer of $H_{\theta,0}(.)$

\item $v$ is even.

If there exists $x_0>0$ such that $\gamma(\beta(x_0))\in (0,\infty)$, we are through by previous argument. Othewise, since $\mu$ is not degenerate at $0$, there exists $x_0<0$ such that  $\gamma(\beta(x_0))\in (0,\infty)$. Again setting $\theta_{2c}:=x_0^{-v}\gamma(\beta(x_0))\in (0,\infty)>0$ with $v$ even, the same  proof works.
\end{itemize}

\item{For any $\theta>0$, let $x_\theta$ be any non-negative global optimizer of $H_{\theta,0}(.)$. Then the map $\theta\mapsto x_\theta$ is non-decreasing.}
\\
Suppose by way of contradiction there exists $0<\theta_1<\theta_2<\infty$ such that $0\le x_{\theta_2}<x_{\theta_1}$. % is a global maximizer of $H_{\theta_2}$, then any global maximizer of $H_{\theta_1}$ satisfies $x_1 \le x_2$. Assume the contrapositive: say there exists $x_1>x_2$ global maximizer of $H_{\theta_1}$. 
By optimality of $x_{\theta_1}$  we have
\begin{eqnarray*}
      \theta_1 x^v_{\theta_1} -\gamma(\beta(x_{\theta_1})) &\ge &\theta_1 x^v_{\theta_2} - \gamma(\beta(x_{\theta_2}))\\
      \implies \theta_1 (x^v_{\theta_1}-x^v_{\theta_2}) &\ge &\gamma(\beta(x_{\theta_1}))-\gamma(\beta(x_{\theta_2}))\\
     \implies \theta_2 (x^v_{\theta_1}-x^v_{\theta_2}) &>&\gamma(\beta(x_{\theta_1}))-\gamma(\beta(x_{\theta_2})).
   % &\implies H_{\theta_2,0}(x_{\theta_1})> H_{\theta_2,0}(x_{\theta_2}),
    \end{eqnarray*}
    Here the last implication uses the fact $x_{\theta_1}>x_{\theta_2}\ge 0$. But this
contradicts the fact that $x_{\theta_2}$ is a global maximizer for $H_{\theta_2,0}(.)$. 
\\

Combining the last three claims, the conclusion of part (iii)(b) follows on setting
$
    \theta_c := \sup_{\theta>0} \{0 \text{ is a global maximizer of } H_{\theta,0}(.)\}.$
    \\
    
\end{itemize}
\end{proof}
   
			\section{Proof of Main Lemmas}\label{sec:fifth}

						\begin{proof}[Proof of~\cref{lem:pivotlem}] 
							%Throughout this proof, we will write $\lesssim$ to hide constants that depend only on $H,p,q$. 
       As before, we also choose $\tp\in (1,p)$ and $\tq\in(1,q)$ such that $\frac{1}{\tp}+\frac{1}{\tq}=1$. Note that $\Upsilon(\cdot,\cdot)$ is  well-defined on $\mr$ (see~\cref{def:pivotlem}) by using~\cref{lem:Tgraphon0} (ii), with $W(\cdot,\cdot)$ as is,  $\phi(x_1,\ldots ,x_v)=\prod_{a=2}^v x_a$, and $p,q$ replaced with $\tp,\tq$.
							
							\vspace{0.1in}
							
							\emph{(i)} Recall the definition of $\vartheta_{W,\nu}(\cdot)$ from~\eqref{eq:varthdef}, and the connection $\Upsilon(W,\nu)=\mathrm{Law}\left(U_1,{\vartheta_{W,\nu}(U_1)}\right)$ between  $\Upsilon$ and $\vartheta_{W,\nu}$ from~\cref{def:pivotlem}. We will prove the following stronger claim.
							\begin{equation}\label{eq:pivotlem1}
							\sup_{W:\ \lVert W\rVert_{q\Delta}\leq C}\sup_{\nu:\ \fmm_p(\nu)\leq C}\int_0^1 |\vartheta_{W,\nu}(u)-\vartheta_{W,\nu_{L}}(u)|\,du\to 0,\quad \mbox{as}\ L\to\infty.
							\end{equation}
						Towards this end, fix $L>1$ and note that
						{\small \begin{align*}
						&\;\;\;\;|\vartheta_{W,\nu}(u)-\vartheta_{W,\nu_L}(u)|\\ & \le v\sum_{\substack{A\subseteq \{2,\ldots ,v\},\\ |A|\geq 1}}\E\left[|\mathrm{Sym}[W](u,U_2,\ldots ,U_v)|\left(\prod_{a\in A}|V_a|1\{|V_a|>L\}\right)\left(\prod_{a\in A^c}|V_a|1\{|V_a|\leq L\}\right)\right].
						\end{align*}}
				      	For every non-empty fixed set $A\subseteq \{2,\ldots ,v\}$, an application of~\cref{lem:Tgraphon0} part (ii) 
           with $W(\cdot,\cdot)$ as is, $$\phi(x_1,\ldots ,x_v)=\left(\prod_{a\in A} |x_a|\mathbbm{1}(|x_a|\geq L)\right)\left(\prod_{a\in A^c} |x_a|\mathbbm{1}(|x_a|\leq L)\right),$$
				      	and $p,q$ replaced by $\tp$, $\tq$ on the above bound, gives
				      	\begin{align}\label{eq:neednow}
				      		&\;\;\;\;\sup_{W:\ \lVert W\rVert_{q\Delta}\leq C}\sup_{\nu:\ \fmm_p(\nu)\leq C}\int_0^1 |\vartheta_{W,\nu}(u)-\vartheta_{W,\nu^{(L)}}(u)|\,du\nonumber \\ &\le v \sup_{W:\ \lVert W\rVert_{q\Delta}\leq C}\sup_{\nu:\ \fmm_p(\nu)\leq C} \sum_{A\subseteq \{2,\ldots ,v\},\ |A|\geq 1} \lVert W\rVert_{\tq\Delta}^{|E(H)|} \Bigg(\left(\prod_{a\in A} \E_{\nu}[|V_a|^{\tp}\mathbbm{1}(|V_a|\geq L)]\right)\nonumber \\&\;\;\;\;\;\;\qquad\left(\prod_{a\in A^c} \E_{\nu}[|V_a|^{\tp}\mathbbm{1}(|V_a|\leq L)]\right)\Bigg)^{\frac{1}{\tp}} \nonumber \\ &\le v2^{v} \sup_{W:\ \lVert W\rVert_{q\Delta}\leq C}\sup_{\nu:\ \fmm_p(\nu)\leq C} L^{\tp - p}\lVert W\rVert_{\tq \Delta}^{|E(H)|}(1+\fmm_p(\nu))^{\frac{v-1}{p}}\to 0,
				      	\end{align}
			      	    as $L\to\infty$. This proves~\eqref{eq:pivotlem1}, and hence completes part (i).
			      	    
			      	    \vspace{0.1in}
			      	    
			      	    \emph{(ii)} Given $W\in\mathcal{W}$, $\nu\in\widetilde{\mathcal{M}}$ and any $u\in [0,1]$, define
			      	    \begin{equation}\label{eq:rowsumtrun}
			      	    	\mathfrak{R}(u;W):=\E[\mathrm{Sym}[|W|](u,U_2,\ldots ,U_v)]
			      	    \end{equation}
			      	    %for $k=\{1,2,\ldots ,\infty\}$ and 
              where $U_2,\ldots ,U_v\overset{i.i.d.}{\sim} \mathrm{Unif}[0,1]$. For $k<\infty$,\ and $T>0$, define 
			      	    $$c_k^{(T)}(u):=1\{\mathfrak{R}(u;W_k)\leq T,\ \mathfrak{R}(u;W_{\infty})\leq T\},$$
			      	    %$$W_{k,\infty}^{(T)}(u_1,\ldots ,u_v):=\mathrm{Sym}[W_k](u_1,\ldots ,u_v)c_k^{(T)}(u_1),$$
			      	    %$$W_{\infty,k}^{(T)}(u_1,\ldots ,u_v):=\mathrm{Sym}[W_{\infty}](u_1,\ldots ,u_v)c_k^{(T)}(u_1),$$
			      	    for $u_1\in [0,1]$. With this notation, triangle inequality gives
            \begin{align}
            d_\ell\Big(\Upsilon(W_k,\nu),\Upsilon(W_\infty,\nu)\Big)=&d_\ell\Big(\mathrm{Law}\big(U, \vartheta_{W_k,\nu}(U)\big), \mathrm{Law}\big(U, \vartheta_{W_\infty,\nu}(U)\big)\Big) \nonumber \\
                 \le &d_\ell\Big(\mathrm{Law}\big(U, \vartheta_{W_k,\nu}(U)\big), \mathrm{Law}\big(U, \vartheta_{W_k,\nu}(U)c_k^{(T)}(U)\big)\Big) \nonumber \\
                 +&d_\ell\Big(\mathrm{Law}\big(U, \vartheta_{W_\infty,\nu}(U)\big), \mathrm{Law}\big(U, \vartheta_{W_\infty,\nu}(U)c_k^{(T)}(U)\big)\Big) \nonumber \\
                 +&d_\ell\Big(\mathrm{Law}\big(U, \vartheta_{W_k,\nu}(U)c_k^{(T)}(U)\big), \mathrm{Law}\big(U, \vartheta_{W_\infty,\nu}(U)c_k^{(T)}(U)\big)\Big). \label{eq:method_of_moment}
                 \end{align}
                 Using the above display,
                 %by a truncation followed by a simple method of moments argument,
                 the conclusion in part (ii) will follow if we can show the following:
\begin{eqnarray}\label{eq:pivotlem2}
	&&\lim\limits_{T\to\infty}\limsup_{k\to\infty} \sup_{\nu\in\wmm^{(L)}} \int_0^1 |\vartheta_{W_k,\nu}(u)(1-c_k^{(T)}(u))|\,du=0,\\
	\label{eq:pivotlem3}
&&\lim\limits_{T\to\infty}\limsup_{k\to\infty} \sup_{\nu\in\wmm^{(L)}} \int_0^1 |\vartheta_{W_{\infty},\nu}(u)(1-c_k^{(T)}(u))|\,du=0,\\
\label{eq:pivotlem4}
	&&\sup_{\nu\in\wmm^{(L)}} \Bigg|\int_0^1 \big(\vartheta_{W_k,\nu}(u)c_k^{(T)}(u)\big)^r u^s\,du-\int_0^1\big(\vartheta_{W_{\infty},\nu}(u)c_k^{(T)}(u)\big)^r u^s\,du\Bigg|\to 0,
	      	\end{eqnarray}
      	         as $k\to\infty$, for every $T>0$, and all non-negative integers $r,s$. Indeed, \eqref{eq:pivotlem2} and \eqref{eq:pivotlem3} imply uniform convergence in $L_1$ (and hence in weak topology) of the first two terms in the RHS of \eqref{eq:method_of_moment}, whereas the third term in the RHS of \eqref{eq:method_of_moment} can be controlled using \eqref{eq:pivotlem4}, and the fact that for bounded random variables moments convergence imply distribution convergence.

                 %Before proving the three displays above, we note the following inequalities which will be important in the sequel:
                       %for all $W\in\mathcal{W}$.
                      % which gives
                       %\begin{align}\label{eq:simineq2}
                      % \sup_{k\ge 1}\sup_{\nu\in \mathcal{M}^{(L)}}\left(|\vartheta_{W_k,\nu}(u)c_k^{(T)}(u)|\vee |\vartheta_{W_{\infty},\nu}(u)c_k^{(T)}(u)|\right)\le L^{v-1}T,
                      % \end{align}
      	                \vspace{0.1in}
      	                
      	                \emph{Proof of~\eqref{eq:pivotlem2}.} 
                       To begin, for any $W\in\mathcal{W}$ we have the bound
                       \begin{align*}%\label{eq:simineq1}
                       \sup_{\nu\in \wmm^{(L)}}|\vartheta_{W,\nu}(u)|\le L^{v-1}\mathfrak{R}(u;W),
                       \end{align*}which gives
                       $$|\vartheta_{W_k,\nu}(u)(1-c_k^{(T)}(u))|\le L^{v-1} \mathfrak{R}(u; W_k)\left(1\{\mathfrak{R}(u;W_k)>T\}+1\{\mathfrak{R}(u;W_{\infty})>T\}\right).$$ Therefore,~\eqref{eq:pivotlem2} will follow if we can show that
      	                \begin{equation}\label{eq:pivotlem5}
      	                \lim_{T\to\infty}\limsup_{k\to\infty} \int_0^1 \mathfrak{R}(u; W_k)\left(1\{\mathfrak{R}(u;W_k)>T\}+1\{\mathfrak{R}(u;W_{\infty})>T\}\right)\,du=0.
      	                \end{equation}
      	                We now complete the proof based on the following claim, whose proof we defer.
                       \begin{align}\label{eq:simineq3}
                       \limsup_{k\to\infty}\int_0^1 \mathfrak{R}^q(u;W_k)\,du<\infty, \quad \int_0^1 \mathfrak{R}^q(u;W_{\infty})\,du<\infty.
                       \end{align}
                       We will now deal with \eqref{eq:pivotlem5} term by term. First note that:
                       \begin{align*}
                           \int_0^1 \mathfrak{R}(u; W_k)1\{\mathfrak{R}(u;W_k)>T\}\,du\le \frac{1}{T^{q-1}}\int_0^1 \mathfrak{R}^{q}(u;W_k)\,du.
                       \end{align*}
                       By the first claim in \eqref{eq:simineq3}, the right hand side above converges to $0$ by taking $k\to\infty$ followed by $T\to\infty$, thus proving the first claim in \eqref{eq:pivotlem5}. For the second claim in \eqref{eq:pivotlem5}, setting $\tp=q/(q-1)$  H\"older's inequality gives
                   \begin{align*}
                      &\;\;\;\; \int_0^1 \mathfrak{R}(u; W_k)1\{\mathfrak{R}(u;W_{\infty})>T\}\,du\\ &\le \left(\int_0^1 \mathfrak{R}^q(u;W_k)\,du\right)^{\frac{1}{q}}  \left(\int_0^1 1\{\mathfrak{R}(u;W_{\infty})>T\}\,du\right)^{\frac{1}{\tp}}\\ &\le \left(\int_0^1 \mathfrak{R}^q(u;W_k)\,du\right)^{\frac{1}{q}}  \frac{1}{T^{\frac{1}{p}}}\left(\int_0^1 \mathfrak{R}(u;W_{\infty})\,du\right)^{\frac{1}{\tp}},
                   \end{align*}
                   where the final quantity above converges to $0$ taking $k \rightarrow \infty$ followed by $T \rightarrow \infty$ using both claims in \eqref{eq:simineq3}. 
                            This proves the second claim in~\eqref{eq:pivotlem5}, and hence completes the verification of~\eqref{eq:pivotlem2}, subject to proving \eqref{eq:simineq3}. 

                            \noindent\emph{Proof of \eqref{eq:simineq3}.}  Note that
                            \begin{align*}
                                \mathfrak{R}^q(u;W_k)&=\left(\E[\mathrm{Sym}[|W_k|](u,U_2,\ldots,U_v)]\right)^q\le \E[\mathrm{Sym}[|W_k|^q](u,U_2,\ldots ,U_v)],
                            \end{align*}
			      	    where the inequality follows from Lyapunov's inequality (the function $r\mapsto \E [|X|^r]^{1/r}$ is non-decreasing on $(0,\infty)$). On integrating over $u$ we get
              $$\int_0^1 \mathfrak{R}^q(u;W_k)\,du\le \E[\mathrm{Sym}[|W_k|^q](U_1,\ldots ,U_v)]\le \lVert W_k\rVert_{q\Delta}^q,$$
              where the last inequality follows from \cref{lem:Tgraphon0}, part (iii). By our assumption $\limsup_{k\to\infty}\lVert W_k\rVert_{q\Delta}<\infty$, the first conclusion in \eqref{eq:simineq3} follows. The second conclusion follows similarly.
			      	    \vspace{0.1in}

			      	    \emph{Proof of~\eqref{eq:pivotlem3}.} This follows the exact same line of argument as the proof of~\eqref{eq:pivotlem2}, and hence is omitted for brevity.
			      	    
			      	    \vspace{0.1in}
			      
            	    \emph{Proof of~\eqref{eq:pivotlem4}.} Set $h_{\nu}(u):=\E_{\nu}[V|U=u]$, and use the definition $\vartheta_{W_k,\nu}(\cdot)$ in \eqref{eq:varthdef} to note that
              \begin{align*}
                  &\;\;\;\;\int_0^1 \left(\vartheta_{W_k,\nu}(u_1)c_k^{(T)}(u_1)\right)^r u_1^s\,du_1\\ &=\int_0^1 c_k^{(T)}(u_1) u_1^s\left(\int_{[0,1]^{(v-1)r}
                  }\prod_{i=1}^r \left(\mathrm{Sym}[W_k](u_1,u_2^{(i)},\ldots ,u^{(i)}_v)\prod_{a=2}^v h_{\nu}(u_a^{(i)})\,d u_a^{(i)}\right)\right)\,du_1
              \end{align*}
              We can similarly write out an expression for $\int_0^1 \left(\vartheta_{W_{\infty},\nu}(u_1)c_k^{(T)}(u_1)\right)^r u_1^s\,du_1$ with $\mathrm{Sym}[W_k]$ is replaced by $\mathrm{Sym}[W_{\infty}]$. Accordingly, to establish \eqref{eq:pivotlem4}, replacing each $\mathrm{Sym}[W_k](u_1,u_2^{(i)},\ldots ,u_v^{(i)})$ by $\mathrm{Sym}[W_{\infty}](u_1,u_2^{(i)},\ldots ,u_v^{(i)})$ sequentially, it suffices to show that: 
              \begin{equation}\label{eq:pivotlem6}
			      	    \lim_{k\to\infty}\sup_{\nu\in\wmm^{(L)}} \ \big|\mathfrak{F}_{k}^{\nu,A}\big|=0,
			      	    \end{equation}
              for every fixed $L>0$ and $A\subseteq \{2,\ldots ,r\}$, where 
			      	    \begin{small}
			      	    \begin{align*}
			      	    &\mathfrak{F}_{k}^{\nu,A}:=\int_0^1 \bigg(\int_{[0,1]^{v-1}}(\mathrm{Sym}[W_k](u_1,u_2^{(1)},\ldots ,u_v^{(1)})-\mathrm{Sym}[W_{\infty}](u_1,u_2^{(1)},\ldots ,u_v^{(1)}))c_k^{(T)}(u_1)\\ & \prod_{a=2}^v h_{\nu}(u_a^{(1)})\,d u_a^{(1)}\bigg) \left(\int_{[0,1]^{|A|\times (v-1)}} \prod_{i\in A} \left(\mathrm{Sym}[W_k](u_1,u^{(i)}_2,\ldots ,u^{(i)}_v)c_k^{(T)}(u_1)\prod_{a=2}^v h_{\nu}(u^{(i)}_a)\prod_{a=2}^v \,d u^{(i)}_a\right)\right)\\ & \left(\int_{[0,1]^{|A^c|\times (v-1)}} \prod_{i\in A^c} \left(\mathrm{Sym}[W_{\infty}](u_1,u^{(i)}_2,\ldots ,u^{(i)}_v)c_k^{(T)}(u_1)\prod_{a=2}^v h_{\nu}(u^{(i)}_a)\prod_{a=2}^v \,d u^{(i)}_a\right)\right) c_k^{(T)}(u_1) u_1^s\,d u_1.
			      	    \end{align*}
		      	        \end{small}			      	    
			      	    
		      	        In order to establish \eqref{eq:pivotlem6}, let us further define
		      	        $$\mathfrak{n}_{k}^{\nu,(T)}(u):=\int_{[0,1]^{v-1}} \mathrm{Sym}[W_k](u,u_2,\ldots ,u_v)c_k^{(T)}(u)\prod_{a=2}^v h_{\nu}(u_a)\prod_{a=2}^v \,d u_a,$$
		      	        $$\mathfrak{p}_{k}^{\nu,(T)}(u):=\int_{[0,1]^{v-1}} \mathrm{Sym}[W_{\infty}](u,u_2,\ldots ,u_v)c_k^{(T)}(u)\prod_{a=2}^v h_{\nu}(u_a)\prod_{a=2}^v \,d u_a,$$
		      	        and note that
		      	        \begin{align}\label{eq:pivotlem7}
		      	        \sup_{\nu\in\wmm^{(L)}}\sup_{k\geq 1} \max\left\{\lVert\mathfrak{n}_{k}^{\nu,(T)}\rVert_{\infty},\lVert\mathfrak{p}_{k}^{\nu,(T)}\rVert_{\infty}\right\}\leq L^{v-1} T.
		      	        \end{align}
	      	            Proceeding to show \eqref{eq:pivotlem6},  %Next we analyse $\mathfrak{F}_{k}^{\nu,A}$ defined above. 
                    integrating with respect to all the variables other than $u_1, u^{(i_0)}_2,\ldots ,u^{(i_0)}_v$, we get
	      	            \begin{align*}
	      	            \big|\mathfrak{F}_{k}^{\nu,A}\big|&=\bigg|\int_0^1 \Bigg(\int_{[0,1]^{(v-1)}}  \Bigg(\big(\mathrm{Sym}[W_k](u_1,u_2^{(1)},\ldots ,u_v^{(1)})-\mathrm{Sym}[W_{\infty}](u_1,u_2^{(1)},\ldots ,u_v^{(1)})\big)\\ & \quad c_k^{(T)}(u_1) \prod_{a=2}^v h_{\nu}(u_a^{(1)})\prod_{a=2}^v \,d u_a\Bigg) \Bigg) \left(\mathfrak{n}_{k}^{\nu,(T)}(u_1)\right)^{|A|}\left(\mathfrak{p}_{k}^{\nu,(T)}(u_1)\right)^{|A^c|} u_1^s\,d u_1\bigg|\\ &\leq \frac{1}{v!}\sum_{\sigma\in S_v}\Bigg|\int \left(\prod_{(a,b)\in E(H)} W_{k}(u_{\sigma(a)},u_{\sigma(b)})-\prod_{(a,b)\in E(H)} W_{\infty}(u_{\sigma(a)},u_{\sigma(b)})\right)\\ & \left(\prod_{a=2}^v h_{\nu}(u_a)\right) c_k^{(T)}(u_1)\left(\mathfrak{n}_{k}^{\nu,(T)}(u_1)\right)^{|A|}\left(\mathfrak{p}_{k}^{\nu,(T)}(u_1)\right)^{|A^c|} u_1^s\prod_{a=1}^v \,d u_a\Bigg|.
	      	            \end{align*}
      	             Observe that $|h_{\nu}|$'s are bounded by $L$ for $\nu\in \wmm^{(L)}$, $c_k^{(T)}$ is bounded by definition, and further $\mathfrak{n}_{k}^{\nu,(T)}$ and $\mathfrak{p}_{k}^{\nu,(T)}$ are both bounded by \eqref{eq:pivotlem7}. The conclusion in \eqref{eq:pivotlem6} then follows from~\cite[Proposition 3.1, part (ii)]{bhattacharya2024ldp}.
      	                
			      	    \vspace{0.1in}
			      	    
			      	    \emph{(iii)} Note that there exists a sequence of bounded continuous functions $W_{m}\in \mathcal{W}$ such that $\lVert W_{m}-W\rVert_{q}\to 0$ as $m\to\infty$. The triangle inequality implies that given any $m\geq 1$, $k\geq 1$, we have:
			      	    \begin{align}\label{eq:trng}
			      	        d_{\ell}(\Upsilon(W,\nu_k),\Upsilon(W,\nu_{\infty}))&\leq d_{\ell}(\Upsilon(W,\nu_k),\Upsilon(W_m,\nu_{k}))+d_{\ell}(\Upsilon(W_m,\nu_{k}),\Upsilon(W_m,\nu_{\infty}))\nonumber \\ &+d_{\ell}(\Upsilon(W,\nu_{\infty}),\Upsilon(W_m,\nu_{\infty}))
			      	    \end{align}
			      	    By part (ii), we have:
			      	    $$\lim_{m\to\infty}  \sup_{k\in [1,\infty]} d_{\ell}(\Upsilon(W_m,\nu_k),\Upsilon(W,\nu_k))=0.$$
			      	    Further from the definition of weak convergence we have, for every fixed $m$,
			      	    $$d_{\ell}(\Upsilon(W_m,\nu_k),\Upsilon(W_m,\nu_{\infty}))\to 0,\quad \mbox{as}\ k\to\infty.$$
			      	    Combining the two displays above with~\eqref{eq:trng} establishes part (iii).
						\end{proof}
						
\begin{proof}[Proof of \cref{lem:pfcmpct}]

Recall from \eqref{eq:funrel} that $\mathfrak{B}^*_{\theta}=\{ \nu_1\circ\mathfrak{G}^{-1}_1: \nu_1 \in \mathfrak{B}_{\theta}\}$, where $\mathfrak{G}_1(x,y)=(x,\alpha'(y))$ with $\alpha'(.)$ continuous  (see \cref{def:tilt} for definition of $\alpha(.)$). 
The facts that $\mathfrak{B}_{\theta}\subseteq \widetilde{M}_q$ and $\mathfrak{B}^*_{\theta}\subseteq \widetilde{M}_p$ follow directly from \eqref{eq:lip12}. It thus suffices to prove compactness of $\mathfrak{B}_{\theta}$ (which will imply compactness of $\mathfrak{B}^*_{\theta}$).

%We will now establish compactness of $\mathfrak{B}_{\theta}$ and $\mathfrak{B}^*_{\theta}$.  the compactness of $\mathfrak{B}^*_{\theta}$ in the weak topology will follow directly if we can prove the compactness of . 
  % To prove the compactness of $\mathfrak{B}_{\theta}$, recall that 
    To this effect, invoking \eqref{eq:lip12} there exists $C>0$ such that $\Xi(F_{\theta})\in \widetilde{\mathcal{M}}_{p,C}$ (see \eqref{eq:deftilde} for the definition of $\widetilde{\mathcal{M}}_{p,C}$). Also by \cref{cor:Upcont} the function $\Upsilon(W,\cdot)$ is continuous on $\Xi(F_{\theta})$ with respect to weak topology. Since $\Xi(F_{\theta})$ is compact in the weak topology (see~\cref{prop:freen} part (iii)), and  $\mathfrak{B}_{\theta}=\Upsilon(W,\Xi(F_{\theta}))$ (from \eqref{eq:highord2}), compactness of $\mathfrak{B}_{\theta}$ follows.
   
\end{proof}

					\begin{proof}[Proof of~\cref{lem:auxtail}]
     \emph{(i)}
					For any $L>0$ under $\mathfrak{R}_{n,\theta}^{(1)}$ we have
					\begin{equation}\label{eq:auxxxt1}
						    \frac{1}{n}\sum_{i=1}^n \E\left|d_i\left(X_i-X_i^{(L)}\right)\right|\leq \frac{D}{nL^{p-1}}\sum_{i=1}^n \E|X_i|^p,
							\end{equation}
							where $X_i^{(L)}:=X_i1\{|X_i|\le L\}$ and $\|{\bf d}\|_\infty\le D$. The RHS of \eqref{eq:auxxxt1} converges to $0$ as $n\to\infty$ followed by $L\to\infty$ by using~\eqref{eq:lip11}.
					Since $\alpha'(\theta m_i)=\E[X_i|X_j,\ j\in [n],\ j\neq i]$,
						%	Also note that there exists $L>0$ such that $\lVert \mathbf{d}^{(N)}\rVert_{\infty}\leq L$ for all $N\geq 1 $. Recall that $X_i^{(M)}=X_i\mathbbm{1}(|X_i|\leq M)$. 
						%	Therefore, under $\mathbb{R}_{n,\theta}^{(1)}$ as in~\eqref{eq:highordu}, we have:
							 setting $$ \mJ_i^{(L)}:=\E\left[X_i^{(L)}|X_j,\ j\neq i\right],$$
							 we note that
        \begin{align*}
        \bigg|\alpha'(\theta m_i)-\mJ_i^{(L)}\bigg|\le \E\left[|X_i|1(|X_i|>L)|X_j,\ j\neq i\right]\le \frac{1}{L^{p-1}}\E\left[|X_i|^p|X_j,\ j\neq i\right].
        \end{align*}
        Consequently,
			\begin{align}\label{eq:auxxxt2}
				\frac{1}{n}\sum_{i=1}^n \E\bigg|d_i\left(\alpha'(\theta m_i)-\mJ_i^{(L)}\right)\bigg| \leq \frac{D}{nL^{p-1}}\sum_{i=1}^n \E|X_i|^p,
							\end{align}
						which converges to $0$ as $n\to\infty$ followed by $L\to\infty$, by using \eqref{eq:lip11} and the fact that $p>1$.
			Combining~\eqref{eq:auxxxt1}~and~\eqref{eq:auxxxt2}, it suffices to show $\sum_{i=1}^n d_i\left(X_i^{(L)}-\mJ_i^{(L)}\right)=o_P(1)$. Towards this direction, we further define, for $i\neq j$,
							$$\mJ^{(L)}_{i,j}:=\E\left[X_i^{(L)}|X_k,\ k\neq \{i,j\},\  X_j=0\right],$$
							and observe that
							\begin{align*}%\label{eq:auxt2}
								&\;\;\;\;\E\left[\frac{1}{n}\sum_{i=1}^n d_i \Big(X_i^{(L)}-\mJ_i^{(L)}\Big)\right]^2\nonumber \\ &=\frac{1}{n^2}\sum_{i=1}^n d_i^2\E\left(X_i^{(L)}-\mJ_i^{(L)}\right)^2+\frac{1}{n^2}\sum_{i\neq j} d_id_j\E\left[\left((X_i^{(L)}-\mJ_i^{(L)}\right)\left(X_j^{(L)}-\mJ_j^{(L)}\right)\right]\nonumber \\&\leq \frac{4D^2 L^2}{n}+\frac{1}{n^2}\sum_{i\neq j} d_id_j\E\left[\left(X^{(L)}_i-\mJ_{i,j}^{(L)}+\mJ_{i,j}^{(L)}-\mJ_i^{(L)}\right)\left(X_j^{(L)}-\mJ_j^{(L)}\right)\right].
							\end{align*}
							For $i\neq j$ the random variable $X_i^{(L)}-\mJ_{i,j}^{(L)}$ is measurable with respect to the sigma field generated by $\{X_k,\ k\in [n],\ k\neq j\}$, and consequently,
					\begin{equation*}%\label{eq:auxt3}
		\E\left[\left(X_i^{(L)}-\mJ_{i,j}^{(L)}\right)\left(X_j^{(L)}-\mJ_j^{(L)}\right)\right]=0,
							\end{equation*}
							for $i\neq j$. Combining the last two displays gives 
			\begin{equation}\label{eq:auxttt3}
				\E\left[\frac{1}{n}\sum_{i=1}^n d_i \Big(X_i^{(L)}-\mJ_i^{(L)}\Big)\right]^2\leq \frac{4D^2L^2}{n}+\frac{2LD^2}{n^2}\sum_{i\neq j}\E\left|\mJ_{i,j}^{(L)}-\mJ_{i}^{(L)}\right|.
							\end{equation}
							It suffices to show that the second term in the RHS of~\eqref{eq:auxttt3} converges to $0$ for every fixed $D,L$. 
							To control this second term,
						define \begin{equation}\label{eq:ref1}m_{i,j}:=\frac{v}{n^{v-1}}\sum_{\substack{(k_2,\ldots ,k_v)\\ \in \mathcal{S}(n,v,\{i,j\})}}\tQh(i,k_2,\ldots ,k_v)\left(\prod_{m=2}^v X_{k_m}\right)\end{equation}
							for $i\neq j$, where $\mathcal{S}(n,v,\{i,j\})$ denotes the set of all distinct tuples in
$[n]^{v-1}$, such that none of the elements equal to $i$ or $j$. For any $K>0$, by the triangle inequality we have the following for any $i\neq j$,
							\begin{align}\label{eq:trncon}
							  &\;\;\;\;\;\frac{1}{n^2}\sum_{i\neq j}\E\left|\mJ_{i,j}^{(L)}-\mJ_i^{(L)}\right|\nonumber \\
							 &\leq \frac{1}{n^2}\sum_{i\neq j}\E\left[\left(\left|\mJ_{i,j}^{(L)}\right|+\left|\mJ_{i}^{(L)}\right|\right)\left(\mathbbm{1}(|m_{i,j}|\geq K)+\mathbbm{1}(|m_i|\geq K)\right)\right]\nonumber \\&\;\;\;\;\;\;\;\;\;\;+\frac{1}{n^2} \sum_{i=1}^n \E\left[\left|\mJ_{i,j}^{(L)}-\mJ_{i}^{(L)}\right|\mathbbm{1}(|m_{i,j}|\leq K,\ |m_i|\leq K)\right]\nonumber \\ &\leq \frac{2L}{n^2 K}\sum_{i\neq j} \E\left(|m_{i,j}|+|m_i|\right)+\frac{1}{n^2} \sum_{i\neq j}^n \E\left[\left|\mJ_{i,j}^{(L)}-\mJ_{i}^{(L)}\right|\mathbbm{1}(|m_{i,j}|\leq K,\ |m_i|\leq K)\right].
							\end{align}
							It suffices to show that the RHS of~\eqref{eq:trncon} converges to $0$ as $n\to\infty$, followed by $K\to\infty$. 
							Now let us complete this proof based on the following claim, whose proof we defer:
							\begin{equation}\label{eq:claimtrun}
							\sum_{i\neq j}  \E|m_i-m_{i,j}|=O(n).
							\end{equation}
							By combining~\eqref{eq:claimtrun} with~\eqref{eq:lip13}, we also have:
							\begin{equation}\label{eq:claimcon}
							\sum_{i\neq j} \E|m_{i,j}| = O(n^2).
							\end{equation}
						    	
							By combining~\eqref{eq:claimcon} with~\eqref{eq:lip13}, it is immediate that the first term in the RHS of~\eqref{eq:trncon} converges to $0$ as $n\to\infty$, followed by $K\to\infty$. For the second term in the RHS of~\eqref{eq:trncon}, let us define the function:
							$$ \mathfrak{E}_L(t):=\frac{\int_{|x|\leq L} x\exp(tx)\,d\mu(x)}{\int_{-\infty}^{\infty} \exp(tx)\,d\mu(x)}=\E_{X\sim \mu_t}[X\mathbbm{1}(|X|\leq L)],$$
							where $\mu_t$ is the exponential tilt of $\mu$ as introduced in~\cref{def:tilt}. From standard properties of exponential families, $\mathfrak{E}_L(\cdot)$ has a continuous derivative on $\R$ and therefore, 
							$$\sup_{|t|\leq |\theta| K} |\mathfrak{E}_L'(t)|\leq \mathfrak{c},$$
							where $\mathfrak{c}>0$ depends on $|\theta|$, $L$, and $K$.  Hence, %With this observation in mind, observe that
							\begin{align}\label{eq:ref2}
							    &\;\;\;\;\;\frac{1}{n^2}\sum_{i\neq j}\E\left[\left|\mJ_{i,j}^{(L)}-\mJ_{i}^{(L)}\right|\mathbbm{1}(|m_{i,j}|\leq K,\ |m_i|\leq K)\right]\nonumber\\ &=\frac{1}{n^2}\sum_{i\neq j}\E\left[\left|\mathfrak{E}(\theta m_{i,j})-\mathfrak{E}(\theta m_i)\right|\mathbbm{1}(|m_{i,j}|\leq K,\ |m_i|\leq K)\right]\nonumber \\ &\leq \frac{\mathfrak{c}|\theta|}{n^2}\sum_{i\neq j}\E|m_{i,j}-m_i|=O\left(\frac{1}{n}\right),
							\end{align}
							for every fixed $\theta$, $L$, and $K$. This completes the proof that~\eqref{eq:trncon} converges to $0$ as $n\to\infty$ followed by $K\to\infty$.
							\vspace{0.1in}
							
							\emph{Proof of~\eqref{eq:claimtrun}.} The symmetry of $\mathrm{Sym}[Q_n]$ implies
		\begin{align*}%\label{eq:auxt4}
								| m_i - m_{i,j}| %& \frac{Mv\theta}{n^{v-1}}\sum_{(k_1,\ldots ,k_v)\in S(n,v,i|j) \tQh(i,k_2,\ldots ,k_v)\left(\prod_{m=2}^v |X_{k_m}|\right)\\
								\leq \frac{v}{n^{v-1}} |X_j|\sum_{\substack{(k_3,\ldots ,k_v)\\ \in \mathcal{S}(n,v-1,\{i,j\})}} \mathrm{Sym}[|Q_n|](i,j,k_3,\ldots ,k_v)\left(\prod_{m=3}^v |X_{k_m}|\right).
							\end{align*}
							 %and $$S(n,v,i|j):=\{(k_2,\ldots ,k_v)\in \mathcal{S}(n,v,i):\ \exists \ m\in \{2,\ldots ,v\},\ k_m=j\}.$$
						%	Further, by symmetry of the RHS in~\eqref{eq:auxt4} in terms of $k_2,\ldots ,k_m$, we further have: and observe that
							%$$\E[X_i|X_k,\ k\neq \{i,j\},\ X_j=0]=\alpha'(\theta m_{i,j}).$$
							%Once again, by Jensen's inequality, we have:
							%\begin{equation}\label{eq:auxxxt5}
							%\frac{M}{n^2}\sum_{i\neq j}\E\left|\mJ_{i,j}^{(M)}-\alpha'(\theta m_{i,j})\right|\leq \frac{1}{nM^{p-2}}\sum_{i=1}^n  \E|X_i|^{p}.
      %\begin{equation}\label{eq:p2i3}
					%\frac{1}{n}\sum_{i=1}^n |\alpha'(\theta m_i)-\alpha'(\theta\bar{m})|=o_P(1),
				%\end{equation}
				%\begin{equation}\label{eq:p2i4}
					%|\alpha'(\theta \bar{m})|=O_p(1).
				%\end{equation}
							%\end{equation}
							%As $p> 2$,~\eqref{eq:lip11} implies that the right hand side above converges to $0$ if we take $n\to\infty$ first, followed by taking $M\to\infty$. Therefore, by combining~\eqref{eq:auxxxt2},~\eqref{eq:auxttt3}~and~\eqref{eq:auxxxt5}, the proof of the current result will be completed if we can show that
							%\begin{equation}\label{eq:auxxxt4}
							%\frac{1}{n^2}\sum_{i\neq j}\E|\alpha'(\theta m_i)-\alpha'(\theta m_{i,j})|\overset{n\to\infty}{\longrightarrow}0.
							%\end{equation}
				%\begin{align}\label{eq:auxt10}
					%		|\alpha'(\theta m_i)-\alpha'(\theta m_{i,j})|\leq \frac{Mv(v-1)\theta}{n^{v-1}} |X_j|\sum_{\substack{(k_3,\ldots ,k_v)\\ \in \mathcal{S}(n,v,\{i,j\})}} \tQh(i,j,k_3,\ldots ,k_v)\left(\prod_{m=3}^v |X_{k_m}|\right),
					%		\end{align}
				 	Using this, we bound the left hand side of \eqref{eq:trncon} below. 
				\begin{align*}
								\frac{1}{n}\sum_{i\neq j} \left| m_i- m_{i,j}\right|&\leq \frac{v}{n^{v}}\sum_{i\neq j}\sum_{\substack{(k_3,\ldots ,k_v)\\ \in \mathcal{S}(n,v,\{i,j\})}} \mathrm{Sym}[|Q_n|](i,j,k_3,\ldots ,k_v)|X_j|\left(\prod_{m=3}^v |X_{k_m}|\right)\nonumber \\&\leq v\left(\frac{1}{n^v}\sum_{(k_1,\ldots ,k_v)\in [n]^v} |\mathrm{Sym}[|Q_n|](k_1,\ldots ,k_v)|^q\right)^{\frac{1}{q}}\left(\frac{1}{n}\sum_{i=1}^n  |X_i|^p\right)^{\frac{v-1}{p}}\\ &\le v\lVert W_{Q_n}\rVert_{q\Delta}\left(\frac{1}{n}\sum_{i=1}^n  |X_i|^p\right)^{\frac{v-1}{p}},
							\end{align*}
							%Here $C', C'', C'''$ denote constants depending only on $L,M,v,\theta$. 
						where the second inequality follows from H\"older's inequality, and the third inequality uses \cref{lem:Tgraphon0} part (c). The above display, on taking expectation, gives 
       \begin{align*}%\label{eq:auxt5}
       \frac{1}{n}\sum_{i\neq j}\E|m_i-m_{i,j}|\le v\lVert W_{Q_n}\rVert_{q\Delta}\left(\frac{1}{n}\sum_{i=1}^n  \E|X_i|^p\right)^{\frac{v-1}{p}}.
       \end{align*}
       Here, we have used Lyapunov's inequality coupled with the observation that $v-1\le p$. 
       As $\limsup_{n\to\infty} \lVert W_{Q_n}\rVert_{q\Delta}<\infty$ by~\eqref{eq:q}, an application of~\eqref{eq:lip11} in the last display above completes the proof of~\eqref{eq:claimtrun}.
						\end{proof}

\begin{proof}[Proof of~\cref{lem:auxtail}]
%% Once again, we will be using $\lesssim$ to hide constants that depend only on $H,p,q$. 
\emph{(ii)} Choose $q'<q$ and $\tp<p$ such that $\tp^{-1}+q'^{-1}=1$. Fixing $L>0$ we have
$$\frac{1}{n}\bigg|\sum_{i=1}^n m_i(X_i-X_i^{(L)})\bigg|\le \frac{1}{L^{\frac{p}{\tp}-1}}\left(\frac{1}{n}\sum_{i=1}^n |m_i|^{q'}\right)^{\frac{1}{q'}}\left(\frac{1}{n}\sum_{i=1}^n |X_i|^p\right)^{\frac{1}{\tp}}=o_p(1),$$
where the limit is to be understood as $n\to\infty$ followed by $L\to\infty$. Here we used~\eqref{eq:lip11} and~\eqref{eq:lip13}. Now, from standard analysis we have the existence of a $C_1$ function $\psi_L:\R\to\R$ such that $\psi_L(x)=x$ for $|x|\le L$, $|\psi_L(x)|\le |x|$, $\lVert \psi_L\rVert_{\infty}<\infty$, and $\lVert \psi_L'\rVert_{\infty}<\infty$. This gives
$$|m_i-\psi_L(m_i)|^{q'}\le 2^{q'}|m_i|^{q'} { 1}\{|m_i|>L\}\le  \frac{2^{q'}}{L^{q-q'}}|m_i|^q.$$
Using this bound along with H\"{o}lder's inequality, we get:
\begin{align*}
\frac{1}{n}\bigg|\sum_{i=1}^n (m_i-\psi_L(m_i))X_i^{(L)}\bigg| &\le \left(\frac{1}{n}\sum_{i=1}^n |m_i-\psi_L(m_i)|^{q'}\right)^{\frac{1}{q'}}\left(\frac{1}{n}\sum_{i=1}^n |X_i|^{\tp}\right)^{\frac{1}{\tp}}\\ &\le \frac{2}{L^{\frac{q}{q'}-1}}\left(\frac{1}{n}\sum_{i=1}^n |m_i|^q\right)^{\frac{1}{q'}}\left(\frac{1}{n}\sum_{j=1}^n |X_j|^{\tp}\right)^{\frac{1}{\tp}}=o_P(1)
\end{align*}
as $n\to\infty$ followed by $L\to\infty$, on using~\eqref{eq:lip11} and \eqref{eq:lip13}. Combining the above displays we get
\begin{align*}%\label{eq:reff1}
\frac{1}{n}\bigg|\sum_{i=1}^n m_iX_i-\sum_{i=1}^n \psi_L(m_i)X_i^{(L)}\bigg|=o_p(1),
\end{align*}
as $n\to\infty$ followed by $L\to\infty$. A similar computation shows
\begin{align*}%\label{eq:reff2}
\frac{1}{n}\bigg|\sum_{i=1}^n m_i\E[X_i|X_j,\ j\neq i]-\sum_{i=1}^n \psi_L(m_i)\E[X_i^{(L)}|X_j,\ j\neq i]\bigg|=o_p(1)
\end{align*}
in the same sense. Using the last two displays above, %~\eqref{eq:reff1}~and~\eqref{eq:reff2}, 
it suffices to show that 
\begin{align}\label{eq:reff3}
\frac{1}{n}\bigg|\sum_{i=1}^n\psi_L(m_i)\left(X_i^{(L)}-\E[X_i^{(L)}|X_j,\ j\neq i]\right)\bigg|=o_p(1),
\end{align}
as $n\to\infty$ for fixed $L$. Towards this direction, we will use the definitions of $m_{i,j}$, $\mathcal{J}_i^{(L)}$, $\mathcal{J}_{i,j}^{(L)}$  from the proof of~\cref{lem:auxtail} part (a). 
Observe that
\begin{align*}%\label{eq:reff4}
&\;\;\;\;\;\frac{1}{n^2}\E\bigg|\sum_{i=1}^n \psi_L(m_i)\left(X_i^{(L)}-\mathcal{J}_i^{(L)}\right)\bigg|^2\nonumber \\ &\leq \frac{L^2\lVert \psi_L\rVert_{\infty}^2}{n}+\frac{1}{n^2}\sum_{i\neq k}\E\left[\psi_L(m_i)\psi_L(m_k)(X_i^{(L)}-\mathcal{J}_i^{(L)})(X_k^{(L)}-\mathcal{J}_k^{(L)})\right].
\end{align*}
By Markov's inequality, in order to establish~\eqref{eq:reff3}, it suffices to show that the above display converges to $0$ as $n\to\infty$ for any fixed $L$. As
$$\E\left[\psi_L(m_{i,k})\psi_L(m_k)(X_i^{(L)}-\mathcal{J}_{i,k}^{(L)})(X_k^{(L)}-\mathcal{J}_k^{(L)})\right]=0$$
for $i\neq k$, it suffices to show that
$$ \frac{1}{n^2}\sum_{i\neq k}\E|\mathcal{J}_{i}^{(L)}-\mathcal{J}_{i,k}^{(L)}|=o(1), \qquad \frac{1}{n^2}\sum_{i\neq k}\E|\psi_L(m_i)-\psi_L(m_{i,k})|=o(1).$$
The left hand display is what we bounded in~\eqref{eq:trncon}. As $|\psi_L(m_i)-\psi_L(m_{i,k})|\le \lVert \psi_L'\rVert_{\infty}|m_i-m_{i,k}|$, the right hand display above follows directly from~\eqref{eq:claimtrun}. 
\end{proof}

	\begin{appendix}
	\section*{}\label{sec:appen}

In this Section, we will prove the auxiliary lemmas from earlier in the paper. \cref{sec:appenaux0} collects all results on the properties of the base measure $\mu$, and \cref{sec:appenaux} contains some general probabilistic convergence results.

\subsection{Proofs of Lemmas \ref{lem:fixsol}, \ref{lem:KLcont},  and 
 \ref{lem:fixsol0}}\label{sec:appenaux0}
       \begin{proof}[Proof of~\cref{lem:fixsol}] 
						
		With $\mv_{\theta,B,\mu}(x)=\theta x^2+Bx-x\beta(x)+\alpha(\beta(x))$ as in the statement of the lemma,	differentiation  gives
							$\mv'_{\theta,B,\mu}(x)=2\theta x+B-\beta(x)$. Using~\cref{lem:fixsol0} part (ii) we get $\lim_{x\to\pm \infty}\mv'_{\theta,B,\mu}(x)=\pm \infty$ (since $p\ge 2)$, and so the continuous function $\mv_{\theta,B,\mu}(.)$ attains its global maximizers on $\R$, and any  maximizer (local or global)
							satisfies $\mv'_{\theta,B,\mu}(x)=2\theta x+B-\beta(x)=0$, which is equivalent to solving $\tph_{\theta,B,\mu}(x)=0$, where
							\begin{equation}\label{eq:fixsol1}
		\tph_{\theta,B,\mu}(x):=x-\alpha'(2\theta x+B),\quad  \tph'_{\theta,B,\mu}(x):=1-2\theta \alpha''(2\theta x+B).
					\end{equation}
						 	%\begin{align}\label{eq:fixsol1}
	%	\tph'_{\theta,0,\mu}(x)=1-2\theta \alpha''(2\theta x),\qquad \tph''_{\theta,0,\mu}(x)=-4\theta^2 \alpha'''(2\theta x).
			%				\end{align}
						%	As the case $\theta=0$ is trivial, we will consider $\theta>0$ throughout the rest of the proof.
							
							\emph{(i)} Here $B=0$, and symmetry of $\mu$ gives $\alpha'(0)=\tph_{\theta,0,\mu}(0)=0$. To show that $0$ is the only root of $\tph_{\theta,0,\mu}(\cdot)$ (and hence the unique maximizer of $\mv_{\theta,0,\mu}$), using symmetry of $\mu$ it suffices to show that $\tph_{\theta,0,\mu}$ does not have any other roots on $(0,\infty).$
							To this effect, using~\eqref{eq:secasn} it follows that $\alpha''(.)$ is non-increasing on $(0,\infty)$, and so $\tph_{\theta,0,\mu}$ is convex using \eqref{eq:fixsol1}. Since $\tph'_{\theta,0,\mu}(0)=0$, it follows that $0$ is also a global minimizer of $\tph_{\theta,0,\mu}(\cdot)$, and so $\tph_{\theta,0,\mu}$ is non-positive. If there exists a positive root $x_0$ of $\tph_{\theta,0,\mu}(\cdot)$, then by convexity (and symmetry) we have $\tph_{\theta,0,\mu}\equiv 0$ on $[-x_0,x_0]$. But this implies $\alpha'(\cdot)$ is linear on this domain, and so $\alpha(\cdot)$ must be a quadratic, which is only possible only if $\mu$ is a Gaussian. This contradicts \eqref{eq:tailp}, and hence completes the proof of part (i).
					
     \vspace{0.1in}
					
								\emph{(ii)} By symmetry, it suffices to consider the case $B>0$. Comparing $x$ with $-x$ and using the symmetry of $\mu$, it follows that all global maximizers lie in $[0,\infty)$. Also in this case $\alpha'(B)>0$, which implies $\tph_{\theta,B,\mu}(0)<0$. As $\lim\limits_{x\to\infty} \tph_{\theta,B,\mu}(x)=\infty$ by~\cref{lem:fixsol0} part (i),
				%~\eqref{eq:intmax}, either
	$\tph_{\theta,B,\mu}(\cdot)$ either has a unique positive root, or at least $3$ positive roots. If the latter holds, using~\eqref{eq:fixsol1} $\alpha''(\cdot)$ must have two positive roots $(x_1,x_2)$, which on using~\eqref{eq:secasn} gives that $\alpha'''(\cdot)\equiv 0$ on the interval $[x_1,x_2]$. As in part (i), this implies that $\mu$ is Gaussian, a contradiction to \eqref{eq:tailp}.
		Thus $\mv(\cdot)$ has a unique positive maximizer $t_{\theta,B,\mu}$. 
  
  \vspace{0.1in}
							
	\emph{(iii)} In this case $\tph_{\theta,B,\mu}(0)=0$ and $\tph_{\theta,B,\mu}'(0)<0$. Therefore, $\tph_{\theta,B,\mu}(\cdot)$ either has a unique positive root or at least $3$ positive roots. From there we argue, similar to part (ii) above, that $\tph_{\theta,0,\mu}(\cdot)$ has exactly one positive root $t_{\theta,0,\mu}$.  By symmetry, it follows that $-t_{\theta,0,\mu}$ is the unique negative root of $\tph_{\theta,0,\mu}(\cdot)$, and $\pm t_{\theta,0,\mu}(\cdot)$ are the global maximizers of $\mv_{\theta,0,\mu}$. 
       \end{proof}

        \begin{proof}[Proof of~\cref{lem:KLcont}]
					The function $\beta(.)$ is smooth ($C^\infty$) on $\mathcal{N}$, and the function $\gamma(.)$ is smooth on $\R$. Consequently, the function $\gamma(\beta(.))$ is smooth on $\mathcal{N}$. To verify continuity on ${\rm cl}(\mathcal{N})$, 
						it suffices to cover the (possible) boundary cases:
					\begin{itemize}	
					\item
				If $a:=\sup\{\mathcal{N}\}<\infty$, then $\lim_{u\to a}\gamma(\beta(u))=\gamma(\beta(a))=\gamma(\infty)$.
				
				\item
				If $b:=\inf\{\mathcal{N}\}>-\infty$, then
			$\lim_{u\to b} \gamma(\beta(u))=\gamma(\beta(b))=\gamma(-\infty)$.
					\end{itemize}
						
							We will only prove the first case, as the other case follows similarly. Note that,
							$$\lim_{u\to a}\beta(u)=\infty\Rightarrow \lim_{u\to a}\mu_{\beta(u)}=\delta_a,$$
							where the second limit is in weak topology. Further, 
						 \begin{equation}\label{eq:lowersem}
			\liminf_{u\to a}\gamma(\beta(u))=\liminf_{u\to\sup\{\mathcal{N}\}}D(\mu_{\beta(u)}|\mu)\ge D(\delta_a|\mu)= \gamma(\infty)
							\end{equation} 
        by the lower semi-continuity of Kullback-Leibler divergence. 
						%	where $a:=\sup\{\mathcal{N}\}<\infty$. %As linear shift does not affect the Kullback-Leibler divergence, assume without loss of generality that $a>0$. 
							If $\mu(\{a\})=0$,  then $\gamma(\infty)=\infty$, and~\eqref{eq:lowersem} yields the desired conclusion.
							If $\mu(\{a\})>0$, then $\gamma(\infty)=-\log{\mu(\{a\})}$. Also, for any $\theta\in\R$, we have
						$$\alpha(\theta)=\log{\int \exp(\theta x)\,d\mu(x)}\geq \theta a+\log{\mu(\{a\})}.$$
							For all $u$ such that $\beta(u)>0$ (which holds for all $u$ close to $a$), this gives $$\gamma(\beta(u))=u\beta(u)-\alpha(\beta(u))\leq u\beta(u)-a\beta(u)-\log{\mu(\{a\})}\leq \log{\mu(\{a\})}=\gamma(\infty).$$
							Combining the above display with~\eqref{eq:lowersem} gives $\lim_{u\to a}\gamma(\beta(u))=\gamma(\infty),$ as desired.
							\end{proof}

\noindent \begin{proof}[Proof of~\cref{lem:fixsol0}] 
						\emph{(i)}
						We prove
	$\lim_{\theta\to\infty}\frac{\alpha'(\theta)}{\theta^{\frac{1}{p-1}}}=0$, noting that the proof of the other limit is similar. To this effect, we consider the following two cases separately:
						\begin{itemize}
						    \item $\mu(0,\infty)>0$.

						Fixing $\theta>0$ and $\delta>0$, we have:
	\begin{align*}\frac{|\alpha'(\theta)|}{\theta^{\frac{1}{p-1}}}&\le \frac{\int_\R |y|\exp(\theta  y)\,d\mu(y)}{\theta^{\frac{1}{p-1}}\int_\R \exp(\theta y)\,d\mu(y)}\\ &\leq \frac{\delta \theta^{\frac{1}{p-1}} \int_{|y|\leq \delta \theta^{\frac{1}{p-1}}} \exp(\theta y)\,d\mu(y)+\int_{|y|\geq \delta \theta^{\frac{1}{p-1}}}|y|\exp(\theta y)\,d\mu(y)}{\theta^{\frac{1}{p-1}}\int \exp(\theta y)\,d\mu(y)}\\
								&\leq \delta + \frac{\int_\R |y| \exp(|y|^p \delta^{1-p}) d\mu(y)}{\theta^{\frac{1}{p-1}} \int_\R \exp(\theta y)d\mu(y)},%\quad \mbox{as}\ \theta\to\infty,
				\end{align*}
    where we use the bound $|\theta y|\le |y|^p \delta^{1-p}$ on the set $|y|\ge \delta |\theta|^{\frac{1}{p-1}}$. Letting $\theta\to \infty$ we have $\int_\R\exp(\theta y)d\mu(y)\to \infty$, as $\mu(0,\infty)>0$. Since the numerator in the second term  in the display above is finite invoking \eqref{eq:tailp}, the second term above converges to $0$ as $\theta\to\infty$, allowing us to conclude
		$\limsup_{\theta\to\infty}\frac{|\alpha'(\theta)|}{\theta^{\frac{1}{p-1}}}\le \delta$. Since $\delta>0$ is arbitrary, the desired limit follows.
						\\
      
							\item $\mu(0,\infty)=0$.
							\\
								In this case, $\alpha'(\theta)\le 0$. Since $\alpha'(\cdot)$ is non-decreasing, $\lim_{\theta\to\infty}\alpha'(\theta)$ exists as a finite (non-positive) number. Consequently we have
					$\lim_{\theta\to \infty}\frac{\alpha'(\theta)}{\theta^{\frac{1}{p-1}}}= 0$. 
							\end{itemize}
							\emph{(ii)}
							We only study the case when $x\to \sup\{\mathcal{N}\}$. If $\sup\{\mathcal{N}\}<\infty$, then the conclusion is immediate as the denominator converges to a finite number while the numerator diverges. Therefore, we only focus on the case $\sup\{\mathcal{N}\}=\infty$. To this effect, fixing $M>0$ using part (i) gives that for all $x$ large enough (depending on $M$) we have
			$$\alpha'(M x)\le x^{\frac{1}{p-1}}\Leftrightarrow Mx	\le \beta\left(x^{\frac{1}{p-1}}\right).$$
							Taking limit gives
		$$\liminf_{x\to\infty}\frac{\beta\left(x^{\frac{1}{p-1}}\right)}{x}\ge M.$$
							Since $M$ is arbitrary, we conclude the desired conclusion follows.
							\end{proof}
     \subsection{Proofs of \cref{lem:Tgraphon0} and Lemmas \ref{lem:proj_cont_map} and  
     \ref{lem:specond}}\label{sec:appenaux}

     \begin{proof}[Proof of \cref{lem:Tgraphon0}]
    \emph{(i) and (ii)} These are direct consequences of \cite[Proposition 2.19]{borgsdense1} and \cite[Lemma 2.2]{bhattacharya2024ldp}.

    \emph{(iii)} 
    With $\mathrm{Sym}[|W|]$ as in~\cref{def:symmfndef}, we have
   %  \begin{align*}
   % &\;\;\;\;\frac{1}{n^v}\sum_{(k_1,\ldots ,k_v)\in [n]^v} |\tQh(k_1,\ldots ,k_v)|^q\nonumber \\ &=\frac{1}{n^v}\sum_{(k_1,\ldots ,k_v)\in [n]^v} \bigg|\frac{1}{v!}\sum_{\sigma\in \mathcal{S}_v} \prod_{(a,b)\in E(H)} Q_n(k_{\sigma(a)},k_{\sigma(b)})\bigg|^q\nonumber \\ &\le \frac{1}{v!}\sum_{\sigma\in\mathcal{S}_v}\frac{1}{n^v}\sum_{(k_1,\ldots ,k_v)\in [n]^v} \prod_{(a,b)\in E(H)}|Q_n(k_{\sigma(a)},k_{\sigma(b)})|^q\le \lVert W_{Q_n}\rVert_{q\Delta}^{q|E(H)|}
   % \end{align*}
   % To begin, use H\"{o}lder's inequality with exponents $q$ and $p$ to get:
  % \begin{align*}
%&\;\;\;\;\;\;\E\left[\mathrm{Sym}[|W|](U_1,\ldots ,U_v)|\phi(V_1,\ldots ,V_v)|\right]\\ &\le\left(\E\left[\mathrm{Sym}[|W|](U_1,\ldots ,U_v)\right]^q\right)^{1/q}\left(\E|\phi(V_1,\ldots ,V_v)|^p\right)^{1/p}.
  % \end{align*}
  % The first term in the RHS above can be bounded as follows:
  % To begin, note that
    \begin{align*}
   \E\left[\mathrm{Sym}[|W|](U_1,\ldots ,U_v)\right]^q &=\E \bigg|\frac{1}{v!}\sum_{\sigma\in \mathcal{S}_v} \prod_{(a,b)\in E(H)} |W|(U_{\sigma(a)},U_{\sigma(b)})\bigg|^q\nonumber \\ &\le \frac{1}{v!}\sum_{\sigma\in\mathcal{S}_v}\E \prod_{(a,b)\in E(H)}|W|^q(U_{\sigma(a)},U_{\sigma(b)})\le \lVert W\rVert_{q\Delta}^{q|E(H)|},
   \end{align*}
  where the first inequality uses Lyapunov's inequality, and the second inequality follows from \cref{lem:Tgraphon0} part (ii), with $W$ replaced by $|W|^q$. 
   % Combining the two displays above completes the proof.
 \end{proof}
 
\begin{proof}[Proof of~\cref{lem:proj_cont_map}]
          By using \eqref{eq:tight}, it follows that the sequence $\{\xi_n\}_{n\ge 1}$ is tight. Passing to a subsequence, w.l.o.g.~we can assume $\xi_n\stackrel{d}{\to}\xi_\infty$, where $\P(\xi_\infty\in \mf)=1$ (as $\mf$ is closed). By the  Portmanteau Theorem,
          \begin{align}\label{eq:proj1}
          \P(\xi_{\infty}\in K^c)\le \limsup\limits_{n\to\infty}\P(\xi_n\in K^c)=0.
          \end{align}
          Next we will show that $g(\xi_n)\stackrel{d}{\to} g(\xi_{\infty})$. Towards this direction let $H\subseteq g(\mf)$ be a closed set. We will write $g^{-1}(H)$ to denote the inverse image of the set $H$ under $g$. Another application of the Portmanteau Theorem implies:
          \begin{align}\label{eq:proj2}
              \notag\limsup\limits_{n\to\infty}\P(g(\xi_n)\in H,\ \xi_n\in K)=&\limsup\limits_{n\to\infty} \P(\xi_n\in g^{-1}(H)\cap K)\\
              \ge& \P(\xi_{\infty}\in g^{-1}(H)\cap K).
          \end{align}
          The last line uses the fact that $g^{-1}(H)\cap K$ is closed which in turn follows from the continuity of $g$ on $K$. Finally, by \eqref{eq:tight} and \eqref{eq:proj1}, we have:
          \begin{align}\label{eq:proj3}
        \limsup\limits_{n\to\infty}|\P(g(\xi_n)\in H,\ \xi_n\in K)-\P(g(\xi_n)\in H)|=0,  
          \end{align}
          \begin{align}\label{eq:proj4}
          \P(\xi_{\infty}\in g^{-1}(H)\cap K)=\P(\xi_{\infty}\in g^{-1}(H)).
          \end{align}
          By combining \eqref{eq:proj2}, \eqref{eq:proj3}, and \eqref{eq:proj4}, it follows that 
          $$\limsup\limits_{n\to\infty} P(g(\xi_n)\in H)\ge \P(g(\xi_{\infty})\in H).$$
          By the Portmanteau theorem, this yields $g(\xi_n)\stackrel{d}{\to} g(\xi_{\infty})$. 
          So for any $\vep>0$ we get
          \begin{align*}
          \limsup_{n\to\infty}\P(d_Y(g(\xi_n),g(\mf))\ge \varepsilon)\le \P(d_Y(g(\xi_\infty),g(\mf))\ge \varepsilon)=0
          \end{align*}
          as $g(\xi_\infty)\in g(\mf)$ a.s.~Since $\varepsilon>0$ is arbitrary, $d_Y(g(\xi_\infty),g(\mf))\stackrel{P}{\to}0$, as desired.
          \end{proof}

          \begin{proof}[Proof of~\cref{lem:specond}]
 
\emph{(i)} Since $\limsup_{n\to\infty}\E |f_n(U)|^p<\infty$,  it follows that the sequence $\{|f_n(U)|^{p'}\}_{n\geq 1}$ is uniformly integrable, and so $\E|f_{\infty}(U)|^{p'}<\infty$. By standard approximation results, given any $\varepsilon>0$, there exists $h:[0,1]\to \R$ (depending on $\varepsilon$) such that $h$ is continuous on $[0,1]$ and $\E|h(U)-f_{\infty}(U)|^{p'}<\varepsilon$. Continuous mapping theorem gives $f_n(U)-h(U)\overset{D}{\longrightarrow} f_{\infty}(U)-h(U)$. Since $
|f_n(U)-f_\infty(U)|^{p'}$ is uniformly integrable, 
      $\lVert f_n-f_{\infty}\rVert_{p'}\longrightarrow \lVert f_{\infty}-h\rVert_{p'}\le \varepsilon$. 
      As $\varepsilon>0$ is arbitrary, this completes the proof of part (a).
      
      \emph{(ii)}
      The conclusion follows by applying part (a) on the sequence of measures alternating between $(U,f(U))$ and $(U,g(U))$ along odd and even subsequences.
      \end{proof}

\end{appendix}

\section*{Acknowledgments}
The authors would like to thank the anonymous referees, an Associate Editor, and the Editor for their constructive comments.

\subsection*{Funding}
The third author's research is partially supported by NSF grant DMS-2113414.

\small
\bibliographystyle{plainnat}
\bibliography{template, Newtemplate}

\end{document}